\documentclass{amsart}

\usepackage{amsmath,latexsym,amscd,amsbsy,amssymb,amsfonts,amsthm,fleqn,leqno,
euscript, graphicx, color}

\newtheorem{thm}{Theorem}[section]
\newtheorem{pro}[thm]{Proposition}
\newtheorem{lem}[thm]{Lemma}
\newtheorem{cor}[thm]{Corollary}
\newtheorem{ex}[thm]{Example}
\newtheorem{que}[thm]{Question}

\newtheorem{df}[thm]{Definition}
 
\newtheorem{rem}[thm]{Remark}

\newcommand{\Bd}{\mbox{{\rm bd}}}
\newcommand{\Int}{\mbox{{\rm int}}\,}
\newcommand{\cl}{\mbox{{\rm cl}}}
\newcommand{\st}{\mbox{{\rm st}}}
\newcommand{\pr}{\mbox{{\rm pr}}}
\newcommand{\id}{\mbox{{\rm id}}}

\newcommand{\St}{\mbox{{\rm St}}}
\newcommand{\aut}{\mbox{{\rm Aut}}}
\newcommand{\CL}{\mbox{{\rm CL}}}
\newcommand{\Hom}{\mbox{{\rm Hom}}}

\begin{document}

\title{On proper compactifications of topological groups}\thanks{The paper is partially supported by ISF grant 3241/24}

\author{K.\,L.~Kozlov}\thanks{The study of the first author was carried out with the help of the Center of Integration in Science, Ministry of Aliyah and Integration, Israel,  8123461.}
\address{Department of Mathematics, Ben-Gurion University of the Negev, Beer Sheva, Israel}
\email{kozlovk@bgu.ac.il}
\author{A.\,G.~Leiderman}
\address{Department of Mathematics, Ben-Gurion University of the Negev, Beer Sheva, Israel}
\email{arkady@math.bgu.ac.il}

\date{}

\maketitle
 
\begin{abstract}
In the present paper, we examine in detail the method of ''graph compactifications''\ of topological groups. The graph and Ellis methods of constructing proper compactifications of topological groups are applied for the investigation of possible extensions of algebraic operations on a topological group to its compactifications, and give descriptions of Roelcke, Ellis, {\rm WAP}, and graph compactifications of topological groups. Additionally, using dichotomy theorems of A.\,V.~Arhangelskii, we show that the description of compactifications can be effectively used in the investigation of topological properties of their remainders. As examples, subgroups of the permutation group (in the permutation topology) and the automorphism group of a LOTS (in the topology of pointwise convergence) are examined. 

\medskip
 

\medskip

\end{abstract}


\tableofcontents

\section{Introduction}

\medskip

The study of proper compactifications of topological groups was initiated by A.~Weil, who described in terms of a right (or left) uniformity on a topological group when the topological group has a compactification, which is a compact topological group. In fact, this is the unique $G$-compactification of a topological group, examined with its action by multiplication on itself, see, for example, \cite{ChK2}. Therefore, the question of what algebraic operations on a group can be extended to continuous operations on its compactifications is natural. 

A topological group has a proper semitopological semigroup compactification iff it can be represented as a group of isometries of a reflexive Banach space~\cite{Stern} (see also~\cite{Megr2001-1}). The maximal proper semitopological semigroup compactification of $G$ is a {\rm WAP}-compactification of $G$. There are topological groups that have no proper semitopological semigroup compactifications~\cite{Megr2001}. In~\cite{BITsankov}, the following question, in fact, was initiated. Does the involution on a group $G$ extend to the involution on a {\rm WAP}-compactification of $G$, in case the latter is proper?

Every topological group $G$ has an Ellis compactification, a $G$-compactification of $G$ that is a right topological semigroup, for instance, the {\it greatest ambit} (the Samuel compactification of $G$ with respect to the right uniformity). In~\cite[Corollary 4.11]{GlasnerMegr2008}, the characterization is given of when the Roelcke compactification of a topological group is an Ellis compactification. 

\medskip

To obtain compactifications of a topological group $G$, we use a representation of $G$ in a compact space $X$ and apply two ``functional''\ approaches. 

Since every element $g\in G$ can be examined as a map, an element of $X^X$, an {\it Ellis enveloping group}~\cite{Ellis} can be used. In the original, R.~Ellis  obtained proper compactifications of semitopological groups. To preserve the topology of $G$, one must examine those representations of $G$ in compacta $X$ for which the topology of pointwise convergence induced by the embedding of $G$ in $X^X$ coincides with the original topology on $G$. This approach is used, for example, in~\cite{GlasnerMegr2008}, and examined in detail in~\cite{KozlovSorin}, and~\cite{KozlovSorin2025}. See also~\cite{SorinG1} and~\cite{SorinG2} where the automorphism group of a circularly ordered set is examined. The characterization of when a $G$-compactification of a topological group is its  Ellis compactification can be found in~\cite{KozlovLeiderman2025}. 

Since every element $g\in G$ can be examined as a graph of a map, a closed subset of $X\times X$, $G$ can be examined as a subspace of the hyperspace $2^{X\times X}$ with the Vietoris topology. The same works if $X$ is a locally compact space and $G$ is examined as a subspace of the hyperspace $2^{X\times X}$ with the Fell topology.  This approach was studied in~\cite{Kennedy}, \cite{usp2001}, and its different applications can be found, for example, in~\cite{U1998}, \cite{usp2008}, \cite{Yamashita}. To use this approach, one must ensure that the topological group is topologically embedded in the hyperspace $2^{X\times X}$. On these compactifications of a topological group, the involution of a group and the right and left actions of a group on itself by multiplication are continuously extended. Therefore, it is ``algebraically near''\ to the Roelcke compactification of a group. This consideration was effectively used, for example, in~\cite{U1998}, \cite{usp2001}, and~\cite{Megr2001}. 

In~\cite{KozlovLeiderman2025}, the classification of proper compactifications from the point of view of the possibility of extension of algebraic operations is given. It uses both the above-mentioned methods of constructing compactifications. 

\medskip

In this paper, we examine in detail the method of ''graph compactifications'',  apply both methods to investigate possible extensions of algebraic operations on a topological group to its compactifications, and provide descriptions of compactifications of subgroups of the permutation group (in the permutation topology) and automorphism groups of LOTS (in the topology of pointwise convergence). Additionally, using dichotomy theorems of A.\,V.~Arhangelskii~\cite{Arh2008} and~\cite{Arh2009}, we show that the description of compactifications can be used in the investigation of topological properties of their remainders. The paper is a continuation of papers~\cite{KozlovSorin}, \cite{KozlovSorin2025}, \cite{KozlovLeiderman2025}. Notations and terminology are from these papers. The connections of new results with results previously obtained by other authors are given in the corresponding remarks. 

\bigskip

{\it Auxiliary results of the paper.} 

In \S\ \ref{tauRepr}, a notion of $\tau_g$-representation of a topological group is introduced (Definition~\ref{corepres}). A case of $\tau_g$-representation of a topological group in a locally compact space is used to obtain compactifications of the acting group as the subspace of the hyperspace of binary relations in the Fell topology.  

Examples of $\tau_p$-representations are the isometry groups of metric spaces, or more generally, uniformly equicontinuous actions.  $\tau_p$-representation of a topological group allows one to obtain the right topological semigroup compactification of the acting group as an enveloping Ellis semigroup~\cite{KozlovSorin2025}. $\tau_p$-representation implies $\tau_g$-representation (Corollary~\ref{coinctopolog}).  $\tau_g$- and  $\tau_p$-representations are preserved when passing to $G$-extensions (Lemmas~\ref{l11} and~\ref{l1}). 

\medskip

In \S\ \ref{loccompext}, an order on (equivalence classes of) locally compact extensions of topological spaces is introduced. 

\medskip

In \S\ \ref{hyper} maps of hyperspaces induced by a prefect map of spaces or an embedding map of spaces are introduced. Their properties are given in Propositions~\ref{maphyperF} and~\ref{subsethyperF}. 

\medskip

In \S\ \ref{class}, it is noted that if an Ellis compactification $e G$ of a topological group $G$ is less than or equal to an sm-compactification $b G$ of $G$, then $e G$ is an sm-compactification of $G$ (Proposition~\ref{hersm}). Lemma~\ref{comptauprep} and Theorem~\ref{compRoelcke} provide an opportunity to compare Ellis and sm-compactifications of a group and its subgroup, using their agreed $\tau_p$-representations. 

\medskip

In \S\ \ref{actionhyperspace}, induced actions on the hyperspace of binary relations are defined, and in \S\ \ref{maphyperspace}, induced maps of hyperspaces of binary relations are defined. 

\bigskip

{\it Main results of the paper.}

Embedding of a group in a hyperspace of binary relations is discussed in Section~\ref{binrel}. If 
$$G\ \mbox{is topologically isomorphic to}\ \i_X^{\Gamma}(G)\eqno{(emb)},$$
then its graph compactification (Definition~\ref{graphcomp}) is defined. If a topological group $G$ is $\tau_g$-representable in a compact space $X$, then (emb) is valid~\cite{usp2001} (Lemma~\ref{lemcompemb}).

Proposition~\ref{homomF} shows that if a topological group $G$ is $\tau_g$-representable in a locally compact space $X$, then the map that identifies elements of $G$ with its graph is an isomorphism of $G$ on a subgroup of $(2^{X\times X}, \tau_F)$, and is uniformly continuous with respect to the Roelcke uniformity $L\wedge R$ on $G$ and the unique uniformity on compactum $2^{X\times X}$.

Theorem~\ref{propPreimage} shows that the property (emb) is preserved when passing to perfect equivariant preimages of $X$ or to spaces in which $X$ is an invariant subspace. Theorem~\ref{mapcompV} sharpens the relations of the corresponding graph compactifications of $G$. 

Theorems~\ref{mapcompVel} and~\ref{restrmap} provide characterizations of when the property (emb) is preserved when passing to perfect equivariant images of $X$ and open invariant subsets of $X$, respectively, and provide relations of the corresponding graph compactifications of $G$. 

Theorem~\ref{partialcases} shows that if a topological group $G$ is $\tau_g$-representable in a locally compact space $X$, then the following conditions 
\begin{itemize}
\item[{\rm (a)}] the family of homeomorphisms from $G$ is topologically equicontinuous, or 
\item[{\rm (b)}] $X$ is a locally compact, locally connected space,  
\end{itemize}
are sufficient for the equivalence of graph compactifications of $G$ from its $\tau_g$-representations in $X$ and its Alexandroff one-point compactification. 

In Theorem~\ref{suffcond}, the criterion is given when a group is Roelcke precompact and its Roelcke compactification is equal to the graph compactification from its $\tau_g$-representation in a compact space $X$. 

\bigskip

{\it Examples and applications.}

The usage of the unitary group of a (separable) Hilbert space in representation theory is a link between group theory and functional analysis.  Recall that a unitary group is strongly Eberlein (Roelcke-compactification and Hilbert compactification coincide (and coincide with WAP-compactification))~\cite{GlasnerMegr}.  Theorem~\ref{comptauprep} and Proposition~\ref{discrsmcomp} yield that any subgroup $G$ of the permutation group has proper WAP-compactification (Corollary~\ref{properpermgroup}). The proof uses an embeddability of $G$ in a unitary group. 

In Theorem~\ref{equivcompperm} and Corollary~\ref{wappermgr}, it is shown that for a permutation group ${\rm S}(X)$ of an infinite set $X$ (in permutation topology), the Roelcke-compactification, WAP-compactification, Ellis compactification from its $\tau_p$-representation in the Alexandroff one-point compactification $\alpha X$, and graph compactifications from its $\tau_g$-representations in $X$ and $\alpha X$ coincide. 

Moreover, firstly, this compactification is realized as the closure of ${\rm S}(X)$ in the Roelcke-compactification of the unitary group ${\rm U}(\ell^2(X))$.  Secondly,  the Roelcke-compactification of ${\rm S}(X)$ is algedraically isomorphic to the symmetric inverse semigroup $I_X$ (Proposition~\ref{prop2}, see, also, \cite[Theorem 4.1]{KozlovSorin2025}). 

\medskip

In \S\ \ref{grcompautGO} graph compactifications of the group $\aut (X)$ of automorphisms of an ultrahomogeneous chain $X$ (in the permutation topology) are examined. The relations of graph compactifications are in Theorem~\ref{Prodescrlotbchain}. The graph compactification of $\aut (X)$ from its $\tau_g$-representation in the maximal equivariant compactification of $X$ is the Roelcke compactification of $\aut (X)$ (item (I) of Theorem~\ref{thmOrderchain}). 

In Theorem~\ref{thmOrderchain}, the comparison of Ellis and graph compactifications of $\aut (X)$ (using the fact that the map of Ellis compactifications is a homomorphism of semigroups~\cite[Proposition 3.26]{KozlovLeiderman2025}) yields that the graph compactification of $\aut (X)$ from its $\tau_g$-representation in the Alexandroff one-point compactification $\alpha X$ of $X$ is the WAP-compactification of $\aut (X)$. 

Moreover, firstly, WAP compactification is realized as the closure of $\aut (X)$ in the Roelcke-compactification of the unitary group ${\rm U}(\ell^2(X))$ (Corollary~\ref{autchainWAP}). Secondly, the WAP-compactification of $\aut (X)$ is algebraically isomorphic to the semitopological inverse semigroup of partial automorphisms of $X$. 

\medskip 

In Section~\ref{ultLOT}, graph compactifications of the group $\aut (X)$ of automorphisms of an ultrahomogeneous LOTS $X$ (in the topology of pointwise convergence) are examined. The relations of graph compactifications are in Theorem~\ref{thmOrderLOTS}. The graph compactification of $\aut (X)$ from its $\tau_g$-representation in the maximal equivariant compactification of $X$ is the Roelcke compactification of $\aut (X)$ (Theorem~\ref{thmOrderLOTS}). 

In Theorem~\ref{WAPtrivial}, the comparison of Ellis and graph compactifications of $\aut (X)$ yields that the WAP-compactification of $\aut (X)$ is trivial.

\medskip  

In Section~\ref{remainders}, the descriptions of compactifications of the above topological groups and dichotomy theorems of A.\,V.~Arhangelskii yield the following results (Theorem~\ref{remainderstp}).  

Let $X$ be a discrete space.
\begin{itemize}
\item[{\rm (1)}] Every remainder of $G={\rm U}\, (\ell^2(X))$ {\rm(}in strong operator topology{\rm)}, $({\rm S}(X), \tau_p)$ or $(\aut X, \tau_{\partial})$ {\rm(}$X$ is ultrahomogeneous chain{\rm)} is Lindel\"of, $\sigma$-compact and $G$ is \v{C}ech-complete iff $X$ is countable. 
\item[{\rm (2)}] Every remainder of $G={\rm U}\, (\ell^2(X))$ {\rm(}in strong operator topology{\rm)}, $({\rm S}(X), \tau_{\partial})$ or $(\aut X, \tau_{\partial})$ {\rm(}$X$ is ultrahomogeneous chain{\rm)} is pseudocompact, Baire and $G$ is not \v{C}ech-complete iff $X$ is uncountable. 
\end{itemize}

Let $X$ be an ultrahomogeneous {\rm LOTS}. 
\begin{itemize}
\item[{\rm (1')}] Every remainder of $(\aut X, \tau_p)$ is Lindel\"of, $\sigma$-compact and  $(\aut X, \tau_p)$ is \v{C}ech-complete iff $X$ is continuously dense and separable. 
\item[{\rm (2')}]  Every remainder of $(\aut X, \tau_p)$ is Lindel\"of,  Baire and  $(\aut X, \tau_p)$ is not \v{C}ech-complete iff $X$ is not continuously dense and separable. 
\item[{\rm (3')}]  Every remainder of $(\aut X, \tau_p)$ is pseudocompact, Baire and $(\aut X, \tau_p)$ is not \v{C}ech-complete iff  $X$ is not separable. 
\end{itemize}

These considerations give examples of pseudocompact weakly Lindel\"of spaces which are not compact.


We would like to thank  professor M.~Megrelishvili and  G.~Sorin for their attention to this work and for their invaluable help.

\medskip

We adopt terminology and notations from~\cite{Engelking} and~\cite{RD}. 

All sets $X$ are infinite, $\Sigma_X$ is the family of all finite subsets of $X$.  $|X|$ is the cardinality of the set (space) $X$. Equivalence relations on a set are equipped with a natural order. For a set $X$ and an equivalence relation $\sim$, $X/\sim$ is the set of equivalence classes. 

All spaces considered are (topological) Tychonoff, and we denote, if necessary, a (topological) space as $(X, \tau)$ where $\tau$ is a topology on the set $X$. We order topologies: $\tau'\geq\tau$ iff $\tau\subset\tau'$. An abbreviation ``nbd''  (of a set) refers to an open neighbourhood. A nbd of a point $x$ is denoted $O_x$. For a subset $X$ of a space $Y$, $\Int X$ (or $\Int_Y X$) and $\cl X$ (or $\sl_Y X$) are the interior and closure of $X$ in $Y$, respectively.
 
The family of all nbds of the unit $e$ in the topological group $G$ is denoted $N_{G}(e)$. Isomorphism $f: G\to H$ of topological groups $G$ and $H$ is a {\it topological isomorphism}, if $f$ is a homeomorphism. All the necessary (and additional) information about topological groups can be found in~\cite{ArhTk}. 

Maps are continuous maps. The onto quotient map $f: X\to Y$ is an {\it elementary map} if the preimages of all points of $Y$, except maybe one, are one-point sets~\cite{Fed2003} (if $Z$ is a not one-point preimage then we use the notation $Y=X/Z$). $\id$ is the identity map. 

A proper compactification  $(bX, b)$ of a space $X$ is a compact space $b X$ and a dense embedding $b: X\to b X$. The order on compactifications: $(b'X, b')\geq (bX, b)$ iff $\exists\ \varphi: b' X\to b X$ such that $\varphi\circ b'=b$. $\varphi$ is called the {\it map of compactifications}. $\mathbb{C} (X)$ denotes the set  of compactifications of $X$.

Uniform space is denoted as $(X, \mathcal U)$, $u$ is a uniform cover and ${\rm U}$ is the corresponding entourage of $\mathcal U$. Uniformity on a topological space is compatible with its topology. We write $\mathcal U\subset\mathcal V$ if the identity map  $\id: (X,\mathcal V)\to (X,\mathcal U)$ of uniform spaces is uniformly continuous. 

For ${\rm U}\in\mathcal U$ ${\rm U}\circ {\rm U}=\{(x, z)\ |\ \exists\ y\in X\ \mbox{such that}\ (x, y)\in {\rm U}, (y, z)\in {\rm U}\}$. For $x\in X$ and ${\rm U}\in\mathcal U$ the {\it ball} (with center $x$ and radius ${\rm U}$) is the set $B(x, {\rm U})=\{y\in X\ |\ (x, y)\in {\rm U}\}$. The star of a point $x$ with respect to the cover $u$ is denoted $\st (x, u)$. If ${\rm U\in \mathcal U}$ corresponds to a uniform cover $u\in \mathcal U$, then $B(x, {\rm U})=\st (x, u)$. If a cover $v$ is a refinement of a cover $u$, we use the notation $v\succ u$. For $A\subset X$, $A\ne\emptyset$, 
$$\st(A, u)=B(A, {\rm U})=\bigcup\{B(x, {\rm U})=\st (x, u)\ |\ x\in A\}.$$

For a uniform space $(X, \mathcal U)$ and $Y\subset X$ denote by $\mathcal U\wedge Y$ the restriction to $Y$ of the uniformity $\mathcal U$.  $(Y, \mathcal U\wedge Y)$ is the subspace of the uniform space $(X, \mathcal U)$~\cite[Ch.~8, \S\ 8.2]{Engelking}. For a cover $u$ of $Y$ and $X\subset Y$ $u\wedge X=\{U\cap X\ |\ U\in u\}$. If $X$ is a subset of a uniform space $(Y, \tilde{\mathcal U})$, then $\mathcal U=\{u\wedge X\ |\ u\in\tilde{\mathcal U}\}$ is the base of the {\it subspace uniformity} on $X$. The completion of $(X, \mathcal U)$ is denoted $(\tilde X, \tilde{\mathcal U})$. All the necessary (and additional) information about uniform structures (on groups) can be found in~\cite{Isbell} and~\cite{RD}.

\begin{lem}\label{cont}
The map $f: (X, \mathcal U)\to (Y, \mathcal V)$ of uniform spaces is continuous {\rm(}in the induced topologies on $X$ and $Y$ by the correspondent uniformities{\rm)} at the point $x\in X$ iff $\forall\ V\in\mathcal V$  $\exists\ U\in\mathcal U$ such that $f(B(x, U))\subset B(f(x), V)$. 
\end{lem}

\begin{proof} 
Necessity. $\forall\ V\in\mathcal V$  the set $\Int B(f(x), V)$ is an open nbd of $f(x)$. There exist a nbd $O_x$ of $x$ such that $f(O_x)\subset\Int B(f(x), V)$ and $U\in\mathcal U$ such that $B(x, U)\subset O_x$. Then $f(B(x, U))\subset B(f(x), V)$. 

Sufficiency. $\forall\ O_{f(x)}$ there exist $V\in\mathcal V$ such that $B(f(x), V)\subset O_{f(x)}$ and $U\in\mathcal U$ such that $f(B(x, U))\subset B(f(x), V)$. Then $\Int B(x, U)$ is a nbd of $X$ and $f(\Int B(x, U))\subset  O_{f(x)}$.
\end{proof}


\section{Preliminaries} 

\subsection{Topological groups (uniformities, involution, left action)}\label{topgroup}
On a topological group $G$ four uniformities are well-known. The {\it right uniformity} $R$ (the base is formed by the covers $\{Og=\bigcup\{hg\ |\ h\in O\}\ |\ g\in G\}$, $O\in N_G(e)$), the {\it left uniformity} $L$ (the base is formed by the covers $\{gO=\bigcup\{gh\ |\ h\in O\}\ |\ g\in G\}$, $O\in N_G(e)$), the {\it two sided uniformity} $R\vee L$ (the least upper bound of the right and left uniformities) and the {\it Roelcke uniformity} $L\wedge R$ (the greatest lower bound of the right and left uniformities) (the base is formed by the covers $\{OgO=\bigcup\{hgh'\ |\ h, h'\in O\}\ |\ g\in G\}$, $O\in N_G(e)$). A group  $G$ is {\it Roelcke precompact} if Roelcke uniformity is totally bounded. 

All the necessary information about these uniformities can be found in~\cite{RD}. 

{\it Roelcke compactification} $b_r G$ of $G$ is the Samuel compactification of $(G, L\wedge R)$, i.e. the completion of $(G, (L\wedge R)_{fin})$, where $(L\wedge R)_{fin}$ is the {\it precompact replica} of $L\wedge R$. If $G$ is Roelcke precompact, then $(L\wedge R)_{fin}=L\wedge R$ and {\it Roelcke compactification} $b_r G$ is the completion of $(G, L\wedge R)$.

\medskip 

On a topological group $G$ the left multiplication defines the {\it left action} $\alpha: G\times G\to G$, $\alpha (g , h)=gh$. The action $\alpha$ defines the right uniformity $R$ on $G$ with the base 
$$\{Ox=\{\alpha (g, x)\ |\  g\in O\}\ |\  x\in G\},\ O\in N_G(e).$$

\medskip 

An {\it involution on a topological group} $G$ is a self-inverse map $*:G\to G$, $*(g)=g^{-1}$, $g\in G$, and $*(gh)=*(h)*(g)$. An {\it inverse} of $g$ is $g^{-1}$. The {\it inverse} ({\it left}) {\it action}  
$$\alpha^*(g, h)=*(\alpha (g, *(h)))=hg^{-1},\ g, h\in G,$$
is defined and an action and its iverse commutes 
$$\alpha^*(g, \alpha (h, x))=\alpha^*(g, hx)=hxg^{-1}=\alpha(h, \alpha^*(g, x)),\ g, h, x\in G.$$

The base of the left uniformity $L$ on $G$ can be defined as 
$$\{xO=\{\alpha^* (g, x))\ |\  g\in O\}\ |\ x\in G\},\ O\in N_G(e),$$
and the base of the Roelcke uniformity $L\wedge R$ on $G$ can be defined as 
$$\{OxU=\{\alpha(h, \alpha^* (g, x))\ |\  h\in O,\ g\in U\}\ |\ x\in G\},\ O, U\in N_G(e).$$


\subsection{$G$-spaces and $G$-extensions, self-inverse maps}

A $G$-space is a triple $(G, X, \theta)$, where $G$ is a topological group, $X$ is a space and $\theta: G\times X\to X$ is a left continuous action (abbreviation {\it $X$ is a $G$-space}). The following abbreviations $G\curvearrowright X$, $\theta (g, x)=gx$ are used for the action if it is clear from the context. 
$$\theta (A, Y)=AY=\bigcup\{\theta (g, x)=gx\ |\ g\in A\subset G,\ x\in Y\subset X\}.$$
A subset $Y$ of a $G$-space $X$ is an {\it invariant subset} if $GY=Y$. The restriction $\theta|_{G\times Y}$ of an action $\theta: G\times X\to X$ to an invariant subset $Y$ is an action and  $(G, Y, \theta|_{G\times Y})$ is a $G$-space.

A subgroup $\St_x=\{g\in G\ |\ gx=x\}$ of $G$ is a {\it stabilizer} of a point $x\in X$. An action is {\it effective} if $\{g\in G\ |\ gx=x,\ \forall\ x\in X\}=\bigcap\{\St_x\ |\ x\in X\}=e$. Only effective actions are considered in the paper. 

A map $f: X\to Y$ is a {\it $G$-map} (or {\it equivariant map}) of $G$-spaces $(G, X, \theta_X)$ and  $(G, Y, \theta_Y)$ if $f\circ\theta_X=\theta_Y\circ (\id\times f)$.  If $f$ is not continuous, then it will be noted.

\medskip

A map $s: X\to X$ is a {\it self-inverse} map if $s\circ s=\id$. 

\begin{df}~\cite[Definition 2.1]{KozlovLeiderman2025}\label{commmapsinv}
For spaces $Y$ and $X$ with self-inverse maps $s_Y$ and $s_X$ respectively, a map $f: Y\to X$ {\it commutes with self-inverse maps} if $$f\circ s_Y=s_X\circ f.$$ 
\end{df}

\begin{lem}\label{invaction}~\cite[Lemma 2.2]{KozlovLeiderman2025}
If $(G, X, \theta)$ is a $G$-space with a self-inverse map $s$, then the {\it inverse action}  
$$\theta^* (g, x)=s(\theta (g, s(x))),\ g\in G,\ x\in X,$$
is correctly defined and $(\theta^*)^*=\theta$. 
\end{lem}

\begin{df}\label{commactioninv}
For a $G$-space $(G, X, \theta)$ with a self-inverse map $s$ the action $\theta$ commutes with inverse action if 
$$\theta (g, \theta^*(h, x))=\theta^*(h, \theta (g, x)),\ g, h\in G,\ x\in X.$$
\end{df}

\begin{lem}\label{ginversg}
Let $(G, Y, \theta_Y)$ and  $(G, X, \theta_X)$ be $G$-spaces with self-inverse maps $s_Y$ and $s_X$ respectively, and let $f: Y\to X$ be a not continuous $G$-map and commutes with self-inverse maps $s_Y$ and $s_X$. Then $f$ is a not continuous $G$-map of $G$-spaces $(G, Y, \theta^*_Y)$ and $(G, X, \theta^*_X)$.

If, additionally, the action $\theta_Y$ commutes with inverse action and $f$ is a surjection, then the action $\theta_X$ commutes with inverse action.
\end{lem}

\begin{proof}
The first statement is~\cite[Lemma 2.3]{KozlovLeiderman2025}. 

If $\theta_Y$ commutes with inverse action and $f$ is a surjection, then
$$\theta_X (g, \theta_X^*(h, x))\stackrel{x=f(y)}{=}\theta_X (g, \theta_X^*(h, f(y)))=f(\theta_Y (g, \theta_Y^*(h, y)))=$$
$$=f(\theta^*_Y (h, \theta_Y(g, y)))=\theta_X^*(h, \theta_X (g, x)),\ g, h\in G,\ x\in X.$$
\end{proof}

\medskip

A uniformity $\mathcal U_X$ on $X$ ($(G, X, \theta_X)$ is a $G$-space, $\mathcal U_X$ unduces the original topology on $X$) is called an {\it equiuniformity}~\cite{Megr1984} if the action $G\curvearrowright X$ is {\it saturated} (any homeomorphism from $G$ is uniformly continuous) and is {\it bounded} (for any $u\in\mathcal U$ there are $O\in N_G(e)$ and $v\in\mathcal U$ such that the cover $Ov=\{OV\ |\ V\in v\}\succ u$). In this case  $(G, X, \theta_X)$ is called a {\it $G$-Tychonoff space}  (abbreviations {\it $X$ is a $G$-Tychonoff space}, $\mathcal U_X$ is an equiuniformity for $\theta_X$). The action is extended to the continuous action $\theta_{\tilde X}: G\times \tilde X\to\tilde X$ on the completion $(\tilde X, \tilde{\mathcal U}_X)$ of $(X, \mathcal U_X)$ and $(G, \tilde X, \theta_{\tilde X})$ is a $G$-Tychonoff space. The extension $\tilde{\mathcal U}_X$ of $\mathcal U_X$ to $\tilde X$ is an equiuniformity on $\tilde X$. The embedding $\jmath: X\to\tilde X$ is a $G$-map of $G$-spaces $(G, X, \theta_X)$ and $(G, \tilde X, \theta_{\tilde X})$, and $\jmath (X)$ is a dense {\it invariant} subset of $\tilde X$~\cite{Megr0}. $(G, \tilde X, \theta_{\tilde X})$ or  $(G, (\tilde X, \jmath), \theta_{\tilde X})$ is called a {\it $G$-extension of $(G, X, \theta_X)$}  (abbreviation the pair $(\tilde X, \jmath)$ or simply  $\tilde X$ is a $G$-extension of $X$). 

If $\mathcal U_X$ is a totally bounded equiuniformity on $X$, then $G$-extension $b X=\tilde X$  of $X$ is a compact space. 
$(G, b X, \theta_{b X})$ is a {\it $G$-compactification} or an {\it equivariant compactification} of  $(G, X, \theta)$  (abbreviation $b X$ is a $G$-compactification of $X$). If $(G, X, \curvearrowright)$ is a $G$-Tychonoff space, then the maximal (totally bounded) equiuniformity on $X$ exists. The maximal totally bounded equiuniformity  $\mathcal U_X^{bm}$ on $X$ is the {\it precompact replica} of the maximal equiuniformity $\mathcal U^m_X$ on $X$ (the greatest upper bound of totally bouded uniformities less or equal than $\mathcal U^m_X$ (see~\cite[Ch.~2]{Isbell})). By Smirnov's correspondence  {\it the maximal $G$-compactification} $\beta_G X$ of $X$ corresponds to  $\mathcal U_X^{bm}$ and is the {\it Samuel compactification of $(X, \mathcal U_X^{m})$} (see~\cite[Ch.\ 8, Problem 8.5.7]{Engelking}).

If $(G, X, \curvearrowright)$ is a $G$-Tychonoff space, then $\mathbb{GC}(X)$ is the poset of $G$-compactifications of $X$ (maps of compactifications are $G$-maps).

\begin{lem}\label{invunif}~\cite[Lemma 2.6]{KozlovLeiderman2025}
Let $(G, X, \theta)$ be a $G$-Tychonoff space with a self-inverse map $s$ and let $\mathcal U_X$ be an equiuniformity on $X$. Then 
$(G, X, \theta^*)$ is a $G$-Tychonoff space and 
$$\mathcal U^*_X=\{u^*=\{s(U)\ |\ U\in u\}\ |\ u\in\mathcal U_X\}$$
is an equiuniformity on $X$.
\end{lem}

\begin{cor}~\cite[Lemma 2.7]{KozlovLeiderman2025}
Let $(G, X, \theta)$ be a $G$-Tychonoff space, and let $s$ be a self-inverse map on $X$. 

If $\mathcal U_X$ is an equiuniformity for $\theta$, and $\mathcal U_X=\mathcal U^*_X$, then $\mathcal U_X$ is an equiuniformity for the inverse $\theta^*$.

If $\mathcal U_X$ is an equiuniformity for $\theta$ and its inverse $\theta^*$, then $\mathcal U^*_X$ is an equiuniformity for $\theta$ and $\theta^*$.
\end{cor}

\begin{pro}\label{unifcontRoelcke}~\cite[Proposition 2.8]{KozlovLeiderman2025}
Let $(G, X, \theta)$ and $(G, X, \theta')$ be $G$-Tychonoff spaces, $\mathcal U_X$ be an equiuniformity on $X$ for $\theta$ and $\theta'$. If $f: G\to X$ is a {\rm(}not necessarily continuous{\rm)} $G$-map of $G$-spaces $(G, G, \alpha)$,  $(G, X, \theta)$ and $(G, G, \alpha^*)$,  $(G, X, \theta')$, then $f: (G, L\wedge R)\to (X, \mathcal U_X)$ is a uniformly continuous and, hence, a continuous map.
\end{pro}

From Proposition~\ref{unifcontRoelcke} and Lemma~\ref{ginversg} it follows. 

\begin{cor}\label{cor1}
{\rm(A)} Let $(G, X, \theta)$ be a $G$-space with a self-inverse map $s$. If $f: G\to X$ is a {\rm(}not necessarily continuous{\rm)} $G$-map of $G$-spaces $(G, G, \alpha)$ and $(G, X, \theta)$ and commutes with involution on $G$ and self-inverse map $s$, then $f$ is a {\rm(}not necessarily continuous{\rm)} $G$-map of $G$-spaces $(G, G, \alpha^*)$ and  $(G, X, \theta^*)$.

{\rm(B)} If, additionally to {\rm(A)}, $f$ is a surjection, then $\theta$ commutes with inverse action. 

{\rm(C)} If, additionally to {\rm(A)}, $X$ is a $G$-Tychonoff space and $\mathcal U_X$ is an equiuniformity on $X$ for $\theta$ and  $\theta^*$, then $f: (G, L\wedge R)\to (X, \mathcal U_X)$ is a uniformly continuous and, hence, continuous map.
\end{cor}


\subsection{Representability of topological groups ($g$-topology, compact-open topology and topology of pointwise convergence)}\label{tauRepr}
For a space $X$ let $\Hom (X)$ be a group of homeomorphisms of $X$. A group $G$ is {\it representable in $X$} if $G$ is a subgroup of  $\Hom (X)$. 

If $G$ is representable in $X$, then the subbase of the {\it $g$-topology} $\tau_{g}$ on $G$ is formed by the sets 
$$[K, O]=\{g\in G\ |\ gK\subset O\},\ \mbox{where}\ K\ \mbox{is a closed},\ O\ \mbox{is an open subset of}\ X,$$
$$\mbox{and either}\ K\ \mbox{or}\ X\setminus O\ \mbox{is compact}.$$

\begin{rem}\label{rem1}
{\rm (i) For groups $G$ representable in locally compact spaces $X$ the notion of a $g$-topology is introduced by R.~Arens~\cite{Arens}. In this case $\tau_g$ is an {\it admissible group topology}, i.e. $(G, \tau_g)$ is a topological group and the action  $(G, \tau_g)\curvearrowright X$ is continuous. Moreover, $\tau_g$ 
is the least admissible group topology on $G$. 

(ii) If $G$ is representable in a compact or a locally compact locally connected space $X$, then $g$-topology coincides with  {\it compact-open topology} (abbreviation c-o.t.) $\tau_{co}$~\cite{Arens}. The subbase of $\tau_{co}$ on $G$ is formed by the sets 
$$[K, O]=\{g\in G\ |\ gK\subset O\},\ \mbox{where}\ K\ \mbox{is a compact},\ O\ \mbox{is an open subset of}\ X.$$
In~\cite{Dijkstra} this result is improved. If $G$ is representable in a noncompact space $X$ and $\forall\ x\in X$ $\exists$ a nbd of $x$ that is a continuum, then $\tau_{co}=\tau_g$ is the least admissible group topology on $G$.

(iii) If $G$ is representable in a locally compact space $X$, then the $g$-topology $\tau_g$ on $G$ coincides with the c-o.t $\tau_{co}$ on $G$ for its representation on the Alexandroff one-point compactification $\alpha X$.}
\end{rem}

\begin{df}\label{corepres}
A topological group $(G, \tau)$ is {\it $\tau_g$-representable in $X$} if $G$ is representable in $X$,  $\tau_g=\tau$ and $((G, \tau), X, \curvearrowright)$ is a $G$-Tychonoff space.
\end{df}

\begin{lem}\label{l11} If a topological group $(G, \tau)$ is $\tau_g$-representable in $X$ and $Y$ is a $G$-extension of $X$, then $(G, \tau)$ is $\tau_g$-representable in $Y$.
\end{lem}

\begin{proof} 
Take $Z\in\mathbb{GC} (Y)$. Since $Y$ is a $G$-extension of $X$, $Z\in\mathbb{GC} (X)$. For the representation of $G$ in $Z$ (extension of action) the c-o.t. $\tau_{co}^Z$ is the least admissible group topology on $G$. Therefore, $\tau_{co}^Z\leq\tau$. 

Take arbitrary open set $[K, O]$ from the subbase of $\tau$, where $K$ is a closed, $O$ is an open subset of $X$ and either $K$ is compact or $X\setminus O$ is compact.  
In the case $K$ is compact, take an open in $Z$ set $O'$ such that $O'\cap X=O$. Then the set $[K, O']$ is from the subbase of $\tau^{Z}_{co}$ and if $g\in [K, O']$, then $g\in [K, O]$, since $X$ is an invariant subset of $Z$. Hence, $[K, O']=[K, O]$.

In the case $X\setminus O$ is compact, take compact set $K'=\cl_Z\ F$ and open set $O'=Z\setminus (X\setminus O)$ in $Z$. Then the set $[K', O']$ is from the subbase of $\tau^{Z}_{co}$ and if $g\in [K', O']$, then $g\in [K, O]$, since $X$ is an invariant subset of $Z$. Hence, $[K', O']=[K, O]$.

Therefore, $\tau\leq\tau^{Z}_{co}$. Finally, $\tau=\tau^{Z}_{co}$. 

The above resonings yield the inequalities $\tau\leq\tau^Y_g\leq\tau^{Z}_{co}$. Hence,  $\tau=\tau^Y_g$.
\end{proof}

If $G$ is representable in $X$, then the subbase of the {\it topology of pointwise convergence} (abbreviation t.p.c.) $\tau_p$ on $G$ is formed by the sets 
$$[x, O]=\{g\in G\ |\ gx\in O\},\ \mbox{where}\ x\in X,\ O\ \mbox{is an open subset of}\ X.$$
Evidently, $\tau_p\leq\tau_{co}$. If $\tau_{p}$ is an admissible group topology on $G$ for its representation in $X$, then $\tau_{p}$ is the least admissible group topology on $G$~\cite[Lemma 3.1]{Kozlov2022}. This is possible, for example, if the action $G\curvearrowright X$ in representation is uniformly equicontinuous~\cite[Proposition 3.2]{Kozlov2022}. 

\begin{lem}\label{tauptauco} 
Let $G$ be representable in $X$ and $\tau_p$ be an admissible group topology on $G$. Then $\tau_p=\tau_g$. 
\end{lem}

\begin{proof}
Since $\tau_p\leq\tau_g$, it remains to show that any nbd $[K, O]$ of $g\in G$, where either $K$ or $X\setminus O$ is a compact subset of $X$, from the subbase of $\tau_g$, contains a nbd of $g$ from $\tau_{p}$. 

The case $K$ is compact. Since $gK\subset O$, $\forall\ x\in gK$ $\exists$ nbd $O_x$ of $x$ and $\exists$ $V_x\in N_G (e)$ in $\tau_{p}$ such that $V_xO_x\subset O$. Compactness of $gK$ implies that there exists a finite family of nbds $O_{x_1}, \ldots, O_{x_n}$ such that $gK\subset O'=\bigcup\limits_{i=1}^n O_{x_i}$. $V=\bigcap\limits_{i=1}^n V_{x_i}\in N_G (e)$ in $\tau_{p}$. The inclusion $VO'\subset O$ implies that  $Vg\subset [K, O]$. 

The case $X\setminus O$ is compact. Since $gK\subset O$, $\forall\ x\in X\setminus O$ $\exists$ nbd $O_x$ of $x$ and $\exists$ $V_x\in N_G (e)$ in $\tau_{p}$ such that $V_xO_x\subset X\setminus gK$. Compactness of $X\setminus O$ implies that there exists a finite family of nbds $O_{x_1}, \ldots, O_{x_n}$ such that $gK\subset X\setminus O'$, $O'=\bigcup\limits_{i=1}^n O_{x_i}$. $V=\bigcap\limits_{i=1}^n V_{x_i}\in N_G (e)$ in $\tau_{p}$. The inclusion $VO'\subset X\setminus gK$ implies that  $Vg\subset [K, O]$. 

Therefore,  $\tau_g\leq\tau_{p}$ and, finally, $\tau_p=\tau_g$. 
\end{proof}

\begin{rem} 
{\rm Let $G$ be representable in a locally compact space $X$. Then $\tau_p$ is an admissible group topology on $G$ iff $\tau_p=\tau_g$.}
\end{rem}

\begin{df}~\cite[Definition 3.4]{KozlovSorin2025}\label{taurepres}
A topological group $(G, \tau)$ is {\it $\tau_{p}$-representable in $X$} if $G$ is representable in $X$,  $\tau_{p}=\tau$ and $((G, \tau), X, \curvearrowright)$ is a $G$-Tychonoff space.
\end{df}

\begin{cor}\label{coinctopolog}
If a topological group $(G, \tau)$ is $\tau_p$-representable in $X$, then $(G, \tau)$ is $\tau_g$-representable in $X$.
\end{cor}

\begin{rem} 
{\rm Example from~{\rm\cite{Megr1988}} shows that $G$  may be representable in $X$ and the t.p.c. $\tau_p$ may be an admissible group topology on $G$,  but the $G$-space $(G=(G, \tau_p), X,  \curvearrowright)$ may not be $G$-Tychonoff.}
\end{rem}

\begin{lem}\label{l1} 
If a topological group $(G, \tau)$ is $\tau_{p}$-representable in $X$ and $Y$ is a $G$-extension of $X$, then $(G, \tau)$ is $\tau_{p}$-representable in $Y$.
Moreover, $\tau$ is the least admissible group topology for the action  $G\curvearrowright Y$. 
\end{lem}

\begin{proof}
Take $Z\in\mathbb{GC} (Y)$. Since $Y$ is a $G$-extension of $X$, $Z\in\mathbb{GC} (X)$. For the representation of $G$ in $Z$  (extension of action) the c-o.t. $\tau_{co}^Z$ is the least admissible group topology on $G$. Therefore, $\tau_{co}^Z\leq\tau$. 

For the t.p.c. $\tau_p^Z$ on $G$ for its representation in $Z$ one has $\tau\leq\tau_p^Z\leq\tau_{co}^Z$. For the t.p.c. $\tau_p^Y$ on $G$ for its representation in $Y$ one has $\tau\leq\tau_p^Y\leq\tau_p^Z$. Therefore, $\tau=\tau_p^Y$. 
The extremal property of the t.p.c. finishes the proof.
\end{proof}

\begin{rem}
{\rm Lemma~\ref{l1} is proved in~\cite[Lemma 2.2]{KozlovSorin2025} for $Y\in\mathbb{GC}(X)$.}
\end{rem}

\begin{rem}
{\rm If a topological group $(G, \tau)$ is $\tau$-representable in a locally compact space $X$ and $\tau=\tau_g$ or  $\tau=\tau_p$, then the condition $((G, \tau), X, \curvearrowright)$ is a $G$-Tychonoff space in Definitions~\ref{corepres} and~\ref{taurepres} respectively, can be omitted.}
\end{rem}


\subsection{Locally compact extensions of a topological space}\label{loccompext}

A pair $(Y, \varphi)$ where $Y$ is a locally compact space, $\varphi: X\to Y$ is an embedding and $\cl\ \varphi (X)=Y$ is a {\it locally compact extension of $X$}. 

Locally compact extensions $(Y, \varphi)$ and $(Z, \psi)$ of $X$ are {\it equivalent} if $\exists$ an open suhset $\varphi (X)\subset U\subset Y$ and a perfect onto map $f: U\to Z$ and $\exists$  an open suhset $\psi (X)\subset V\subset Z$ and a perfect onto map $h: V\to Y$ such that $\psi=f\circ\varphi$ and  $\varphi=h\circ\psi$. The defined equivalence is an equivalence relation on the set of locally compact extensions of $X$. 

Denote by $\mathbb{LC} (X)$ the set of  equivalence classes of locally compact extensions of $X$. A partial order $\geq_{LC}$ on $\mathbb{LC} (X)$
$$(Y, \varphi)\geq_{LC} (Z, \psi)\ \mbox{if}\ \exists\  \mbox{an open suhset}\ \varphi (X)\subset U\subset Y,$$
$$\ \mbox{and a perfect onto map}\ f: U\to Z\ \mbox{such that}\ \psi=f\circ\varphi\eqno{({\rm od})}$$
is  well defined. 

\begin{pro}
\begin{itemize}
\item  $\geq_{LC}$ induces the natural order on $\mathbb{C} (X)\subset\mathbb{LC} (X)$.
\item If $(Y, \varphi)\in\mathbb{LC} (X)\setminus\mathbb{C} (X)$, $(Z, \psi)\in\mathbb{C} (X)$, then $(Y, \varphi)\not\geq_{LC} (Z, \psi)$.
\item If $(Y, \varphi)\geq_{LC} (Z, \psi)$ and $(K, \xi)\in\mathbb{C} (Y)$, then $(K, \xi\circ\varphi)\in\mathbb C(X)$ and $(K,  \xi\circ\varphi)\geq_{LC} (Z, \psi)$.  
\end{itemize}
\end{pro}

If $X$ is a $G$-Tychonoff space, then the  partial order $\geq_{GLC}$ on $\mathbb{GLC} (X)$, the set of  equivalence classes of locally compact $G$-extensions of $X$, 
$$(Y, \varphi)\geq_{GLC} (Z, \psi)\ \mbox{if}\ \exists\  \mbox{an open invariant suhset}\ \varphi (X)\subset U\subset Y,$$
$$\quad\quad\ \quad\quad\ \quad\ \mbox{and a perfect onto $G$-map}\ f: U\to Z\ \mbox{such that}\ \psi=f\circ\varphi\eqno{({\rm odG})}$$
is induced by the  the order  $\geq_{LC}$.  


\subsection{Hyperspace}\label{hyper}

Designations are from~\cite{Beer}. $\CL (X)$ is the family of nonempty closed subsets of $X$, $2^X=\CL (X)\cup\{\emptyset\}$. For $E\subset X$ 

$$E^+=\{F\in\CL (X)\ |\ F\subset E\},$$
$$E^-=\{F\in\CL (X)\ |\ F\cap E\ne\emptyset\}.$$

\medskip

{\it Vietoris topology.}
\begin{df}~\cite[Definition 2.2.4]{Beer} 
Let $X$ be a Hausdorff space. The Vietoris topology $\tau_V$ on $\CL (X)$ has as a subbase all sets of the form $W^-$  and all sets 
of the form $W^+$, where $W$ is open in $X$. 
\end{df}

The Vietoris topology is Tychonoff if and only if $X$ is normal. If $X$ is compact, then $(\CL (X), \tau_V)$ is compact.  $(\CL (X), \tau_V)$ is a clopen subset of  $(2^X, \tau_V)$ and $\{\emptyset\}$ is an isolated point of $(2^X, \tau_V)$. 

A base for the Vietoris topology $\tau_V$ on $\CL (X)$ consists of all sets of the form 
$$[V_1, V_2, \ldots, V_n]=\{F\in\CL (X)\ |\ \forall\ i\leq n\ (F\cap V_i\ne\emptyset),\ F\subset\bigcup\limits_{i=1}^n V_i\},\eqno{\rm (V)}$$ 
where $V_1, V_2, \ldots, V_n$ is a finite family of open nonempty subsets of $X$. 

\medskip

{\it Fell topology.}
\begin{df}~\cite[Definition 5.1.1]{Beer} 
Let $X$ be a Hausdorff space. The Fell topology $\tau_F$ on $\CL (X)$ has as a subbase all sets of the form $W^-$, where $W$ is a nonempty open subset of $X$, and all sets of the form $W^+$, where $W$ is a nonempty open subset of $X$ with compact complement. 
\end{df}

The natural extension of the Fell topology on $\CL (X)$ to $2^X$ (a local base for the extended Fell topology at the empty set 
consists of all sets of the form $\{F\in 2^X\ |\ F\cap K=\emptyset\}$, where $K$ is a compact subset of $X$) is without fail a compact (non Hausdorff) space. 
Local compactness of a Hausdorff space $X$ implies that $(\CL (X), \tau_F)$ is a locally compact Hausdorff space and $(2^X, \tau_F)$ (with the extended Fell topology) is a compact Hausdorff space. If $X$ is compact, the Fell and Vietoris topologies on $\CL (X)$ (and $2^X$) coincide. 


\medskip

{\it Action on a hyperspace.}
Every homeomorphism $\varphi\in\Hom (X)$ induces homeomorphism ${\varphi}^H:(2^X, \tau_V)\to (2^X, \tau_V)$ (${\varphi}^H:(2^X, \tau_F)\to (2^X, \tau_F)$), ${\varphi}^H(A)=\varphi (A)$. 

\begin{pro}\label{actionhyper}  {\rm(}see, for example, \cite[Remark 4.4]{DJPLP}{\rm)}
Let $(G, X, \theta)$ be a $G$-space and $X$ be a locally compact space. Then the {\it induced action} $\theta^H: G\times (2^X, \tau_F)\to  (2^X, \tau_F)$, $\theta^H(g, A)=\theta(g, A)=g(A)=\{gx\ |\ x\in A\}$, is continuous and $(G, (2^X, \tau_F), \theta^H)$ is a $G$-space.
\end{pro}

If $s: X\to X$ is a self-inverse map, then the map $s^H: (2^X, \tau_V)\to (2^X, \tau_V)$ ($s^H:(2^X, \tau_F)\to (2^X, \tau_F)$), $s^H(A)=s (A)$, is a self-inverse map.

\begin{pro}\label{hyperinversemap} 
Let $(G, X, \theta)$ be a $G$-space with a self-inverse map $s$ and $X$ be a locally compact space. Then $(\theta^H)^*=(\theta^*)^H$ {\rm(}with respect to the self-inverse map $s^H${\rm)}.

Moreover, if $\theta$ commutes with inverse action, then $\theta^H$  commutes with inverse action.
\end{pro}

\begin{proof} Follows from equalities
$$(\theta^H)^*(g, A)=s^H(\theta^H(g, s^H(A)))=s(\theta(g, s (A)))=\theta^*(g, A)=(\theta^*)^H (g, A),$$
$$\theta^H (g, (\theta^H)^*(h, A))=\theta(g, \theta^*(h, A))=\theta^*(h, \theta (g, A))=(\theta^H)^*(h, \theta^H (g, A)),\ g, h\in G,\ A\in 2^X.$$ 
\end{proof}


\medskip

{\it Maps of hyperspaces.}

\begin{pro}\label{maphyperF} Let $Y$, $X$ be locally compact spaces, and a map $f: Y\to X$ be onto and perfect. 

{\rm (A)} The map $F: (2^Y, \tau_F)\to (2^X, \tau_F)$,  $F (A)=f (A)$,  $A\in 2^Y$, is onto and perfect. 

\medskip

{\rm (B)} If $(G, Y, \theta_X)$, $(G, X, \theta_Y)$ are $G$-spaces and $f: Y\to X$ is a $G$-map, then $F: 2^Y\to 2^X$ is a $G$-map of $G$-spaces $(G, 2^Y, \theta_Y^H)$ and $(G, 2^X, \theta_X^H)$. 

\medskip

{\rm (C)} If $Y$, $X$ are spaces with self-inverse maps $s_X$ and $s_Y$ respectively, and $f: Y\to X$ commutes with self-inverse maps $s_Y$ and $s_X$, then $F: 2^Y\to 2^X$ commutes with self-inverse maps $s^H_Y$ and $s^H_X$. 

\medskip

{\rm (D)}  If $(G, Y, \theta_X)$, $(G, X, \theta_Y)$ are $G$-spaces with self-inverse maps $s_X$ and $s_Y$ respectively, and a $G$-map $f: Y\to X$ commutes with self-inverse maps $s_Y$ and $s_X$, then $F: 2^Y\to 2^X$ is a $G$-map of $G$-spaces $(G, 2^Y, (\theta^*_Y)^H)$ and $(G, 2^X, (\theta^*_X)^H)$. 

If  $\theta_Y$ commutes with the inverse action, then $\theta_X$, $\theta^H_Y$ and $\theta^H_X$ commute with the correspondent inverse actions.
\end{pro}

\begin{proof} 
(A) $F: (2^Y, \tau_F)\to (2^X, \tau_F)$ is correctly defined and is onto. The contiuous map of compacta is perfect. It remains to verify continuity of $F$. 

To verify continuity of $F$ at a point $A\in 2^Y$ it is enough to show that the preimages of nbds of $F (A)=f (A)$ from the subbase are open. 

Let $U^-$, $U$ is open in $X$, be a nbd of $F (A)\ne\emptyset$ (from the subbase of the Fell topology on $2^X$). For $V=f^{-1} U$ one has 
$$(B\in V^-)\Longrightarrow (F (B)\in U^-),\ B\in 2^Y,\ \mbox{and}\ A\in V^-.$$
Hence, $F (V^-)=U^-$.

Take a nbd $(X\setminus K)^+$, $K$ is a compact subset of $X$, of $F (A)$ (from the subbase of the Fell topology on $2^X$). Since $K$ is a compact subset of $X$ and $f$ is a perfect map, $T=f^{-1} (K)$ is a compact subset of $Y$ and $A\cap T=\emptyset$. Therefore,  $A\in (Y\setminus T)^+$ and 
$$(B\in (Y\setminus T)^+)\Longrightarrow (F (B)\in (X\setminus K)^+),\ B\in 2^Y.$$
Hence, $F ((Y\setminus T)^+)=(X\setminus K)^+$.

\medskip

(B) $F (\theta^H_Y (g, A))=f(\theta_Y (g, A))=\theta_X (g, f(A))=\theta^H_X (g, F (A)),\ g\in G,\ A\in 2^Y$.   

\medskip

(C) Let $f\circ s_Y=s_X\circ f$. Then 
$$F (s^H_Y (A))=f (s_Y(A))=s_X(f(A))=s^H_X(F(A)),\ A\in 2^Y.$$

(D) follows from Proposition~\ref{hyperinversemap} and Lemma~\ref{ginversg}.
\end{proof}

\begin{rem}
{\rm If $X$, $Y$ are compacta and $f: Y\to X$ is an onto map, then the map 
$F: (2^Y, \tau_V)\to (2^X, \tau_V)$,  $F(A)=f(A)$,  $A\in 2^Y$, is onto and perfect~\cite[Ch.\,3, \S\ 3.12, Problem 3.12.27 (e)]{Engelking}.}
\end{rem}

\begin{pro}\label{subsethyperF} Let $Y$ be a locally compact space, $X\subset Y$ be an open subset {\rm(}and, hence, locally compact{\rm)}. 

{\rm(A)} The map $H_{|}: (2^Y, \tau_F)\to (2^X, \tau_F)$,  $H_{|}(A)=A\cap X$,  $A\in 2^Y$, is onto and perfect {\rm(}$H_{|}^{-1}(\emptyset)=\{A\in 2^Y\ |\ A\subset (Y\setminus X)${\rm)}. 

\medskip

{\rm (B)} If, additionally, $(G, Y, \theta_Y)$ is a $G$-space and $X$ is an invarinat subset {\rm(}$\theta_X=\theta_Y|_{G\times X}${\rm)}, then $H_{|}$ is a $G$-map of $G$-spaces $(G, (2^Y, \tau_F), \theta_Y^H)$ and $(G, (2^Y, \tau_F), \theta_X^H)$. 

\medskip

{\rm (C)} If, additionally, $Y$ is a space with self-inverse map $s_Y$ and $s_Y (X)=X$ {\rm(}$s_X=s_Y|_{X}${\rm)}, then $H_{|}$ commutes with self-inverse maps $S_Y$ and $S_X$. 

\medskip

{\rm (D)} If, additionally, $Y$ is a space with self-inverse map $s_Y$ and $s_Y (X)=X$, $(G, Y, \theta_Y)$ is a $G$-space and $X$ is an invarinat subset, then 
$H_{|}$ is a $G$-map of $G$-spaces $(G, (2^Y, \tau_F), (\theta_Y^H)^*)$ and $(G, (2^Y, \tau_F), (\theta_X^H)^*)$. 

If $\theta_Y$ commutes with the inverse action, then $\theta_X$, $\theta^H_Y$ and $\theta^H_X$ commute with the correspondent inverse actions.
\end{pro}

\begin{proof} (A) $H_{|}$ is correctly defined and is onto. The contiuous map of compacta is perfect. It remains to verify continuity of $H_{|}$. 

To verify continuity of $H_{|}$ at a point $A\in 2^Y$ it is enough to show that the preimages of nbds of $H_{|} (A)$ from the subbase are open. 

If $H_{|} (A)=A\cap X\ne\emptyset$, take a nbd $U^-$, $U$ is open in $X$, of $H_{|} (A)$ (from the subbase of the Fell topology on $2^X$).  Since $X$ is open in $Y$, $U$ is an open subset of $Y$. Therefore, $\tilde U^-=\{B\in\CL Y\ |\ B\cap U\ne\emptyset\}$ is a nbd of $A$ in $2^Y$ and $(B\in\tilde U^-)\Longleftrightarrow (H_{|} (B)=B\cap X\in U^-)$. Hence, $H_{|} (\tilde U^-)=U^-$.

Take a nbd $(X\setminus K)^+$, $K$ is a compact subset of $X$, of $H_{|} (A)$ (from the subbase of the Fell topology on $2^X$). Since $K$ is a compact in $X$, $K$ is compact in $Y$. Therefore,  $(Y\setminus K)^+$ is a nbd of $A$ in $2^Y$ and $(B\in (Y\setminus K)^+)\Longleftrightarrow (H_{|} (B)=B\cap X\in (X\setminus K)^+)$.  Hence, $H_{|} ((Y\setminus K)^+)=(X\setminus K)^+$.

\medskip

(B) $H_{|} (\theta^H_Y (g, A))=H_{|}(\{\theta_Y (g, y)\ |\ y\in A\})=\{\theta_Y (g, y)\ |\ y\in A\cap X\}=\{\theta_X (g, x)\ |\ x\in A\cap X\}=\theta^H_X(g, H_{|}(A)),\ g\in G,\ A\in 2^Y$.   

\medskip

(C) $H_{|} (S_Y (A))=H_{|} (\{s_Y(y)\ |\ y\in A\})=\{s_Y(y)\ |\ y\in A\cap X\}=\{s_X(x)\ |\ x\in A\cap X\}=S_X(\{x\ |\ x\in A\cap X\})=S_X(H_{|}(A)),\ A\in 2^Y$.

\medskip

(D) follows from Proposition~\ref{hyperinversemap} and Lemma~\ref{ginversg}.
\end{proof}

\begin{rem}
{\rm Let $Y$ be a locally compact space and $K$ be a compact subset of $Y$. The equivalence relation $x\sim y\ \Longleftrightarrow$ either $x, y\in K$ or $x=y$ and the quotient (perfect and elementary) map $q: Y\to Y/\sim$ are correctly defined. 

For $X=Y\setminus K=(Y/\sim)\setminus\ q(K)$, by Propositions~\ref{maphyperF} and~\ref{subsethyperF} the maps $Q: (2^Y, \tau_F)\to (2^{Y/\sim}, \tau_F)$,  $H_{|}: (2^Y, \tau_F)\to (2^X, \tau_F)$ and $H'_{|}: (2^{Y/\sim}, \tau_F)\to (2^X, \tau_F)$ are defined and $H_{|}=H'_{|}\circ Q$. Indeed for $A\in 2^Y$ 
$$H_{|} (A)=A\cap X,\ Q (A)=\left\{\begin{array}{cl}
A\cap X, & A\cap K=\emptyset, \\
(A\cap X)\cup\{q(K)\},  & A\cap K\ne\emptyset, \\
\end{array}
\right.\  H'_{|} (Q (A))=A\cap X.$$
}
\end{rem}


\medskip

{\it Uniformity on a hyperspace.} Let $(X, \mathcal U)$ be a uniform space. The family of entourages on $2^X\times 2^X$
$$(\emptyset, \emptyset)\cup\{(A, B)\in\CL (X)\times\CL (X)\ |\ B\subset\st (A, u),\ \&\ A\subset\st (B, u)\},\ u\in\mathcal U,$$
is a base of uniformity $2^{\mathcal U}$ on $2^X$~\cite[Problem 8.5.16]{Engelking} or~\cite[Ch.\ 2,\ \S\ Hyperspace]{Isbell}.


\subsection{Ellis compactifications of a topological group}\label{class}
\begin{df}~\cite[Definition 3.10]{KozlovLeiderman2025} {\rm(}see, also, {\rm\cite[Definition 2.1]{KozlovSorin})}\label{defElliscomp}
Let $G$ be a topological group.  An Ellis compactification $b G$ of $G$ is a $G$-compactification of $G$ such that $(bG, \bullet)$ is a right topological monoid and $\bullet|_{G\times b G}=\tilde\alpha$, where $\tilde\alpha$ is the extension of the action $\alpha$ of $G$ on itself by multiplication on the left to $b G$. 
\end{df}

$\mathbb{E} (G)$ (a nonempty set) is the poset of Ellis compactifications of $G$.

\medskip

Let $G$ be a group of bijections of a set $X$. $X^X$ is a monoid with composition of maps as a binary operation and identity map as a unit. The injective map 
$$\i_X: G\to X^X,\ \i_X(g)(x)=g(x),\ x\in X,$$ 
is defined. On $\i_X(G)$ multiplication is the restriction of composition on $X^X$.  If $X$ is a space, then $X^X$ in Tychonoff product topology is a right topological monoid. If $(G, X, \theta)$ is a $G$-(Tychonoff) space, then $(G, X^X, \theta_{\Delta X})$ is a $G$-(Tychonoff) space with the {\it diagonal action} $\theta_{\Delta X}$. 

If a topological group $(G, \tau)$ is $\tau_p$-representable in a compact space $X$ ($((G, \tau_p), X, \theta)$ is a $G$-Tychonoff space), then 
$$\imath_X: G\to X^X,\ \imath_X (g)(x)=\theta (g, x),\ x\in X,$$
is a topological isomorphism of $G$ onto the subsemigroup of $X^X$. $e_X G=\cl (\imath_X (G))$ is an Ellis compactification of $G$~\cite{KozlovSorin2025} (see~\cite{Ellis} for non-proper compactifications of $G$). 

Let a topological group $G$ be $\tau_p$-representable in compact spaces $X$ and $Y$ ($(G, X, \theta_X)$ and $(G, Y, \theta_Y)$ are correspondent compact $G$-spaces), and $(e_X G, \imath_X)$, $(e_Y G, \imath_Y)$ are correspondent Ellis compactifications of $G$. If $f: X\to Y$ is an onto (perfect) $G$-map, then the $G$-map $F: e_X G\to e_Y G$  of Ellis compactifications of $G$ (induced by $f$) is correctly defined and is a homomorphism of monoids~\cite[Proposition 3.6]{KozlovSorin2025}. 

If a topological group $G$ is $\tau_p$-representable in a space $X$, then by Lemma~\ref{l1} the morphism of posets   
$$\mathfrak{E}_{\theta}:\mathbb{GC}(X)\to \mathbb{E} (G),\ \mathfrak{E}_{\theta}(b X)=e_{b X} G,$$ is defined~\cite[Corollary 3.10]{KozlovSorin2025}.  

\begin{rem}~\cite[Remark 3.9]{KozlovSorin2025}\label{compmapsth}
{\rm A topological group $G$ is $\tau_p$-representable in itself. Therefore,  the map  
$\mathfrak{E}:\mathbb{GC}(G)\to \mathbb{E} (G)$, $\mathfrak{E} (b G)=e_{b G} G$, is defined.  $b G\in\mathbb{E}(G)$ iff $\mathfrak{E}(b G)=b G$~\cite[Theorem 2.22]{KozlovSorin}.

In particular, if $G$ is $\tau_p$-representable in a space $X$, then $\mathfrak{E}\circ \mathfrak{E}_{\theta}=\mathfrak{E}_{\theta}$, since $\mathfrak{E}_{\theta}( bX)\in\mathbb{E} (G)$, $\forall\ b X\in\mathbb {GC}(X)$.}
\end{rem}

\begin{pro}~\cite[Proposition 3.8]{KozlovSorin2025}\label{Reesquot}
Let $f: X\to Y$ be an {\rm(}onto{\rm)} elementary $G$-map {\rm(}$f|_{X\setminus f^{-1}(y)}$, $y\in Y$, is a homeomorphism{\rm)} of compacta. 

If $I=\{h\in e_X G\ |\ h(x)\in f^{-1}(y),\ x\in X\}\ne\emptyset$, then $I$ is a compact ideal of $e_X G$. 

If, additionally, the condition 
$$\forall\ f, h\in e_X G\ \mbox{such that}\ f|_{X\setminus f^{-1}(y)}=h|_{X\setminus f^{-1}(y)}\ne\emptyset\ \Longrightarrow\ f=h\eqno{\rm (EF)}$$
is valid, then the Rees quotient $e_X G/I$ is topologically isomorphic to $e_Y G$.
\end{pro}

Let $G$ be $\tau_p$-representable in a locally compact space $X$, $b X\in\mathbb{GC}(X)$. Since the Alexandroff one-point compactification $\alpha X=X\cup\{\infty\}\in\mathbb{GC}(X)$, the $G$-map $f: b X\to\alpha X$ of $G$-compactifications of $X$ is defined and it is elementary. 

\begin{cor}\label{Reesquot1}
If $I=\{h\in e_{b X} G\ |\ h(x)\in f^{-1}(\infty),\ x\in b X\}\ne\emptyset$, then $I$ is a compact ideal of $e_{b X} G$. If, additionally, the condition {\rm(EF)} is valid, then the Rees quotient $e_{b X} G/I$ is topologically isomorphic to $e_{\alpha X} G$.
\end{cor}

\begin{rem}
{\rm The condition $I=\{h\in e_{b X} G\ |\ h(x)\in f^{-1}(\infty),\ x\in b X\}\ne\emptyset$ in Corollary~\ref{Reesquot1} is equivalent to the condition that the constant map ${\rm const}_{\infty}: \alpha X\to\{\infty\}$ is in $e_{\alpha X} G$.

\medskip

 ${\rm const}_{\infty}\in e_{\alpha X} G$ iff $\forall$ nbd $O_{\infty}$ of $\{\infty\}$ in $\alpha X$ and $\forall\ \sigma\in\Sigma_X$ $\exists\ g\in G$ such that $g(\sigma)\subset O_{\infty}$.}
\end{rem}

\medskip


\begin{df}~\cite[Definition 3.10, Definition 3.17]{KozlovLeiderman2025}
Let $G$ be a topological group.  A $G$-compactification $b G$ of $G$ is a sm-compactification of $G$ if $(bG, \bullet)$ is a semitopological monoid and $g\bullet h=gh$, $g, h\in G$.

A sm-compactification $b G$ of $G$ is a sm$^*$-compactification of $G$ if $(bG, \bullet)$ is a semitopological monoid with continuous involution $s$ and $s|_G=*$.
\end{df}

Evidently, every sm-compactification of $G$ is an  Ellis compactifications of $G$. The set of sm-compactifications of $G$ may be empty. 

\begin{lem}\label{hersm}
Let $H$ be a subgroup of a topological group $G$. If $b G$ is a sm-compactification of $G$, then $\cl\, H$ {\rm(}in $b G${\rm)} is a sm-compactification of $H$.
\end{lem}

\begin{proof}
$\cl\, H$ is a semigroup (the resonings as in~\cite{Ellis}) and the restrictions of the right (left) multiplications to $\cl\, H\times\cl\, H$ are continuous. 
\end{proof}

\begin{cor}\label{heredsm}
Let $H$ be a subgroup of a topological group $G$. 
\begin{itemize}
\item $(\exists$ a sm-compactification of $G)$ $\Longrightarrow$ $(\exists$ a sm-compactification of $H)$.
\item  $(\not\exists$ a sm-compactification of $H)$ $\Longrightarrow$ $(\not\exists$ a sm-compactification of $G)$.
\item If $H$ is a dense subgroup of $G$, then the {\rm WAP}-compactications {\rm(}not necessary proper{\rm)} of $G$ and $H$ are the same.
\end{itemize}
\end{cor}

\begin{proof}
The first two statements immediately follow from Lemma~\ref{hersm} and are noted in~\cite[\S\ 6]{Megr2001}.

The last statement follows from the equality $b_r H=b_r G$~\cite[Proposition 3.24]{RD} and the inequality WAP-compactification $\leq$  Roelcke compactification, see, for example,~\cite{GlasnerMegr2008}. 
\end{proof}

\begin{pro}\label{smcompimage}
Let $b G$ be a sm-compactification and $e G$ be an Ellis compactification of a topological group $G$. If $e G\leq b G$, then $e G$ is a sm-compactification of $G$.
\end{pro}

\begin{proof}
Let $f: b G\to e G$ be the map of compactifications  (homomorphism of semigroups~\cite[\S\ 3.3]{KozlovLeiderman2025}). For $x, y\in e G$ let $O$ be a nbd of $xy$. Take $x'\in b G$ such that $f (x')=x$. $\forall\ y'\in f^{-1} (y)$ $\exists$ a nbd $U_{y'}$ of $y'$ such that $x'U_{y'}\subset f^{-1} O$. Therefore, $\exists$ an open nbd $U$ of $f^{-1} (y)$ such that $x'U\subset f^{-1} O$. The map $f$ is closed and $f^{-1} (y)\subset U$, therefore, $\exists$ a nbd $V$ of $y$ such that $f^{-1} V\subset U$. From the following equalities and inclusions 
$$xV\subset x f(U)=f(x')f(U)=f(x' U)\subset f(f^{-1} O)\subset O$$
continuity on the left of multiplication in $e G$ follows. 
\end{proof}

\begin{lem}\label{comptauprep}
Let a topological group $(H, \tau_H)$ be $\tau_p$-representable in $X$, a topological group $(G, \tau_G)$ be $\tau_p$-representable in $Y$, and let $i: X\to Y$ be an embedding, $\varphi: H\to G$ be a monomorphism such that the following diagramm commutes
$$\begin{array}{ccc}
H\times X & \stackrel{\varphi\times i}{\rightarrow}     & G\times Y \\
\theta_X\downarrow   &  & \downarrow \theta_Y \\
X & \stackrel{i}{\hookrightarrow}  & Y, \\
\end{array}\eqno{(e)}$$
where $\theta_X$ and $\theta_Y$ are actions unduced by representations.

If $G\, i(X)=Y$, then $\varphi$ is a topological embedding. 
\end{lem}

\begin{proof}
For $x\in X$ and its nbd $O$ in $X$ let $O'$ be a nbd of $x$ in $Y$ such that $O'\cap i(X)=i(O)$. Then $[x, O]=[i(x), i(O)]\cap \varphi (H)\in N_H (e)$.  
The equality $[i(x), O']\cap\varphi (H)=[x, O]$ implies that $\tau_H\leq\tau_G|_{H=\varphi (H)}$. 

The condition $G\, i(X)=Y$ implies that $\forall\ y\in Y$ $\exists\ x\in i(X)$ and $\exists\ g\in G$ such that $g(x)=y$. Take $[y, O']\in N_G (e)$, $O=O'\cap g(i(X))$. For $W=[x=g^{-1}(y), g^{-1}(O)]\in N_H(e)$ 
$$g^{-1}( [y, O']\cap \varphi (H))g=W.$$ Indeed, if $h\in [y, O']\cap \varphi (H)$, then $g^{-1}h g(x)\in W$. If $h\in W$, then $h(g^{-1}(y))\in g^{-1}(O)$, and $g(h(g^{-1}(y)))=(ghg^{-1})(y)\in O$. Therefore, $ghg^{-1}\in [y, O']\cap \varphi (H)$. Hence, $\tau_H=\tau_G|_{H=\varphi (H)}$ and $\varphi$ is a topological embedding. 
\end{proof}

\begin{thm}\label{compRoelcke}
Let a topological group $(H, \tau_H)$ be $\tau_p$-representable in $X$, a topological group $(G, \tau_G)$ be $\tau_p$-representable in a compact space $Y$. If   $i: X\to Y$ is an embedding, $\varphi: H\to G$ is a monomorphism such that the diagramm {\rm(}e{\rm)} commutes and $G\, i(X)=Y$, then 
$e_{b X}H\leq e_Y H$, where $b X=cl_{Y} X$. Additionally,
\begin{itemize}
\item if $e_Y G$ is a sm-compactification of $G$, then $e_{b X}H$ is a sm-compactification of $H$,
\item if $e_Y G$ is a sm$^*$-compactification of $G$, then $e_{Y}H$ is a sm$^*$-compactification of $H$.
\end{itemize}
\end{thm}

\begin{proof}
By Lemma~\ref{comptauprep} $H$ is  $\tau_p$-representable in $Y$ and $e_Y H$ is the closure of $H$ in $e_Y G$ (follows from the construction of Ellis compactification). 

$b X$ is an invariant subset of $Y$ ($(H, Y, \theta_Y|_{H\times Y})$ is a $G$-space). Therefore, firstly, $e_{b X}H$ (the closure of $H$ in the product $b X^{b X}$) is identified with the closure of $H$ in $Y^{b X}$ (the product $b X^{b X}$ is a subproduct of  $Y^{b X}$  and a compact set). Secondly, the restriction of the projection ${\rm pr}: Y^Y\to Y^{b X}$ to $H$ is a homeomorphism. Hence, the map of compactifications $ e_Y H$ and $e_{b X}H$ is defined (the composition of the restriction of projection ${\rm pr}$ to $e_Y G$ and identification map). 

If $e_Y G$ is a sm-compactification of $G$, then $e_Y H$ is a sm-compactification of $H$ by Lemma~\ref{hersm} and $e_{b X}H$ is a sm-compactification of $H$ by Proposition~\ref{smcompimage}.

$e_Y H$ is the closure of $H$ in $e_Y G$. Since involution on $e_Y G$ is continuous and $H$ is invariant with respect to involution, $e_Y H$ is a sm$^*$-compactification of $G$. 
\end{proof}


\section{Graph compactifications of a topological group}\label{binrel} 

\subsection{Hyperspace of (closed) binary relations}

A {\it relation on $X$} is a subset of $X\times X$.  If $R, S$ are relations on $X$, then the {\it composition} of $R$ and $S$ (from the right) is the relation $RS=\{(x,y)\ |\ \exists z\ ((x,z)\in S\ \&\ (z,y)\in R)\}$. The symmetry $s_{X\times X}(x, y)=(y, x)$ on $X\times X$ is a self-inverse map and induces an involution on  relations on $X$
$$R\to s_{X\times X}(R)=\{(x, y)\in X\times X\ |\ (y, x)\in R\}.\eqno{({\rm Inv})}$$ 
The set of all relations on $X$ is a monoid (identity is the relation $\{(x, x)\ |\ x\in X\}$) with involution~\cite{CliffPrest}. 

A {\it closed relation} on a topological space $X$ is a closed subset of $X\times X$ and the set of all closed relations (including empty set) can be identified with $2^{X\times X}$. 

\begin{pro}\label{semrel}
$S_X=s^H_{X\times X}$ is a self-inverse map of $(2^{X\times X}, \tau_V)$ or $(2^{X\times X}, \tau_F)$.

If $X$ is a discrete or a compact space, then $(2^{X\times X}, \tau_F)$ {\rm(}$\tau_F=\tau_V$ if $X$ is compact{\rm)} is a compactum and monoid with {\rm(}continuous{\rm)} involution $S_X$.
\end{pro}

\begin{proof}
Since $S_X$ is induced by the symmetry $s_{X\times X}$ on $X\times X$ which is a homeomorphism, the first statement follows from  \S~\ref{hyper}.

If $X$ is discrete, then $2^{X\times X}$ is the set of all subsets of $X\times X$ and composition is correctly defined (see~\cite{CliffPrest} or~\cite[Ch.\, 8, \S\ 5]{Birkhoff},  where composition is taken from the left). 

If $X$ is compact and $R, S\in 2^{X\times X}$, then  $RS\in 2^{X\times X}$ (see~\cite{usp2001}). 
\end{proof}

\begin{rem}
{\rm (a) In general, if $X$ is a locally compact space, then the composition of closed relations may not be defined. Indeed, take $X=(0, 1)$ and closed relations $S=\{(x, -\frac12+2x)\subset X^2\ |\ \frac14< x<\frac34\}$, $R=\{(x, \frac12)\subset X^2\ |\ x\in X\}$ on $(0, 1)$. Then $RS=\{(x, \frac12)\subset X^2\ |\ \frac14< x<\frac34\}$ is not a closed relation on $(0, 1)$. 

\medskip

(b) In general, if $X$ is compactum the composition $(R, S)\to RS$ in $(2^{X\times X}, \tau_V)$ is not continuous on the left (and, hence, on the right). Indeed, take $X=[0, 1]$, and closed relations $R=([0, \frac12]\times\{0\})\cup ([\frac12, 1]\times\{1\})\subset X^2$, $S=\{0\}\times [0, \frac12]\subset X^2$ on $[0, 1]$. Then $RS=\{(0, 0), (0, 1)\}\subset X^2$. 

Fix $R$ and take a nbd  $W=[V_1, V_2]$ of $RS$, where $V_1=X\times [0, \frac12)$, $V_2=X\times (\frac12, 1]$. Any nbd of $S$ contains the set $S'=\{0\}\times [0, a]$ for some $a<\frac12$ and $RS'=\{(0, 0)\}$. Evidently, $RS'\not\in W$ and composition is not continuous on the left. 

\medskip

The same is true for $(2^{X\times X}, \tau_F)$ when $X$ is a discrete space. Ineed, let $X$ be discrete, $x_0\in X$, $R=\{(x, y)\ |\ x, y\ne x_0\}$, $S=\{(x_0, x_0)\}$ be closed relations on $X$. $RS=\{\emptyset\}$. 

Fix $R$ and take a nbd $U=\{F\in 2^{X\times X}\ |\ F\cap K=\emptyset\}$ of $\emptyset$ in $2^{X\times X}$, where $K=\{(x_1, y_1)\}$ is a one point set, $y_1\ne x_0$. An arbitrary nbd $W$ of $S$ may be chosen of the form $[V]\cap\{F\in 2^{X\times X}\ |\ F\cap T=\emptyset\}$, where $V=S$, $T\subset X\times X$ is a finite set. There exists $y_0\ne x_0\in X$ such that $(x_1, y_0)\not\in T$. Therefore, $S'=\{(x_0, x_0), (x_1, y_0)\}\in W$. Since $(y_0, y_1)\in R$, $(x_1, y_1)\in RS'$ and $RS'\not\subset U$, the composition is not continuous on the left.

\medskip

(c) For compacta the semigroup of closed relations with involution is discussed in~\cite{usp2001}.

\medskip

(d) Topologies on $2^{X\times X}$ where $X$ is a discrete space, are discussed in~\cite[\S\ 5.2]{EJMMMP}.}
\end{rem} 


\subsection{Actions on a hyperspace of binary relations}\label{actionhyperspace}
\begin{pro}\label{actionXX}
Let $(G, X, \theta)$ be a $G$-space. The actions 
$$\theta_{\uparrow}: G\times (X\times X)\to X\times X,\ \theta_{\uparrow} (g, (x, y))=(x, \theta(g, y))\ \mbox{and}$$ 
$$\theta_{\rightarrow}: G\times (X\times X)\to X\times X,\ \theta_{\rightarrow} (g, (x, y))=(\theta (g, x), y),\ g\in G,\ (x, y)\in X\times X,$$
on $X\times X$ are correctly defined and $(G, X\times X, \theta_{\uparrow})$, $(G, X\times X, \theta_{\rightarrow})$ are $G$-spaces. 

$\theta_{\rightarrow}=\theta_{\uparrow}^*$ and $\theta_{\uparrow}$ commutes with inverse action. 
\end{pro}

\begin{proof}
It is easy to verify that the actions are correctly defined. Their continuity (in product topology) is evident. 
$$\theta_{\uparrow}^*(g, (x, y))=s_{X\times X}(\theta_{\uparrow}(g, s_{X\times X}((x, y)))=s_{X\times X}(\theta_{\uparrow}(g, (y, x))=s_{X\times X}((y, \theta (g, x)))=$$
$$=(\theta (g, x), y)=\theta_{\rightarrow}(x, y),\ g\in G,\ (x, y)\in X\times X.\ \mbox{Hence},\ \theta_{\rightarrow}=\theta_{\uparrow}^*.$$
$$\theta_{\uparrow}(g, \theta_{\rightarrow}(h, (x, y)))=\theta_{\uparrow}(g, (\theta (h, x),  y))=(\theta (h, x), \theta (g, y))=$$
$$=\theta_{\rightarrow}(h, (x, \theta (g, y)))=\theta_{\rightarrow}(h, \theta_{\uparrow}(g, (x, y))).\ \mbox{Hence},\ \theta_{\uparrow}\ \mbox{commutes with inverse action}.$$ 
\end{proof}

From Propositions~\ref{actionXX} and~\ref{actionhyper}, \ref{hyperinversemap} it follows. 

\begin{cor}\label{actionXXYY}
Let  $(G, X, \theta)$ be a $G$-space and $X$ be locally compact. Then the induced actions 
$$\theta_{\uparrow}^H  (g, R)=\{\theta_{\uparrow}(g, (x, y))=(x, \theta (g, y))\ |\ (x, y)\in R\}\ \mbox{and}\eqno{(\theta_{\uparrow}^H)}$$ 
$$\theta_{\rightarrow}^H (g, R)=\{\theta_{\rightarrow}(g, (x, y))=(\theta (g, x), y)\ |\ (x, y)\in R\},\ R\in 2^{X\times X}, \eqno{(\theta_{\rightarrow}^H)}$$ 
are continuous, $(\theta_{\uparrow}^H)^*=\theta_{\rightarrow}^H$ and $\theta_{\uparrow}^H$ commutes with inverse action.
\end{cor}

\begin{rem}
{\rm For metrizable compacta $X$ the continuity of the actions $\theta_{\uparrow}^H$, $\theta_{\rightarrow}^H$ and the self-inverse map $S_X$ is proved in~\cite{Kennedy}.}
\end{rem} 


\subsection{Maps of hyperspaces of binary relations}\label{maphyperspace}

\begin{pro}\label{mapshyper} 
 Let $Y$, $X$ be locally compact spaces and the map $f: Y\to X$ be perfect and onto. 
\begin{itemize}
\item[{\rm (A)}]The map 
$$F^H: (2^{Y\times Y}, \tau_F)\to (2^{X\times X}, \tau_F),\  F^H(R)=(F=f\times f) (R),\  R\in 2^{Y\times Y},$$ 
is onto, perfect and commutes with self-maps $S_Y$ and $S_X$.
\item[{\rm (B)}] If $(G, Y, \theta_Y)$, $(G, Y, \theta_X)$ are $G$-spaces and $f: Y\to X$ is a $G$-map, then $F^H$ is a $G$-map of compact $G$-spaces $(G, (2^{Y\times Y}, \tau_F), \theta_{Y\uparrow}^{H})$, $(G, (2^{X\times X}, \tau_F), \theta_{X\uparrow}^{H})$ and $(G, (2^{Y\times Y}, \tau_F), \theta_{Y\rightarrow}^{H})$, $(G, (2^{X\times X}, \tau_F), \theta_{X\rightarrow}^{H})$.
\end{itemize}
\end{pro}

\begin{proof}
(A) and (B) follow from Proposition~\ref{maphyperF} and Corollary~\ref{actionXXYY}. 
\end{proof}

\begin{rem}\label{erphi}
{\rm The map $F^H: (2^{Y\times Y}, \tau_F)\to (2^{X\times X}, \tau_F)$ defines an equivalence relation $\sim_{F^H}$ on $2^{Y\times Y}$ 
$$R\sim_{F^H} S\ \Longleftrightarrow\ F^H (R)=F^H (S),\ R, S\in 2^{Y\times Y}.$$
If $x\sim_f y\ (\Longleftrightarrow\ f(x)=f(y),\ x, y\in Y$, and $[x]_f$ is the equivalence class of $x\in Y$), then 
$$R\sim_{F^H} S\ \Longleftrightarrow\ R\subset\bigcup\{[x]_f\times [y]_f\ |\ (x, y)\in S\}\ \mbox{and}\ S\subset\bigcup\{[x]_f\times [y]_f\ |\ (x, y)\in R\},$$
$[R]_{F^H}$ is the equivalence class of $R\in 2^{Y\times Y}$.

Since $F^H$ is the map of compacta, $2^{Y\times Y}/\sim_{F^H}$ is homeomorphic to $2^{X\times X}$.}
\end{rem}

From Proposition~\ref{subsethyperF} it follows. 

\begin{pro}\label{restrhyper} 
Let $Y$ be a locally compact space and $X$ be an open subset of $Y$. 
\begin{itemize}
\item[{\rm (A)}] The map 
$$H_{\dashv}: (2^{Y\times Y}, \tau_F)\to (2^{X\times X}, \tau_F),\  H_{\dashv} (R)=R\cap (X\times X),\  R\in 2^{Y\times Y},$$ 
is onto, perfect and commutes with self-inverse maps $S_Y$ and $S_X$  {\rm(}$H_{\dashv}^{-1}(\emptyset)=\{A\in 2^{Y\times Y}\ |\ A\subset ((Y\setminus X)\times Y)\cup (Y\times (Y\setminus X))${\rm)}. 
\item[{\rm (B)}] If $(G, Y, \theta_Y)$ is a $G$-space and $X$ is an invariant subset {\rm(}$(G, X, \theta_X=\theta_Y|_{G\times X})$ is a $G$-space{\rm)}, then $H_{\dashv}$ is a $G$-map of compact $G$-spaces $(G, (2^{Y\times Y}, \tau_F), \theta_{Y\uparrow}^{H})$, $(G, (2^{X\times X}, \tau_F), \theta_{X\uparrow}^{H})$ and $(G, (2^{Y\times Y}, \tau_F), \theta_{Y\rightarrow}^{H})$, $(G, (2^{X\times X}, \tau_F), \theta_{X\rightarrow}^{H})$.
\end{itemize}
\end{pro}

\begin{ex}
{\rm If in Proposition~\ref{restrhyper} $X$ is a disctere space, then the map $H_{\dashv}$ may not be a homomorphism of semigroups.

Take the Alexandroff one-point compactification $Y=\alpha X=X\cup\{\infty\}$ of (infinite) $X$, $R=\{\infty\}\times\alpha X$, $S=\alpha X\times\{\infty\}$. Then 
$$H_{\dashv} (RS)=H_{\dashv} (\alpha X\times\alpha X)=X\times X\ne\{\emptyset\}=H_{\dashv} (R) H_{\dashv}(S).$$}
\end{ex}

\begin{rem}\label{commmap}
{\rm If $Y, Z$ are locally compact spaces, $X$ is an open subset of $Y$ and $Z$ and $f: Z\to Y$ is a perfect and onto map such that $f|_{X}=\id$ and $f^{-1}(f(X))=X$, then $$H^Y_{\dashv}\circ F^H=H^Z_{\dashv},\ \mbox{where}\ H^Y_{\dashv}:  (2^{Y\times Y}, \tau_F)\to (2^{X\times X}, \tau_F), H^Z_{\dashv}:  (2^{Z\times Z}, \tau_F)\to (2^{X\times X}, \tau_F).$$}
\end{rem}


\subsection{Embedding of a group in a hyperspace of binary relations}\label{embedhyperspace} 

Let $G$ be a subgroup of the permutation group ${\rm S}(X)$ of a set $X$. The injective map 
$$\i_X^{\Gamma}: G\to 2^{X\times X},\ \i_X^{\Gamma}(g)=\{(x, g(x))=(x, \theta (g, x))\ |\ x\in X\}\ (\mbox{graph of}\ g),\ g\in G,$$ 
is defined, where $\theta: G\times X\to X$ is the natural action. On $\i_X^{\Gamma}(G)$ multiplication is the restriction of composition on $2^{X\times X}$, involution is the restriction of the self-inverse map $S_X$.  

\begin{pro}\label{homomF}
Let a topological group $G$ be $\tau_g$-representable in a locally compact space $X$. Then the map $\i_X^{\Gamma}: G\to (2^{X\times X}, \tau_F)$ is
\begin{itemize}
\item  a $G$-map of $G$-spaces $(G, G, \alpha)$, $(G, (2^{X\times X}, \tau_F), \theta^H_{\uparrow})$  and $(G, G, \alpha^*)$, $(G, (2^{X\times X}, \tau_F), \theta_{X\rightarrow}^{H})$ and commutes with self-inverse maps $*$, $S_X$,
\item isomorphism of $G$ and $\i_X^{\Gamma}(G)$,
\item  uniformly continuous with respect to the Roelcke uniformity $L\wedge R$ on $G$ and the unique uniformity on compactum $2^{X\times X}$ and, hence, a continuous isomorphism of $\i_Y^{\Gamma}(G)$ and $\i_X^{\Gamma}(G)$. 
\end{itemize}
\end{pro}

\begin{proof} 
Let us verify that $\i_X^{\Gamma}$ is a $G$-map (not necessary continuous). 
$$\i_X^{\Gamma} (\alpha (h, g))=\i_X^{\Gamma} (hg)=\{(x, (hg)(x))\ |\ x\in X\}=\{(x, h(g(x))\ |\ x\in X\}\stackrel{(\theta^H_{\uparrow})}{=}$$
$$=\{(x, \theta (h, g(x))\ |\ x\in X\}=\theta^H_{\uparrow} (h, \Gamma (g))=\theta^H_{\uparrow} (h, \i_X^{\Gamma} (g))\ g, h\in G.$$

Let us verify that $\i_X^{\Gamma}$ commutes with self-inverse maps $*$ and $S_X$.
$$\i_X^{\Gamma} (g^{-1})=\{(x, g^{-1}(x))\ |\ x\in X\}\stackrel{y=g^{-1}(x)}{=}\{(g(y), y)\ |\ y\in X\}=$$
$$=S_X(\{(y, g(y))\ |\ y\in X\})=S_X(\i_X^{\Gamma} (g)),\ g\in G.$$
The appliction of Lemma~\ref{ginversg}, Proposition~\ref{actionXX} and definition of inverse (left) action in \S~\ref{topgroup} finishes the proof of the first statement. 

Further, 
$$\i_X^{\Gamma} (fg)=\{(x, (fg)(x))\ |\ x\in X\}=\{(x, f(g(x))\ |\ x\in X\}=$$
$$=\{(x, f(x))\ |\ x\in X\}\{(x, g(x))\ |\ x\in X\}=\i_X^{\Gamma} (f)\i_X^{\Gamma} (g),\ f, g\in G,$$
$\i_X^{\Gamma}(e)$ is the unit. Hence, $\i_X^{\Gamma}(G)$ is a group and $\i_X^{\Gamma}$ is an isomorphism of $G$ onto $\i_X^{\Gamma}(G)$.

\medskip

$\i_X^{\Gamma}$ is uniformly continuous with respect to the Roelcke uniformity $L\wedge R$ on $G$ and the unique uniformity on compactum $2^{X\times X}$ by Corollary~\ref{cor1}. Hence, $\i_X^{\Gamma}$ is a continuous isomorphism of groups $G$ and $\i_X^{\Gamma}(G)$.
\end{proof}

\begin{lem}\label{lemcompemb}~\cite{usp2001} 
If a topological group $G$ is $\tau_g$-representable in a compact space $X$, then $\i^{\Gamma}_X$ is a topological isomorphism of $G$ and $\i_X^{\Gamma}(G)$.
\end{lem}

\begin{proof} 
By Proposition~\ref{homomF} the map $\i^{\Gamma}_X$ is a continuous isomorphism of $G$ and $\i_X^{\Gamma}(G)$. 

Let $\mathcal U$ be the unique uniformity on compacta $X$, the Cartesian product $\mathcal U^2$ of the uniformity $\mathcal U$ is the unique uniformity on compacta $X\times X$. The base of $\mathcal U^2$ is formed by the entourages of the diagonal 
$${\rm U}^2=\{((x, y), (x', y'))\ |\ (x, x')\in{\rm U},  (y, y')\in {\rm U}\},\ {\rm U}\in\mathcal U.$$

The compact-open topology on $G$ coincides with the topology induced by the uniformity of uniform  convergence~\cite[Ch.\ 8, \S\ 8.2, Corollary 8.2.7]{Engelking}, which base consists of the entourages of the diagonal 
$$\hat {\rm U}=\{(f, g)\in G\times G\ |\ \forall x\in X\ ((f(x), g(x))\in {\rm U})\},\ {\rm U}\in\mathcal U.$$

The base of the unique uniformity $2^{\mathcal U^2}$ on compacta $(\CL (X\times X), \tau_V)$ (we consider $(\CL (X\times X), \tau_V)$ since $\{\emptyset\}$ is an isolated point in $2^{X\times X}$) is formed by the entourages of the diagonal 
$$\CL^{{\rm U}^2}=\{(F, T)\in \CL (X\times X)\times \CL (X\times X)\ |\ F\subset {\rm B}(T, {\rm U}^2),\  T\subset {\rm B}(F, {\rm U}^2)\},\ {\rm U}\in\mathcal U.$$

An arbitrary $h\in G$ is a uniformly continuous map of $X$. Hence, $\forall\ {\rm U}\in\mathcal U$ $\exists\ {\rm V}\in\mathcal U$, ${\rm V}\subset {\rm U}$, such that if $(x, x')\in {\rm V}$, then $(h(x), h(x'))\in {\rm U}$. If $(\i^{\Gamma}_X (h), \i^{\Gamma}_X (g))\in\CL^{{\rm V}^2}$, then $\forall\ x\in X$ $\exists\ x'\in X$ such that $(x, x')\in {\rm V}$ and $((x', h(x')), (x, g(x))\in {\rm V}^2$ (equivalently $(x, x')\in {\rm V}$, $(h(x'), g(x))\in {\rm V}$ and $(h(x), h(x'))\in {\rm U}$). Therefore, $(h(x), g(x))\in 2{\rm U}$, $x\in X$, the restriction $(\i^{\Gamma}_X)^{-1}|_{\i^{\Gamma}_X (G)}$ is continuous at the point $i^{\Gamma}_X (h)$ by Lemma~\ref{cont} and $\i^{\Gamma}_X$ is a topological isomorphism of $G$ and $\i_X^{\Gamma}(G)$.
\end{proof}

\begin{rem}
{\rm  For a compact metrizable space $X$ and a $\tau_g$-representable in $X$ group $G$, Lemma~\ref{lemcompemb} is, in fact, proved in~\cite{Kennedy}. 

In~\cite{usp2001} the uniform continuity of $\i^{\Gamma}_X$ with respect to the Roelcke uniformity $L\wedge R$ on $G$ and the unique uniformity on compactum $2^{X\times X}$ in Lemma~\ref{lemcompemb} is proved.}
\end{rem} 

\begin{thm}\label{propPreimage} 
Let a topological group $G$ be $\tau_g$-representable in locally compact spaces $X$ and $Y$ and $\i_X^{\Gamma}$ is a topological isomorphism of $G$ and $\i_X^{\Gamma}(G)$. If 
\begin{itemize}
\item[{\rm(a)}] $f: Y\to X$ is a perfect onto $G$-map of $G$-spaces $(G, Y,  \curvearrowright)$ and $(G, X,  \curvearrowright)$, or  
\item[{\rm(b)}] an embedding $j: X\to Y$ is a $G$-map of $G$-spaces $(G, X,  \curvearrowright)$ and $(G, Y,  \curvearrowright)$,
\end{itemize}
then 
$\i_Y^{\Gamma}$ is a topological isomorphism of $G$ and $\i_Y^{\Gamma}(G)$ and $F^H|_{\i_Y^{\Gamma} (G)}$ {\rm(}respectively $H_{\dashv}|_{\i_Y^{\Gamma} (G)}${\rm)} is a topological isomorphism of $\i_Y^{\Gamma}(G)$ and $\i_X^{\Gamma}(G)$ {\rm(}respectively $\i_Y^{\Gamma}(G)$ and $\i_{j(X)}^{\Gamma}(G)=\i_{X}^{\Gamma}(G)${\rm)}. 
\end{thm}

\begin{proof} 
$\i_Y^{\Gamma}$ is a continuous isomorphism of $G$ and $\i_Y^{\Gamma}(G)$ by Proposition~\ref{homomF}. Since $\i_X^{\Gamma}=F^H\circ \i_Y^{\Gamma}$ (respectively $\i_X^{\Gamma}=\i_{j(X)}^{\Gamma}=H_{\dashv}|_{\i_{Y}^{\Gamma} (G)}\circ \i_Y^{\Gamma}$ for $j(X)\subset Y$) and $\i_X^{\Gamma}$ is a topological isomorphism of $G$ and $\i_X^{\Gamma}(G)$, $F^H|_{\i_Y^{\Gamma} (G)}$  (respectively $H_{\dashv}|_{\i_{Y}^{\Gamma} (G)}$) is a topological isomorphism of $\i_Y^{\Gamma}(G)$ and $\i_X^{\Gamma}(G)$, and $\i_Y^{\Gamma}$ is a topological isomorphism of $G$ and $\i_Y^{\Gamma}(G)$.
\end{proof}


\subsection{Graph compactifications of topological groups and their maps}\label{bincomp}

\begin{df}~\cite[Definition 3.1]{KozlovLeiderman2025}
Let $G$ be a topological group.  A $G$-compactification $b G$ of $G$  with self-inverse map $s$ is a $G^*$-compactification of $G$ if $s|_{G}=*$. 
\end{df}

\begin{pro}~\cite[Proposition 3.4]{KozlovLeiderman2025}\label{G*Roelcke}
Let $G$ be a topological group and $b G$ be a $G^*$-compactification of $G$. Then $b G\leq b_r G$. 
\end{pro}

\begin{pro}~\cite[Proposition 3.15]{KozlovLeiderman2025}\label{sm-sm*}
An Ellis-compactification $b G$ of $G$ with self-inverse map $s$ such that $s|_G=*$  is a sm$^*$-compactification of $G$ iff $s$ is an involution on  $b G$.

A $G^*$-compactification $b G$ of $G$ is a sm-compactification $b G$ of $G$ iff $b G$ is a sm$^*$-compactification of $G$.
\end{pro}

\medskip

Since $\i^{\Gamma}_X (G)$ is an invariant subset of a $G$-space $(G, (2^{X\times X}, \tau_F), \theta^H_{\uparrow})$ and invariant with respect to the self-inverse map $S_X$ on $2^{X\times X}$ (and, hence, an invariant subset of a $G$-space $(G, (2^{X\times X}, \tau_F), \theta^H_{\rightarrow})$), its closure $b_X G=\cl (\i^{\Gamma}_X (G))$ in $(2^{X\times X}, \tau_F)$ is an invariant subset for the action $\theta^H_{\uparrow}$ and invariant with respect to the self-inverse map $S_X$ (and, hence, an invariant subset for the action $\theta^H_{\rightarrow}$). Therefore, $b_X G$ is a $G^*$-compactification of $G$ if $G$ is $\tau_p$-representable in a locally compact space $X$ and $\i^{\Gamma}_X$ is a topological embedding. 

\begin{df}\label{graphcomp}
Let a topological group $G$ be $\tau_g$-representable in a locally compact space $X$. If $\i^{\Gamma}_X$ is a topological embedding, then $b_X G$ will be called a graph compactification of $G$ {\rm(}with respect to a $\tau_g$-representation of $G$ in $X${\rm)}. 
\end{df}

\begin{que}
Is every $G^*$-compatification of a topological group $G$ a graph compactification of $G$ for its $\tau_g$-representation in a locally compact space $X$?
\end{que}

From Propositions~\ref{G*Roelcke} and~\ref{sm-sm*} it follows.

\begin{cor}\label{closure}
Let $G$ be a topological group and $b_X G$ be a graph compactification of $G$. Then $b_r G\geq b_X G$. 

$b_X G$ is a sm-compactification of $G$ iff $b_X G$ is a sm$^*$-compactification of $G$.
\end{cor}

From Proposition~\ref{mapshyper} (respectively Proposition~\ref{restrhyper}), Theorem~\ref{propPreimage}  and~\cite[Lemma 3.5.6]{Engelking} it follows.

\begin{thm}\label{mapcompV}  
Let a topological group $G$ be $\tau_g$-representable in locally compact spaces $Y$ and $X$,  $b_X G$ be a graph compactification of $G$. 
If 
\begin{itemize}
\item[{\rm(a)}] $f: Y\to X$ is a perfect onto $G$-map of $G$-spaces $(G, Y,  \curvearrowright)$ and $(G, X,  \curvearrowright)$, or  
\item[{\rm(b)}] an embedding $j: X\to Y$ is a $G$-map of $G$-spaces $(G, X,  \curvearrowright)$ and $(G, Y,  \curvearrowright)$, and $j(X)$ is open in $Y$, 
\end{itemize}
then $b_Y G$ is a graph compactification of $G$, $b_Y G\geq b_X G$, $F^H|_{b_Y G}: b_Y G\to b_X G$ {\rm(}respectively $H_{\dashv}|_{b_Y G}: b_Y G\to b_X G=b_{j(X)} G${\rm)} is a map of compactifications. 
\end{thm}

From Lemma~\ref{l11} and Theorem~\ref{mapcompV}  it follows.

\begin{cor}\label{orderembed}
Let a topological group $G$ be $\tau_g$-representable in $X$. 

If $(Y, \psi), (Z, \varphi)\in\mathbb{GLC}(X)$, $Z\leq_{GLC} Y$ and $b_Z G$ is a graph compactification of $G$, then  $b_Y G$ is a graph compactification of $G$ and $b_Z G\leq b_Y G$.

If $(Y, \psi)$, $(Z, \varphi)$ are equivalent locally compact $G$-extensions of $X$. Then $b_Y G$ is a graph compactification of $G$ iff $b_Z G$ is a graph compactification of $G$ and $b_Y G=b_Z G$ in this case.
\end{cor}

\begin{proof}
Take an open invariant set $\psi(X)\subset U\subset Y$ ($j: U\to Y$ denote the $G$-embedding) and a perfect onto $G$-map $f: U\to Z$ such that $\varphi=f\circ\psi$. By Theorem~\ref{mapcompV} (a) $b_U G$ is graph compactification of $G$ and $F^H: b_U G\to  b_Z G$ is a map of compactifications of $G$. By Theorem~\ref{mapcompV} (b) $b_Y G$ is graph compactification of $G$ and $H_{\dashv}|_{b_Y G}: b_Y G\to b_U G=b_{j(U)} G$ is a map of compactifications. The map $(F^H\circ H_{\dashv})|_{b_Y G}: b_Y G\to b_Z G$ is a map of compactifications of $G$. Hence, $b_Z G\leq b_Y G$.

The same considerations prove the second statement.
\end{proof}


\begin{rem}\label{latticefraphcomp}
{\rm Corollary~\ref{orderembed} shows that if a topological group $G$ is $\tau_g$-representable in $X$, then the poset $\mathbb{GLC}_{\Gamma}(X)\subset\mathbb{GLC}(X)$ of (equivalence classes) locally compact $G$-extensions of $X$ for which the corresponding graph compactifications of $G$ are defined, is an upper lattice. The poset of graph compactifications of $G$ correspondent to elements of $\mathbb{GLC}_{\Gamma}(X)$ is also an upper lattice and the graph compactification construction is an order preseving map of upper lattices.}
\end{rem}


\begin{thm}\label{mapcompVel}  
Let a topological group $G$ be $\tau_g$-representable in locally compact spaces $Y$ and $X$, $f: Y\to X$ be a perfect onto $G$-map and $b_Y G$ be a graph compactification of $G$. 

{\rm(A)} $b_X G$ is a graph compactification of $G$ and $b_Y G\geq b_X G$ iff 
the condition 
$$\forall\ g\in G,\ \forall\ R\in b_Y G\ (F^H (\i^{Y}_{\Gamma}(g))=F^H (R))\ \Longrightarrow\ (\i^{Y}_{\Gamma}(g)=R) \eqno{{\rm(PM)}}$$ 
is valid. 

\medskip

{\rm(B)} If, additionally, $f: Y\to X$ is an elementary $G$-map {\rm(}$K=f^{-1}(x)$, $|f^{-1}(x')|=1$, $x\ne x'${\rm)}, then  the set 
$$0_K=\{R\in  b_{Y} G\ |\ R\subset (Y\times K)\cup (K\times Y)\}$$
is invariant under the action $\theta^H_{Y\uparrow}$ and the self-inverse map $S_Y$ {\rm(}and, hence, is invariant under the action $\theta^H_{Y\rightarrow}${\rm)}; 

$b_X G$ is a graph compactification of $G$, $b_Y G\geq b_X G$ and the map $F^H|_ {b_Y G}: b_Y G\to b_X G$ of compactifications is elementary {\rm(}$F^H|_ {b_Y G\setminus 0_K}$ is a bijection{\rm)} iff $\forall\ R, S\in b_{Y} G$
$$\mbox{if}\ R\cap ((Y\setminus K)\times (Y\setminus K))=S\cap  ((Y\setminus K)\times (Y\setminus K))\ne\emptyset\ \mbox{and}\ F^H (R)=F^H (S)\Longrightarrow\ R=S,$$ 
$$\mbox{if}\ R\cap ((Y\setminus K)\times (Y\setminus K))=S\cap  ((Y\setminus K)\times (Y\setminus K))=\emptyset\ \Longrightarrow\ F^H (R)=F^H (S).\eqno{{\rm(EPM)}}$$
\end{thm}

\begin{proof}
(A) If $b_X G$ is a graph compactification of $G$ ($\i_{X}^{\Gamma}: G\to (2^{X\times X}, \tau^X_F)$ is an embedding), then $F^H|_{b_Y G}: b_Y G\to b_X G$ is a map of compactifications of $G$ and the condition (PM) is valid.

The map $\i_{X}^{\Gamma}$ is a $G$-map, commutes with involution on $G$ and self-inverse map $S_X$ on $2^{X\times X}$ and a bijection by Proposition~\ref{homomF}. The map $F^H: 2^{Y\times Y}\to  2^{X\times X}$ is a perfect $G$-map and commutes with self-inverse maps $S_Y$ and $S_X$ by Proposition~\ref{mapshyper}. The restriction $F^H|_{b_Y G}$ is a perfect $G$-map of compactification $b_Y G$ onto the closure of $\i_{X}^{\Gamma} (G)$ in $(2^{X\times X}, \tau^{X}_F)$,  commutes with the correspondent restrictions of self-inverse maps $S_Y$ and $S_X$ and by the condition (PM) is a bijection of $\i_{Y}^{\Gamma} (G)$ onto $\i_{X}^{\Gamma} (G)$,  $F^H|_{b_Y G} (b_Y G\setminus \i_{Y}^{\Gamma}(G))=b_X G\setminus \i_{X}^{\Gamma}(G)$. Therefore, $F^H|_{b_Y G}|_{\i_{Y}^{\Gamma}(G)}$ is a homeomorphism. Hence, $\i_{X}^{\Gamma}$ is an embedding, $b_X G$ is a graph compactification of $G$ and $b_Y G\geq b_X G$. 

\medskip

(B) If $f: Y\to X$ is an elementary $G$-map, then invariance of $0_K$ follows from the invariance of the set $K$ under the action $G\curvearrowright Y$ and the set $(Y\times K)\cup (K\times Y)\subset Y\times Y$ under the self-inverse map $s_Y$.

If $b_X G$ is a graph compactification of $G$ ($\i_{X}^{\Gamma}: G\to (2^{X\times X}, \tau^X_F)$ is an embedding) and the map $F^H|_ {b_Y G}: b_Y G\to b_X G$ of compactifications is an elementary $G$-map and commutes with the correspondent restrictions of self-inverse maps $S_Y$ and $S_X$, then only elements from $b_{Y} G$ which belong  $(Y\times K)\cup (K\times Y)$ if any, are identified to the one-point set, and the condition (EPM) is valid.

Since the condition (EPM) yields the condition (PM), only elementarity of $F^H|_ {b_Y G}$ must be checked. The map $F^H|_ {b_Y G}: b_Y G\to b_X G$ is a bijection on $b_Y G\setminus 0_K$ due to the condition (EPM). Hence, if $|0_K|\leq 1$, then $F^H|_ {b_Y G}$ is a bijection (compactifications $b_Y G$ and $b_X G$ are equivalent), if $|0_K|>1$, then $F^H|_ {b_Y G}$ is an an elementary map of compactifications of $G$. 
\end{proof}

\begin{thm}\label{restrmap}
Let a topological group $G$ be $\tau_g$-representable in locally compact spaces $Y$ and $X$, $j: X\to Y$ be a $G$-embedding of $G$-spaces $(G, X,  \curvearrowright)$ and $(G, Y,  \curvearrowright)$, $j(X)$ be open in $Y$,  and $b_Y G$ be a graph compactification of $G$.

{\rm(A)} The set 
$$0_{Y\setminus X}=\{R\in  b_{Y} G\ |\ R\subset ((Y\times (Y\setminus j(X)))\cup ((Y\setminus j(X))\times Y))\}$$
is invariant under the action $\theta^H_{Y\uparrow}$ and the self-inverse map $S_Y$ {\rm(}and, hence, is invariant under the action $\theta^H_{Y\rightarrow}${\rm)}; $b_X G$ is a graph compactification of $G$ and $b_Y G\geq b_X G$ iff the condition 
$$\forall\ g\in G,\ \forall\ R\in b_Y G\  (R\cap (j(X)\times j(X))=\i_{Y}^{\Gamma} (g)\cap (j(X)\times j(X)))\ \Longrightarrow\ (R=\i_{Y}^{\Gamma} (g))\eqno{\rm(R)}$$
is valid.

\medskip

{\rm(B)} $b_X G$ is a graph compactification of $G$ and the map $H_{\dashv}|_ {b_Y G}: b_Y G\to b_X G$ of compactifications is an elementary $G$-map and commutes with self-inverse maps $S_Y$ and $S_X$ iff  
$$\forall\ R, S\in b_{Y} G\ (R\cap (j(X)\times j(X))=S\cap (j(X)\times j(X))\ne\emptyset)\ \Longrightarrow\ (R=S).\eqno{\rm(ER)}$$ 
In this case $H_{\dashv}|_ {b_Y G}(0_{Y\setminus X})=\{\emptyset\}$ if $0_{Y\setminus X}\ne\emptyset$ and a homeomorphism if $|0_{Y\setminus X}|\leq 1$.
\end{thm}

\begin{proof}
(A) The invariance of $0_{Y\setminus X}$ follows from the invariance of the subset $Y\setminus j(X)$ of $Y$ and the set $(Y\times (Y\setminus j(X)))\cup ((Y\setminus j(X))\times Y))\subset Y\times Y$ under the self-inverse maps $s_Y$.

If $b_X G$ is a graph compactification of $G$ ($\i_{X}^{\Gamma}: G\to (2^{X\times X}, \tau^X_F)$ is an embedding), then $H_{\dashv}|_{b_Y G}: b_Y G\to b_X G=b_{j(X)} G$ is  the map of compactifications of $G$ and the condition (R) is valid.

The map $\i_{X}^{\Gamma}$ is a $G$-map, commutes with involution on $G$ and self-inverse map $S_X$, and a bijection by Proposition~\ref{homomF}. The map $H_{\dashv}: 2^{Y\times Y}\to  2^{X\times X}$ is a perfect $G$-map and commutes with self-inverse maps $S_Y$ and $S_X$ by Proposition~\ref{restrhyper}. The restriction $H_{\dashv}|_{b_Y G}$ is a perfect $G$-map of compactification $b_Y G$ onto the closure of $\i_{X}^{\Gamma} (G)$ in $(2^{X\times X}, \tau^{X}_F)$ and by the condition (R) is a bijection of $\i_{Y}^{\Gamma} (G)$ onto  $\i_{X}^{\Gamma} (G)$,  $H_{\dashv}|_{b_Y G} (b_Y G\setminus \i_{Y}^{\Gamma}(G))=b_X G\setminus \i_{X}^{\Gamma}(G)$. Therefore, $H_{\dashv}|_{b_Y G}|_{\i_{Y}^{\Gamma}(G)}$ is a homeomorphism. Hence, $\i_{X}^{\Gamma}: G\to (2^{X\times X}, \tau^X_F)$ is an embedding, $b_X G$ is a graph compactification of $G$, $H_{\dashv}|_{b_Y G}$ is a map of compactifications of $G$ and $b_Y G\geq b_X G$. 

\medskip

(B) If $b_X G$ is a graph compactification of $G$ ($\i_{X}^{\Gamma}: G\to (2^{X\times X}, \tau^X_F)$ is an embedding) and the map $H_{\dashv}|_ {b_Y G}: b_Y G\to b_X G=b_{j(X)} G$ of compactifications is an elementary $G$-map, then only sets elements $b_Y G$ which belong to $(Y\times (Y\setminus j(X)))\cup ((Y\setminus j(X))\times Y)$ if any, are identified to the one-point set and the condition (ER) is valid.

Since the condition (ER) yields the condition (R), only elementarity of $H_{\dashv}|_ {b_Y G}$ must be checked. The map $H_{\dashv}|_ {b_Y G}: b_Y G\to b_X G=b_{j(X)} G$ is a bijection on $b_Y G\setminus 0_{Y\setminus X}$ due to the condition (ER). Hence, if $|0_{Y\setminus X}|\leq 1$, then $H_{\dashv}|_ {b_Y G}$ is a bijection (compactifications $b_Y G$ and $b_X G$ are equivalent), if $|0_{Y\setminus X}|>1$, then $H_{\dashv}|_ {b_Y G}$ is an an elementary $G$-map of graph compactifications of $G$. 
\end{proof}

\begin{rem}\label{corrcond}
{\rm The conditions (PM) and (R) in Theorems~\ref{mapcompVel} and~\ref{restrmap} respectively, can be examined as a property of the unique extentions of graphs of homeomorphisms in the closure of $\i_{Y}^{\Gamma} (G)$ in the hyperspace $2^{Y\times Y}$.

\medskip

If a topological group $G$ is $\tau_g$-representable in compact spaces $Y$ and $X$ and $f: Y\to X$ is a perfect elementary onto $G$-map then (taking into account  Lemma~\ref{proj} below) the condition (EPM) is equivalent to the condition 
$$\mbox{if}\ R\cap ((Y\setminus K)\times (Y\setminus K))=S\cap  ((Y\setminus K)\times (Y\setminus K))\ne\emptyset\ \mbox{and}\ F^H (R)=F^H (S)\Longrightarrow\ R=S.$$ 

\medskip

Evidently, the condition (R) imlies the condition (PM) for a locally compact space $Y$, $X=Y\setminus K$, $K$ is compact and $f: Y\to X/K$ is an elementary map.

Let a topological group $G$ be $\tau_g$-representable in a locally compact space $Y$, $b_Y G$ be a graph compactification of $G$, $K$ is a compact invariant subset of $Y$, $f: Y\to X=Y/K$ is an elementary map and $G$ is $\tau_g$-representable in (a locally compact space) $X$. 

Then $H^{X}_{\dashv}\circ F^H=H^{Y}_{\dashv}$, where $H^{X}_{\dashv}:  (2^{X\times X}, \tau_F)\to (2^{(X\setminus f(K))\times (X\setminus f(K))}, \tau_F)$, $H^Y_{\dashv}:  (2^{Y\times Y}, \tau_F)\to (2^{(Y\setminus K)\times (Y\setminus K)}, \tau_F)$ and $F^H:  (2^{Y\times Y}, \tau_F)\to (2^{X\times X}, \tau_F)$ are maps induced by the corresponding restrictions and the map $f$.

The condition (R) holds for $Y$ and $Y\setminus K$ iff the condition (PM) holds for $Y$ and $f$, and the condition (R) holds for $X$ and $X\setminus f(K)$.

The condition (ER) holds for $Y$ and $Y\setminus K$ iff the condition (EPM) holds for $Y$ and $f$, and the condition (ER) holds for $X$ and $X\setminus f(K)$. }
\end{rem}

\begin{cor}\label{emgroup}
Let a topological group $G$ be $\tau_g$-representable in a locally compact space $X$. The following conditions are equivalent
\begin{itemize}
\item[(a)]  $b_X G$ is a graph compactification of $G$, 
\item[(b)] for any $G$-compactification $b X$ of $X$ $b_{b X} G$ is a graph compactification of $G$ and the condition {\rm(R)} for $X$ and $b X$ holds,
\item[(c)] $\exists$ $G$-compactification $b X$ of $X$ such that  $b_{b X} G$ is a graph compactification of $G$ and the condition {\rm(R)} for $X$ and $b X$  holds.
\end{itemize}
\end{cor}

\begin{proof} 
(a) $\Longrightarrow$ (b). Take any $G$-compactification $(G, bX,  \curvearrowright)$ of $(G, X,  \curvearrowright)$. Since the embedding $j: X\to b X$ is a $G$-map and $j(X)$ is open in $b X$, by Proposition~\ref{mapcompV} (b) $b_{b X} G$ is a graph compactification of $G$ and $H_{\dashv}|_{b_{b X} G}$ is the $G$-map of compactifications $b_{b X} G$ and $b_X G$ of $G$. 

Since $H_{\dashv}|_{\i^{b X}_{\Gamma} (G)}$ is a homeomorphism, ($H_{\dashv}|_{b_{b X} G} (b_{b X} G\setminus\i_{b X}^{\Gamma} (G))=b_X G\setminus\i_{X}^{\Gamma} (G)$~\cite[Lemma 3.5.6]{Engelking}. 
Hence, if $R\in b_{b X} G$ and $g\in G$ are such that $R\cap (j(X)\times j(X))=\i_{b X}^{\Gamma} (g)\cap (j(X)\times j(X))$, then $ R=\i_{b X}^{\Gamma} (g)$ and  the condition {\rm(R)} holds.

(b) $\Longrightarrow$ (c) is evident. (c) $\Longrightarrow$ (a) follows from Theorem~\ref{restrmap}. 
\end{proof}

\begin{ex}
{\rm For the group $\Hom (C)$, $C\subset [0, 1]$ is the Cantor set, in the c-o.t., let a subgroup $G$ be the stabilizer of the point 1. Then $G$ is $\tau_g$-representable on $X=C\setminus\{1\}$ by item (iii) of  Remark~\ref{rem1} and $\alpha X=C$. In order to check that the condition {\rm(R)} for $X$ and $\alpha X$ doesn't hold it is enough to show that the set $R=\{(x, x)\ |\ x\in\alpha X\}\bigcup\{(0, 1)\}$ is in $b_{\alpha X} G$. 

Without loss of generality, take a nbd 
$$[V_1\times V_1, \ldots, V_n\times V_n, V_1\times V_n],\ n\in\mathbb N,$$
of $R$, where $\{V_i\ |\ i=1,\ldots, n\}$ is a partition of $X$, each element is a clopen nonempty set homeomorphic $C$.   

Take partitions $\{U_1,U_2, U_3\}$ of $V_1$ and $\{W_1, W_2, W_3\}$ of $V_n$, , each element is a clopen nonempty set homeomorphic $C$. Put    
$g\in G$ be identity on $U_1\bigcup\bigcup\limits_{i=2}^{n-1} V_i\bigcup W_3$; $g: U_2\to U_2\cup U_3$,  $g: U_3\to W_1$, $g: W_1\cup W_2\to W_2$ be gomeomorphisms. Then $g\in [V_1\times V_1, \ldots, V_n\times V_n, V_1\times V_n]$. Hence, $R\in b_{\alpha X} G$.

\medskip

Note, that for an $h$-homogeneous compactum (in particular, Cantor set) $X$, $b_X \Hom (X)=\{R\in \CL (X\times X)\ |\ \pr_1 R=\pr_2 R=X\}$~\cite{usp2001}.}
\end{ex}

\begin{lem}\label{proj}~\cite{usp2001}
Let a topological group $G$ be $\tau_g$-representable in a compact space $X$ and $\pr_1, \pr_2$ be projections of $X\times X$ on the first and the second factor respectively. Then $\pr_1(R)=\pr_2(R)=X$ $\forall\ R\in b_{X} G$.
\end{lem}

\begin{rem}\label{coincompB}
{\rm If in the assumptions of Theorem~\ref{restrmap} $Y=\alpha X$, then the set $O_{Y\setminus X}$ is either a one-point set or an empty set by Lemma~\ref{proj}. Hence, if the condition (ER) is fulfilled, then $H_{\dashv}|_ {b_{\alpha X} G}: b_{\alpha X} G\to b_X G$ is a homeomorphism.}
\end{rem}

\begin{que}
Does there exist an example a topological group $G$ $\tau_g$-representable in a locally compact space $X$ such that for $X$ and the Alexandroff one-point compactification $\alpha X$ the condition {\rm(R)} holds but the the condition {\rm(ER)} doesn't hold? 
\end{que}

The family of maps $\mathcal F: X\to Y$ is {\it topologically equicontinuous}~\cite{Royden} (see, also,~\cite{Corb}) if $\forall\ x\in X$, $\forall\ y\in Y$ and $\forall$ nbd $O_y$ of $y$ $\exists$ nbds $U_x$ and $V_y$ of $x$ and $y$ respectively, such that 
$$(f(U_x)\cap V_y\ne\emptyset)\ \Longrightarrow\ (f(U_x)\subset O_y),\ f\in\mathcal F.$$

\begin{thm}\label{partialcases}
Let a topological group $G$ be $\tau_g$-representable in a locally compact space $X$ {\rm(}$(G, X, \theta)$ is a $G$-Tychonoff space{\rm)}. If 
\begin{itemize}
\item[{\rm (a)}] the family of maps $\theta^g: X\to X$, $g\in G$, is topologically equicontinuous, or 
\item[{\rm (b)}] $X$ is a locally compact, locally connected space,  
\end{itemize}
then for the Alexandroff one-point $G$-compactification $\alpha X=X\cup\{\infty\}$ of $X$ the condition {\rm(ER)} holds. 

\medskip

Hence, $\i_X^{\Gamma}: G\to (2^{X\times X}, \tau^{X}_F)$ is an embedding and $b_{\alpha X} G=b_X G$. 
\end{thm}

\begin{proof}
By Corollary~\ref{coinctopolog} and Lemma~\ref{l11} $G$ is $\tau_g$-representable in $\alpha X$. Therefore, by Lemma~\ref{lemcompemb} $b_{\alpha X} G$ is a graph compactification of $G$. 

Firstly, $\forall\ R\in b_{\alpha X} G$  $(\infty, \infty)\in R$. Indeed, if $(\infty, \infty)\not\in R$, then $\exists$ a nbd $O$ of $(\infty, \infty)$ such that for its closure $K$ in $\alpha X\times\alpha X$ one has $K\cap R=\emptyset$. Then $R\in ((\alpha X\times\alpha X)\setminus K)^+$. However, $((\alpha X\times\alpha X)\setminus K)^+\cap\i_X^{\Gamma} (G)=\emptyset$. Hence,  $R\not\in b_{\alpha X} G$. 

(a) $\forall\ R\in b_{\alpha X} G$ and $\forall\ x\in X$  $|(\{x\}\times\alpha X)\cap R|=1$ (the usage of the self-inverse map $S_X$ implies that $\forall\ R\in b_{\alpha X} G$ and $\forall\ x\in X$  $|(\alpha X\times\{x\})\cap R|=1$). 

By Lemma~\ref{proj}  $\forall\ R\in b_{\alpha X} G$  $\pr_1(R)=\pr_2(R)=\alpha X$. Hence,  $\forall\ R\in b_{\alpha X} G$ and $\forall\ x\in X$  $|(\{x\}\times\alpha X)\cap R|\geq 1$.

Assume that $(x, y), (x, z)\in R$, $y\ne z$, $R\in b_{\alpha X} G$, $ x\in X$ and $y\ne\infty$. There are disjoint nbds $O_z$ and $O_y$ of $z$ and $y$ respectively. Since the family of maps $\theta^g: X\to X$, $g\in G$, is topologically equicontinuous  $\exists$ nbds $U_x$ and $V_y\subset O_y$ of $x$ and $y$ respectively, such that 
$$(\theta^g(U_x)\cap V_y\ne\emptyset)\ \Longrightarrow\ (\theta^g(U_x)\subset O_y),\ g\in G.$$
Then $W=(U_x\times V_y)^-\cap (U_x\times O_z)^-$ is a nbd of $R$. However, $W\cap\i_{\alpha X}^{\Gamma} (G)=\emptyset$. The obtained contradiction shows that $\forall\ R\in b_{\alpha X} G$ and $\forall\ x\in X$  $|(\{x\}\times\alpha X)\cap R|=1$. 

\medskip

(b) Take $R, S\in  b_{\alpha X} G$ such that $(R\cap (X\times X))=(S\cap (X\times X))$. If $R\ne S$, then, without loss of generality, one can assume that $\exists\ x\in X$ such that $(x, \infty)\in S$ and $(x, \infty)\not\in R$.

Since $R\in\CL (X\times X)$, $(x, \infty)\not\in R$ and $X$ is a locally connected space, $\exists$ a connected nbd $U_x$ of $x$ and a nbd $O$ of $(\{x\}\times \alpha X)\cap R$ in $X$ such that $(\cl\, U_x\times\alpha X)\cap R\subset\cl\, U_x\times O$, $\alpha X\setminus\cl\, O$ is a nbd of $\{\infty\}$. Take $W=(U_x\times O)^-\cap (\cl\, U_x\times\Bd\, O)^+\cap (U_x\times (\alpha X\setminus\cl\, O))^-$, $S\in W$. Since $\forall\ g\in\i_X^{\Gamma} (G)$ $g (U_x)$ is connected, if $g (U_x)\cap O\ne\emptyset$ and $g (U_x)\cap (\alpha X\setminus\cl\, O)\ne\emptyset$, then $g (U_x)\cap\Bd\, O\ne\emptyset$. Therefore, $S\not\in b_{\alpha X} G$. 

These observations yeild that for the Alexandroff one-point compactification $\alpha X$ the condition {\rm(ER)} in (a) and (b) holds. By Corollary~\ref{emgroup} $\i_X^{\Gamma}$ is an embedding and by Remark~\ref{coincompB}  $b_{\alpha X} G=b_X G$.
\end{proof}

\begin{rem}\label{unifeqtopeq}
{\rm (A) An action $\theta: G\times (X, \mathcal L)\to  (X, \mathcal L)$ is {\it uniformly equicontinuous} if $\{\theta^g: X\to X,\ \theta^g(x)=\theta (g, x)\ |\ g\in G\}$  is a uniformly equicontinuous family of maps on a uniform space $(X, \mathcal L)$. Equivalently, for any $u\in\mathcal L$ there exists $v\in\mathcal L$  such that $gv\succ u$, for any $g\in G$. If the action  $G\curvearrowright (X, \mathcal L)$ is uniformly equicontinuous, then the t.p.c. $\tau_p$ is the least admissible group topology on $G$~\cite[Lemma 3.1]{Kozlov2022}. Moreover, $((G, \tau_p), X, \theta)$ is a $G$-Tychonoff space~\cite{Megr1984}. 

If an action $\theta: G\times (X, \mathcal L)\to  (X, \mathcal L)$ is uniformly equicontinuous, then  the family of maps $\theta^g: X\to X$, $g\in G$, is topologically equicontinuous, see, for example, \cite[Remark 3.1]{Corb}.

\medskip

(B) For a locally compact, locally connected separable metrizable space $X$ and a $\tau_g$-representable in $X$ group $G$, embeddability of $G$ into $(2^{X\times X}, \tau_F)$ in Theorem~\ref{partialcases} is proved in~\cite{Yamashita}.}
\end{rem}


\begin{thm}\label{suffcond}
Let a topological group $G$ be $\tau_g$-representable in a compact space $X$ {\rm(}$\mathcal U_X$ is the unique uniformity on $X${\rm)}. 
The following conditions are equivalent: 

{\rm(}a{\rm)} $\forall\ {\rm U}\in\mathcal U_X$ $\exists\ {\rm V}\in\mathcal U_X$ such that 
$$\mbox{if}\ (\i_X^{\Gamma} (f), \i_X^{\Gamma} (h))\in 2^{{\rm V}^2},\ f, h\in G,\ \mbox{then}\ \exists\ g\in G$$
$$\mbox{such that}\  (f(x), g(x))\in {\rm U},\  (g^{-1}(x), h^{-1}(x))\in {\rm U}\ \forall\ x\in X,\eqno{(\star)}$$ 

{\rm(}b{\rm)} $G$ is Roelcke precompact and $b_r G=b_X G$.
\end{thm}

\begin{proof} (a) $\Longrightarrow$ (b)
By Corollary~\ref{closure} $b_r G\geq b_X G$ ($L\wedge R\geq 2^{\mathcal U_X}|_{\i_X^{\Gamma} (G)}$).

Since the topology of uniform convegence coincides with the compact-open topology on $G$, one may assume that any $O\in N_G (e)$ (not open) is of the form 
$$O=O_{{\rm U}}=\{g\in G\ |\ (x, g(x))\in  {\rm U},\ x\in X\},\ {\rm U}\in\mathcal U_X.$$
The condition $(\star)$ implies that $\exists\ {\rm V}\in\mathcal U_X$ such that if $(\i_X^{\Gamma} (f), \i_X^{\Gamma} (h))\in 2^{{\rm V}^2}$, then $\exists\ g\in G$ such that $f\in Og$ and $g^{-1}\in Oh^{-1}$ (since $O^{-1}=O$, $g\in hO$). Hence, $f\in OhO$, the map $(\i_X^{\Gamma} (G), 2^{\mathcal U_X}|_{\i_X^{\Gamma} (G)})\to (G, L\wedge R)$ is uniformly continuos and $L\wedge R\leq 2^{\mathcal U_X}|_{\i_X^{\Gamma} (G)}$. Therefore, $G$ is Roelcke precompact and $b_r G=b_X G$. 

(b) $\Longrightarrow$ (a)
The covers $\{OgO\ |\ g\in G\}$ where $O=O_{{\rm U}}$, ${\rm U}\in\mathcal U_X$, form the base of the Roelcke uniformity on $G$. Take $f\in OhO$. Then 
$f=\varphi h\psi$, $\varphi, \psi\in O$. If $g=h\psi$, then $(f(x), g(x))=(\varphi g(x), g(x))\in {\rm U}$ $\forall\ x\in X$ and $(g^{-1}(x), h^{-1}(x))=(g^{-1}(x), \psi g^{-1}(x))\in {\rm U}$ $\forall\ x\in X$. 

If $2^{\mathcal U_X}|_{\i_X^{\Gamma} (G)}=L\wedge R$, $\forall\ {\rm U}\in\mathcal U_X$ take $O=O_{{\rm V}}$, ${\rm V}\in\mathcal U_X$, ${\rm V}\subset {\rm U}$ such that the cover $\{OgO\ |\ g\in G\}$ is a refinement of the cover $\{B(g, {\rm U}^2)\ |\ g\in G\}$ (by balls correspondent to the entourage ${\rm U}^2$ on $X\times X$). Then the condition $(\star)$ is valid. 
\end{proof}

\begin{rem}
{\rm The description of the Roelke uniformity on a topological group $G$ which is $\tau_g$-representable in a metrizable compact space $X$ using the property $(\star)$ is given in~\cite[Lemma 4.1]{Megr2001}.}
\end{rem}

\begin{que}
Let $G$ be a topological group, $b G$ be $G$-compactification of $G$. For what invariant closed subsets $K\subset b G\setminus G$ the graph compactification $b_{b G\setminus K} G$ of $G$ is defined?
\end{que}


\section{Unitary and permutation groups}\label{unpergroups}

\subsection{Unitary group of a Hilbert space}\label{unitarygroup}
Let ${\rm U}({\bf H})$ be the unitary group of a (complex or real) Hilbert space ${\bf H}$. It is a group of isometries for the action of ${\rm U}({\bf H})$ on the unit sphere ${\rm S}\subset{\bf H}$, and  ${\rm U}({\bf H})$ in the  {\it strong operator topology} is $\tau_p$-representable in ${\rm S}$  (see for example, \cite{Rud}). 

Moreover, ${\rm S}$ is a coset space of $({\rm U}({\bf H}), \tau_p)$ under its action~\cite[\S\ 2.1]{Kad} and the stabilizer $\St_x$ of any point $x\in {\rm S}$ is a neutral subgroup of ${\rm U}({\bf H})$~\cite[Remark 3.14]{Kozlov2022}. The maximal equiuniformity $\mathcal U_{\rm S}$ on ${\rm S}$ is totally bounded~\cite[Proposition 3.20]{KozlovSorin} and the maximal $G$-compactification of ${\rm S}$ is the unit ball ${\rm B}$ in ${\bf H}$ in the {\it weak topology}~\cite[Corollary 2.3]{Stojanov}. 

Further, ${\rm U}({\bf H})$ is Roelcke precompact and  its Roelcke compactification $b_r {\rm U}({\bf H})$ is the space of {\it contract operators} on ${\bf H}$ (linear operators on ${\bf H}$ of norm $\leq 1$) in the {\it weak operator topology} and a sm$^*$-compactification of $U({\bf H})$ (involution of operator is the {\it adjoint} operator)~{\rm\cite{U1998}}. Hence,  $b_r {\rm U}({\bf H})$ is a WAP-compactification of ${\rm U}({\bf H})$. Moreover, $b_r {\rm U}({\bf H})=e_{\rm B} {\rm U}({\bf H})$~\cite[Theorem 3.23]{KozlovSorin}. Therefore, 
$$b_r {\rm U}({\bf H})=e_{\rm B} {\rm U}({\bf H})={\rm WAP}\, {\rm U}({\bf H}).$$

\begin{rem}
{\rm (A) ${\rm U}({\bf H})$ is even {\it strongly Eberlein} {\rm(}see definition in~\cite{GlasnerMegr}{\rm)}. 

\medskip

(B) The description of $b_r {\rm U}({\bf H})=\{f\ \mbox{is a linear operator on}\ {\bf H}\ |\ ||f||\leq 1\}$ yields that there are continuum  sm$^*$-compactifications of ${\rm U}({\bf H})$. Indeed, $\forall\ 0<\varepsilon<1$ the ball ${\mathcal B}_{\varepsilon}=\{f\in b_r {\rm U}({\bf H})\ |\ ||f||\leq\varepsilon\}$ is a compact ideal, invariant under involution. The Rees quotients $b_r {\rm U}({\bf H})/{\mathcal B}_{\varepsilon}$ are sm$^*$-compactifications of ${\rm U}({\bf H})$~\cite[Lemma 2.4]{KozlovSorin2025}.

\medskip

(C) The description of $b_r {\rm U}({\bf H})$ as the Ellis compactification $e_{\rm B} {\rm U}({\bf H})$ of ${\rm U}({\bf H})$ and the description of $G$-compactifications of a $G$-space $({\rm U}({\bf H}), {\rm S}, \curvearrowright)$ in~\cite[\S\ 3]{Stojanov} yield that there are continuum  sm-compactifications of ${\rm U}({\bf H})$. Indeed, $\forall\ 0<\varepsilon<1$  the ball ${\rm B}_{\varepsilon}=\{x\in {\bf H}\ |\ ||x||\leq\varepsilon\}$ is a compact invariant subset of the maximal $G$-compactification ${\rm B}$ of ${\rm S}$. The map ${\rm B}\to {\rm B}/{\rm B}_{\varepsilon}$ of $G$-compactifications  of ${\rm S}$ induces the map of the correspondent Ellis compactifications of ${\rm U}({\bf H})$~\cite[\S\ 3]{KozlovSorin2025}. Since $e_{\rm B} {\rm U}({\bf H})$ is a sm-compactifications of ${\rm U}({\bf H})$,  $e_{{\rm B}/{\rm B}_{\varepsilon}} {\rm U}({\bf H})$  is a sm-compactification of ${\rm U}({\bf H})$ by Proposition~\ref{smcompimage}. 

It is not difficult to show that sm-compactifications of ${\rm U}({\bf H})$ constructed in (A) and (B) are pairwise different. 

\medskip

(D) ${\rm U}({\bf H})$ is $\tau_g$-representable in ${\rm B}$. By Corollary~\ref{closure} $b_{\rm B} {\rm U}({\bf H})\leq b_r {\rm U}({\bf H})$ for the graph compactification $b_{\rm B} {\rm U}({\bf H})$ of ${\rm U}({\bf H})$.}
\end{rem}

\begin{que}
Do the equality $b_{\rm B} {\rm U}({\bf H})=b_r {\rm U}({\bf H})$ and  inequalities $b_{{\rm B}/{\rm B}_{\varepsilon}} {\rm U}({\bf H})< b_{\rm B} {\rm U}({\bf H})$, $0<\varepsilon<1$, hold?
\end{que}


\subsection{Permutation group of a discrete space}\label{ultdiscrete}
A discrete space $X$ is a locally compact, locally connected space, ${\rm S} (X)$ is the permutation group of $X$. Any subgroup $G$ of 
$({\rm S} (X), \tau_p)$ is $\tau_p$-representable in $X$ and the action is uniformly equicontinuous (by isometries). 

\begin{pro}\label{discrsmcomp}
Let $G$ be a subgroup of $({\rm S} (X), \tau_p)$. Then $G$ has a sm$^*$-compactification.
\end{pro}

\begin{proof}
Let us show that $\exists$ a Hilbert space ${\bf H}$ such that $G$ is $\tau_p$-representable in ${\bf H}$. Take a {\it sequence space} $\ell^2(X)$ see, for example~\cite{Rud}, with orthonormal basis 
$$e_x(y)=\left\{
\begin{array}{ll}
1 & y=x, \\
0 & y\ne x, \\
\end{array}
\right. x, y\in X.$$
${\rm S}$ is a unit sphere of $\ell^2(X)$. An embedding $i: X\to {\rm S}$, $i(x)=e_x$ and the monomorphism $\varphi: G\to {\rm U}(\ell^2(X))$ (the unique extension of the bijection $g(e_x)=e_{g(x)}$ defined on the basis to $\ell^2(X)$) are defined. By Lemma~\ref{comptauprep}  $\varphi$ is a topological isomorpism of $G$ onto a subgroup of ${\rm U}(\ell^2(X))$. Therefore, $e_{\rm B} G$ is a sm$^*$-compactification of $G$. 
\end{proof}

\begin{cor}\label{properpermgroup}
If $G$ is a subgroup of $({\rm S} (X), \tau_p)$, then the {\rm WAP}-compactification of $G$ is a proper compactification of $G$. 
\end{cor}

\begin{rem}\label{remdiscrcomp}
{\rm (A) The closure of $i(X)=X$ in ${\rm B}$ in the weak topology is the Alexandroff one-point compactification $\alpha X=\{e_x\ |\ x\in X\}\cup\{0\}$. By Theorem~\ref{compRoelcke} $G$ has a sm-compactification $e_{\alpha X}G$, $e_{\alpha X}G\leq e_{\rm B} G$.

\medskip

(B) Since the orthonormal basis $e_x$, $x\in X$, is a discrete subset of the unit sphere ${\rm S}$, ${\rm S} (X)$ is a closed subgroup of ${\rm U}(\ell^2(X))$.

\medskip

(C) In terms of representations theory of dynamical systems on Banach spaces~\cite{GlasnerMegr} the proof of Proposition~\ref{discrsmcomp} yields that the 
dynamical system $(G, X)$ is represented in a Hilbert space. 

\medskip

(D) The group  $({\rm S} (\mathbb N), \tau_p)$ is strongly Eberlein~\cite{GlasnerMegr}.

\medskip

(E) In~\cite[Theorem 1]{SorinG2} is shown that any subgroup $G$ of $({\rm S} (X), \tau_p)$ has a sm-compactification. In~\cite[Corollary 1]{SorinG2} the  Roelcke compactification of a subgroup $G$ of $({\rm S} (X), \tau_p)$ which action is {\it oligomorphic} is described.}
\end{rem}

\begin{que}
Let $G$ be a subgroup of $({\rm S} (X), \tau_p)$. Does the equality $e_{\alpha X} G=e_{\rm B} G$ hold?
\end{que}

\medskip

{\it Ultratransitivity and graph compactifications.} Let $G$ be a subgroup of $({\rm S} (X), \tau_p)$ and the action of $G$ on $X$ is {\it ultratransitive} (for pairwise distinct points $x_1,\ldots, x_n$ and $y_1,\ldots, y_n$ $\exists\ g\in G$ such that $g(x_i)=y_i$, $i=1,\ldots, n$, $n\in\mathbb N$). $G$ is a dense subgroup of $({\rm S}(X), \tau_p)$~\cite[Remark 5.3]{KozlovSorin2025}. 

The maximal equiuniformity $\mathcal U^{\max}_X$ on $X$ is totally bounded. From~\cite[Section 4]{KozlovSorin2025} it follows. The completion of $(X, {\mathcal U}^{\max}_X)$ is the one-point Alexandroff compactification $\alpha X$, $e_{\alpha X} G=b_r G$ and is the maximal sim-compactification of $G$ (and, hence, the WAP-compactification of $G$ is algebrically isomorphic to the symmetric inverse semigroup $I_X$). $b_r G$ is the set of selfmaps $f$ of $\alpha X$  in the t.p.c. such that
\begin{itemize}
\item[{\rm (i)}] $f$ is a bijection on $Y\subset X$, where $Y$ is an arbitrary subset of $X$, 
\item[{\rm (ii)}] $f(\alpha X\setminus Y)=\infty$. 
\end{itemize}

\medskip 

$\mathbb{GLC} (X)=\{X,\ \alpha X\}$ and $G$ is $\tau_g$-representable in $X$ and $\alpha X$ by Corollary~\ref{coinctopolog} and Lemma~\ref{l11}. By Lemma~\ref{lemcompemb} and Theorem~\ref{partialcases} item (a) (or (b)) $b_X G=b_{\alpha X} G$. 

Put $\i_X^{\Gamma}: G\to (2^{X\times X}, \tau_F)$, $\i_X^{\Gamma} (g)=\{(x, gx)\ |\ x\in X\}$, $g\in G$.

\begin{pro}\label{prop2}
The closure $b_X G$ of $\i_X^{\Gamma} (G)$ in $(2^{X\times X}, \tau_F)$ is {\rm(}algebraically{\rm)} the symmetric inverse semigroup $I_X$ and is a sim-compactification of $G$ {\rm(}see definition in~\cite{KozlovLeiderman2025}{\rm)}. 
\end{pro}

\begin{proof} 
Let $\pr_1$ and $\pr_2$ be projections of $X\times X$ on the first and the second factor, respectively. For $A\in 2^{X\times X}$ put ${\rm D} (A)=\pr_1 (A)$, ${\rm I} (A)=\pr_2 (A)$. 

Take arbitrary $A\in b_X G=\cl\ \i_X^{\Gamma} (G)$. If $(x, y_1), (x, y_2)\in A$, $y_1\ne y_2$, then $W=\{(x, y_1)\}^-\cap\{(x, y_1)\}^-$ is a nbd of $A$ in $(2^{X\times X}, \tau_F)$.  Evidently, $W\cap\i_X^{\Gamma} (G)=\emptyset$ and $A\not\in\cl (\i_X^{\Gamma} (G))$. The case  $(x_1, y), (x_2, y)\in A$,  $x_1\ne x_2$ is treated similarly. Hence, $b_X G\subset I_X$.

\medskip

Take $A\in I_X\subset 2^{X\times X}$. If $A\ne\emptyset$, then an arbitrary nbd of $A$ in $(2^{X\times X}, \tau_F)$ may be choosen of the form 
$$\{(x_1, y_1)\}^-\bigcap\ldots\bigcap \{(x_n, y_n)\}^-\bigcap ((X\times X)\setminus K)^+,$$
where $(x_i, y_i)\in A$, $i=1, \ldots, n$, $K$ is a compact (and, hence, finite) subset of $X\times X$ and $A\cap K=\emptyset$.  Ultratransitivity yields, that one  can find $g\in G$ such that $g(x_i)=y_i$, $i=1, \ldots, n$, $(t, g(t))\not\in K$, $t\in X$. Hence, $A\in\cl\ \i_X^{\Gamma} (G)$. 

If $A=\emptyset$, then an arbitrary nbd of $A$ in $(2^{X\times X}, \tau_F)$ is of the form $((X\times X)\setminus K)^+$, where $K$ is a compact (and, hence, finite) subset of $X\times X$. Ultratransitivity yields, that one can find $g\in G$ such that $(t, g(t))\not\in K$, $t\in X$. Hence, $A\in  \i_X^{\Gamma} (G)$. 

\medskip

Fix $A\in b_X G$. For any $B\in b_X G$ let 
$$W=\{(x_1, y_1)\}^-\bigcap\ldots\bigcap \{(x_n, y_n)\}^-\bigcap ((X\times X)\setminus K)^+,$$
where $(x_i, y_i)\in AB$, $i=0, 1, \ldots, n$, $K$ is a compact (and, hence, finite) subset of $X\times X$ and $(AB)\cap K=\emptyset$, be a nbd of $AB=\{(x, y)\ |\ \exists\ z\in X\ (z, y)\in A,\ (x, z)\in B\}$. 

Let $z_i$ be such that $(z_i, y_i)\in A$, $(x_i, z_i)\in B$, $i=0, 1, \ldots, n$, $\{t_1,\ldots, t_m\}=\pr_1 (K)$. If $(\{t_i\}\times X)\cap B\ne\emptyset$, then $(\{t_i\}\times X)\cap B=\{(t_i, b_i)\}$ (is a one point set) and let $(t_{i_1}, b_{i_1}), \ldots, (t_{i_k}, b_{i_k})$ be all such points. For any $ (t_{i_j}, b_{i_j})$, $j=1, \ldots, k$, $(\{b_{i_j}\}\times X)\cap A=\emptyset$. Therefore, the set 
$$U=\{(x_1, z_1)\}^-\bigcap\ldots\bigcap \{(x_n, z_n)\}^-\bigcap \{(t_{i_1}, b_{i_1})\}^-\bigcap\ldots\bigcap \{(t_{i_k}, b_{i_k})\}^-$$
is a nbd of $B$ and $AU\subset W$. Thus, $b_X G$ is a left topological monoid. Since the involution is continuous, $b_X G$ is a sim-compactification of $G$ by~\cite[Remark 3.9]{KozlovLeiderman2025}. 
\end{proof}

\begin{thm}\label{equivcompperm}
$b_r G=e_{\alpha X} G=b_{\alpha X} G=b_X G$. 
\end{thm}

\begin{proof} The map $\Psi: e_{\alpha X} G\to b_{X} G$, $\Psi (f)=\{(x, f(x))\ |\ x\in X,\ f(x)\ne\infty\}$, if $f$ is not a constant map ${\rm const}_{\infty}: \alpha X\to\{\infty\}$, $\Psi ({\rm const}_{\infty})=\emptyset$ is a continuous bijection of compacta and $\Psi (\i_{\alpha X}(g))=\i_X^{\Gamma} (g)$, $g\in G$. Therefore,  $e_{\alpha X} G=b_X G$ and the equalities are proved. 
\end{proof}

\begin{rem}{\rm The map $\Psi$ in Proposition~\ref{equivcompperm} is a topological isomorphism of semigroups. 

\medskip

$b_{\alpha X} G$ is the family of sets $A\in\CL (\alpha X\times\alpha X)$ such that 
\begin{itemize}
\item[{\rm (i)}] $A\cap (X\times X)$ is the graph of a bijection of ${\rm D} (A)$ onto ${\rm I} (A)$, 
\item[{\rm (ii)}] $\pr_1 (A\setminus (X\times X))=\alpha X\setminus {\rm D} (A)$, $\pr_2 (A\setminus (X\times X))=\alpha X\setminus {\rm I} (A)$
\end{itemize}
and is not a semigroup with the binary operation be the composition of relations.}
\end{rem} 

\medskip

From Proposition~\ref{discrsmcomp} and Proposition~\ref{compRoelcke} it follows that $e_{\alpha X} G=e_{{\rm B}}G$. Since  $e_{{\rm B}}G$ is the closure of $G$ in $e_{{\rm B}} {\rm U}(\ell^2(X))$, from Theorem~\ref{equivcompperm} it follows 

\begin{cor}\label{wappermgr}
$b_r G=e_{\alpha X} G=e_{{\rm B}}G$, 

The Roelcke compactification $b_r G$ of $G$ is the closure of $G$ in $b_r  {\rm U}(\ell^2(X))={\rm WAP}\,  {\rm U}(\ell^2(X))$ and is a sim-compactification and  {\rm WAP}-compactification of $G$. $G$ is a {\rm WAP} group {\rm(}see definition in~\cite{GlasnerMegr}{\rm)}.
\end{cor}

\begin{rem}
{\rm Fix an orthonormal basis $\mathcal E$ in ${\bf H}$ and examine a closed subgroup ${\rm U}_{\mathcal E}({\bf H})$ of $({\rm U}({\bf H}), \tau_p)$ which elements are linear extensions of bijections of $\mathcal E$. Then 
\begin{itemize}
\item[{\rm (A)}] ${\rm U}_{\mathcal E}({\bf H})$ is a closed subgroup of $({\rm U}({\bf H}),  \tau_p)$, 
\item[{\rm (B)}] the closure of ${\rm U}_{\mathcal E}({\bf H})$ in $b_r {\rm U}({\bf H})$ is ${\rm PU}_{\mathcal E}({\bf H})$, the set of partial linear isometries of ${\bf H}$ which are linear extensions of partial bijections of $\mathcal E$ assuming that an extension is equal 0 on those elements of 
$\mathcal E$ on which the partial bijection is not defined, 
\item[{\rm (C)}] ${\rm PU}_{\mathcal E}({\bf H})$ is a sm$^*$-compactification of ${\rm U}_{\mathcal E}({\bf H})$, 
\item[{\rm (D)}] $b_r {\rm U}_{\mathcal E}({\bf H})=e_B {\rm U}_{\mathcal E}({\bf H})={\rm WAP}\, {\rm U}_{\mathcal E}({\bf H})={\rm PU}_{\mathcal E}({\bf H})$,
\item[{\rm (E)}] ${\rm U}_{\mathcal E}({\bf H})$  a {\rm WAP} group.
\end{itemize}
The uniformities, correspondent to the Roelcke, Ellis and WAP-compactifications are hereditary on a subgroup of ${\rm U}_{\mathcal E}({\bf H})$ topologically isomorphic to the permutation group ${\rm S} (\mathcal E)$. }
\end{rem}

\begin{rem}{\rm The permutation group ${\rm S}(\mathbb N)=({\rm S}(\mathbb N), \tau_{\partial})$ of naturals $\mathbb N$ is a  non-archimedean Roelcke precompact~\cite{Gau} (see also~\cite[Example 9.14]{RD}) Polish group. 

In~\cite[\S\ 12, Theorem 12.2]{GlasnerMegr2008} the description of Roelcke compactification $b_r  {\rm S}(\mathbb N)$ of ${\rm S}(\mathbb N)$ is given using its Ellis compactification from the $\tau_p$-representation of ${\rm S}(\mathbb N)$ in $\alpha\mathbb N$. It is a symmetric inverse semigroup $I_{\mathbb N}$~\cite{BITsankov} and $b_r {\rm S}(\mathbb N)$ is a WAP-compactification of ${\rm S}(\mathbb N)$. It is proved that $b_r  {\rm S}(\mathbb N)$ is homeomorphic to the Cantor set. 

In~\cite{BITsankov} it is proved that ${\rm S}(\mathbb N)$ is strongly Eberlein, $b_r {\rm S}(\mathbb N)$ is a {\it Hilbert compactification} of ${\rm S}(\mathbb N)$ and all its factors are {\it Hilbert--representable} (see definitions in~\cite{GlasnerMegr}). 

\medskip

See~\cite[\S\ 5.4]{EJMMMP} about topologies on the symmetric inverse monoid $I_X$. In~\cite[Theorem 5.12]{EJMMMP} it is proved, in fact, that the symmetric inverse semigroup $I_X$ in Fell topology is a compact semitopological monoid with continuous inversion.}
\end{rem}

\begin{que}
Let the action of $G$ on a discrete space $X$ be oligomorphic. Do the equalities $e_{\alpha X} G=b_X G={\rm WAP}\, G$ hold?
\end{que}


\subsection{Automorphism group of an ultrahomogeneous chain}\label{ultrchain}

A discrete chain $X$ is a GO-space ({\it generalized ordered space}) and is a locally compact and locally connected space. If $X$ is {\it ultrahomogeneous}  (for pairwise distinct points $x_1<\ldots<x_n$ and $y_1<\ldots< y_n$ $\exists\ g\in G$ such that $g(x_i)=y_i$, $i=1,\ldots, n$, $n\in\mathbb N$), then there exists  the smallest LOTS 
$$X\otimes_{\ell}\{-1, 0, 1\}$$
(topology of linear order is induced by the lexicographic order on $X\times\{-1, 0, 1\}$),  
in which $X$ is a dense subspace, and $X\otimes_{\ell}\{-1, 0, 1\}$ is naturally embedded in any other LOTS in which $X$ is dense~\cite{Miwa}. Hence, any linearly ordered compactification of $X$ is a linearly ordered compactification of $X\otimes_{\ell}\{-1, 0, 1\}$. 
The least linearly ordered compactification $b_m X$ of $X\otimes_{\ell}\{-1, 0, 1\}$ (and, hence, of $X$) is generated by replacing each gap (also improper) by a point with a natural continuation of the order (gaps in $X\otimes_{\ell}\{-1, 0, 1\}$ and  $X$ are naturally identified). Therefore, $b_m X=X^+\cup X^0\cup X^-\cup\Gamma$, where $X^+=X\times\{1\}$, $X^0=X\times\{0\}$ ($X$ and $X^0$ are identified), $X^-=X\times\{-1\}$, $\Gamma$ is the set of gaps. Each gap is defined by a {\it cut} $(A, B)$ ($A, B\subset X$, $A\cap B=\emptyset$, $A\cup B=X$, if $x\in A$, $y\in B$, then $x<y$)~\cite[\S\ I.3]{Engelking}. For $x\in X$ designations $x^-=(x, -1)$, $x^+=(x, 1)$ are used. $b_m X$ is a zero-dimensional compact LOTS. If $X$ is {\it continuously ordered}, then $b_m X=\{\inf\}\cup X^+\cup X^0\cup X^-\cup\{\sup\}$ is the unique linearly ordered compactification of $X$.

Let $X$ be an ultrahomogeneous chain, and $G=(\aut (X), \tau_{\partial})$ is its automorphism group in the t.p.c (permutation topology in this case). $G$ is $\tau_p$-representable in $X$, $b _m X=\beta_G X$~\cite[Proposition 5.4]{KozlovSorin2025}. There are five $G$-compactifications of $X$:  
$$\begin{array}{ccclccc}
& & b^{\leftrightarrow}X=  &  & &  & \\
& \nearrow & =b_m X/_{\inf\sim\sup} &  \searrow    & & & \\
b_m X &   &   &  & b^{+}X=  & \rightarrow  &  \alpha X \\
 & \searrow & b^{\updownarrow}X= &  \nearrow  & =b_m X/_{(x, -1)\sim (x, 1),\ x\in X,\ \inf\sim\sup}  & & \\
& & =b_m X/_{(x, -1)\sim (x, 1),\ x\in X}  &  & &  & 
\end{array}$$

The maps $b_m X\to b^{\leftrightarrow}X$, $b^{\updownarrow}X\to  b^{+}X$ and $b^{+}X\  (b_m X,\ b^{\leftrightarrow}X,\  b^{\updownarrow}X)\to \alpha X$  of compactifications are elementary~\cite[Lemma 5.6]{KozlovSorin2025}. The corresponding poset of Ellis compactifications is 
$$\begin{array}{ccccccc}
& & e_{b^{\leftrightarrow}X} G  &  & &  &  \\
& \nearrow & &  \searrow    & & &  \\
e_{b_m X} G &   &   &  & e_{b^{+}X} G  & \rightarrow  &  e_{\alpha X} G,   \\
 & \searrow & e_{b^{\updownarrow}X} G &  \nearrow  &   & &  \\
\end{array}$$
$e_{ b^{\leftrightarrow}X} G=e_{b_m X} G/I^m_{\leftrightarrow}$, $I^m_{\leftrightarrow}=\{f\in e_{b_m X} G\ |\ f(x)=\inf\,\vee\,\sup,\ x\in b_m X\}$, 
$e_{ b^+ X} G=e_{b^{\updownarrow} X} G/I^{\updownarrow}_{+}$, $I^{\updownarrow}_{+}=\{f\in e_{b^{\updownarrow} X} G\ |\ f(x)=\inf\,\vee\,\sup,\ x\in b^{\updownarrow} X\}$~\cite[Theorem 5.23, Corollary 5.17]{KozlovSorin2025}.  $e_{b_m X} G\geq b_r G$, $b_r G=e_{b_m X} G$ iff  $X$ is a continuously ordered chain~\cite[Corollary 5.16]{KozlovSorin2025}. If $e_{b_m X} G>b_r G$, then $e_{b_m X} G>b_r G>e_{\alpha X} G$ and $b_r G$ is incomparable with $e_{b^{\leftrightarrow}X} G$, $e_{b^{\updownarrow}X} G$ and $ e_{b^{+}X} G$~\cite[Theorem 5.15, Theorem 5.23]{KozlovSorin2025}. 

In~\cite[Theorem 5.15]{KozlovSorin2025} it is proved that $G$ is Roelcke precompact, $b_r G\leq e_{b_m X} G$, there are no Ellis compactifications $b G$ of $G$ such that $b_r G< b G<e_{b_m X} G$, and $e_{b_m X} G$ is the set of self-maps $f$ of $b_m X$  in the topology of pointwise convergence such that  
\begin{itemize}
\item[{\rm (i)}] $f$ is monotone {\rm (}if $x<y$, then $f(x)\leq f(y)${\rm)}, 
\item[{\rm (ii)}] $f(x)\not\in X$ if $x\in b_m X\setminus X$, 
\item[{\rm (iii)}] if $f(x)=f(y)$,  $x\ne y$, then $f(x)\in b_m X\setminus X$, 
\item[{\rm (iv)}] either $f((x, -1))=f((x, 0))=f((x, 1))\not\in X$, or $f((x, -1))=(y, -1)$, $f((x, 0))=(y, 0)$,  $f((x, 1))=(y, 1)$  and 
\item[{\rm (v)}] $f(\inf)=\inf$,  $f(\sup)=\sup$.
\end{itemize}

$e_{\alpha X} G$ is a sim-compactification of $G$ ((algebraically) isomorphic to the inverse monoid $J_X$ of all partial automorphisms of $X$)~\cite[Theorem 5.20]{KozlovSorin2025}. Remark~\ref{remdiscrcomp} yields that $e_{\alpha X} G\leq e_{\rm B} G$, where ${\rm B}$ is a unit ball in $\ell^2(X)$. 


\subsection{Graph compactifications of $(\aut (X), \tau_{\partial})$}\label{grcompautGO}

There are seven locally compact $G$-extensions of $X$:  $Y^-=b_m X\setminus\{\sup\}$, $Y^+=b_m X\setminus\{\inf\}$,  $Y=b_m X\setminus\{\inf, \sup\}$, $Z^-=b^{\updownarrow}X\setminus\{\sup\}$, $Z^+=b^{\updownarrow}X\setminus\{\inf\}$,  $Z=b^{\updownarrow}X\setminus\{\inf, \sup\}$ and $X$. 
The order on locally compact extensions of $X$ is the following:
$$\begin{array}{ccccc}
   &  & b_m X  &   & \\
   &  \  \swarrow      & \downarrow & \searrow  & \\
 &  Y^-   & b^{\updownarrow}X  &  &  Y^+, \\
 & \quad \searrow  \  \ \swarrow  &   \downarrow   &  \searrow \  \swarrow  & \\
 & \quad  Z^- &  b^{+}X  & \quad Z^+ & \\
     & \quad \searrow      &   \downarrow      &  \swarrow  & \\
  &    &  Z &  &  \\
\end{array}
\begin{array}{ccccc}
   &  & b_m X  &   & \\
   &  \  \swarrow      & \downarrow & \searrow  & \\
 &  Y^-   &  b^{\leftrightarrow}X  &  &  Y^+, \\
 & \quad \searrow   &   \downarrow   &  \swarrow  & \\
 &  & Y  & & \\
\end{array}
\begin{array}{ccc}
  b^{\leftrightarrow}X  & \rightarrow  & Y \\
 \downarrow  &        &  \downarrow \\
b^{+}X  & \rightarrow  & Z \\
\downarrow   &   &  \downarrow \\
\alpha X  & \rightarrow & X. \\
\end{array}
$$
By Lemma~\ref{lemcompemb} and Theorem~\ref{mapcompV}  the correspondent graph compactifications of $G$ are defined. 

\medskip

{\it Graph compactification $b_{b_m X} G$.} Let $\i^{\Gamma}_{b_m X}: G\to (2^{b_m X\times b_m X}, \tau_F)$, $\i^{\Gamma}_{b_m X} (g)=\{(x, g(x))\ |\ x\in b_m X\}$ be an embedding of $G$, $\pr_1$ and $\pr_2$ are projections of $b_m X\times b_m X$ on the first and the second factor respectively, $(b_m X\times b_m X, \leq_{\times})$ is a {\it cardinal product} ($(x, y)\leq (x', y')$ iff $x\leq x'$\ \&\ $y\leq y'$)~\cite[Ch.~1, item 7]{Birkhoff}. 

For $A\in 2^{b_m X\times b_m X}$ and $t\in\pr_1 (A)$, put 
$$A^-_t=\inf\{\tau\ |\ (t, \tau)\in A\},\ A^+_t=\sup\{\tau\ |\ (t, \tau)\in A\}.$$

\begin{pro}\label{graphdescrlotbeta}
The graph compactification $b_{b_m X} G$ is the set of linearly ordered {\rm(}the restriction of $\leq_{\times}${\rm)} compacta $A\in\CL (b_m X\times b_m X)$ such that 
\begin{itemize}
\item[(a)] $\pr_1 (A)=\pr_2 (A)=b_m X$, and, hence,
\begin{itemize}
\item  $A^-_{\inf}=\inf$,  $A^+_{\sup}=\sup$, 
\item $[A^-_t, A^+_t]\in A$,  $t\in b_m X$,
\item $A^-_t\leq A^+_t\leq A^-_{t'}\leq A^+_{t'}$, $t<t'$, $t, t'\in b_m X$,
\item if $A^+_t=A^-_{t'}$, $t<t'$, then $A^-_s=A^+_{s}$, $\forall\ t\leq s\leq t'$, $t, t'\in b_m X$,
\end{itemize}
\item[(b)]  $A^-_x=A^+_{x}=y$, $x\in X$, and, additionally, 
\begin{itemize}
\item if $y\in X$, then $A^+_{x^-}=y^-$,  $A^-_{x^+}=y^+$, 
\item if $y\in b_m X\setminus X$, then $A^+_{x^-}=A^-_{x^+}=y$, 
\end{itemize}
\item[(c)]  $A^-_t,\ A^+_{t}\not\in X$, $t\in b_m X\setminus X$. 
\end{itemize}
\end{pro}

\begin{proof} 
Take $A\in b_{b_m X} G=\cl\ \i^{\Gamma}_{b_m X} (G)$, then $\pr_1 (A)=\pr_2 (A)= b_m X$ by Lemma~\ref{proj}.  

If there exist $(x, y), (x', y')\in A$ such that $x<x'$, $y>y'$, then take  $x<t<x'$, $y'<\tau<y$. The set $W=(\pr_1^{-1}(\gets, t)\cap\pr_2^{-1}(\tau, \to))^-\cap (\pr_1^{-1}(t, \to)\cap\pr_2^{-1}(\gets, \tau))^-$ is an open nbd of $A$ and $W\cap\i^{\Gamma}_{b_m X} (G)=\emptyset$. Hence, $A\not\in\cl (i^{\Gamma}_{b_m X} (G))$. The case of $x>x'$, $y<y'$ is verified similarly. Therefore, ${\leq_{\times}}|_A$ is a linear order on $A$. 

The conditions of (a) are valid. 

\medskip

(b) If $\exists\ (x, y_1), (x, y_2)\in A$, $x\in X$, $y_1\ne y_2$, then take open disjoint nbds $U_1$ and $U_2$ of $y_1$ and $y_2$ in $b_m X$ respectively. Then the set $W=(\pr_1^{-1}(x)\cap\pr_2^{-1}U_1)^-\cap (\pr_1^{-1}(x)\cap\pr_2^{-1}U_2)^-$ is an open nbd of $A$ and $W\cap\i^{\Gamma}_{b_m X} (G)=\emptyset$. Hence, $A\not\in\cl (i^{\Gamma}_{b_m X} (G))$.

If $y\in X$, then $A^-_{x^+}=y^+$. Indeed, if $A^-_{x^+}=\tau>y^+$, then the set $W=((x, y))^-\cap (\pr_1^{-1}((x^+, \to))\cap\pr_2^{-1}((y^+, \to)))^-$ is an open nbd of $A$ and $W\cap\i^{\Gamma}_{b_m X} (G)=\emptyset$. Hence, $A\not\in\cl (i^{\Gamma}_{b_m X} (G))$. The case  $A^+_{x^-}=y^-$ is examined similarly. 

If $y\in b_m X\setminus X$, then $A^-_{x^+}=y$. Indeed, if $A^-_{x^+}>y$, then $\pr_2 (A)\ne b_m X$ if $y=z^+$, $z\in X$, or $y\in\Gamma$ and $A\not\in\cl (i^{\Gamma}_{b_m X} (G))$. If $y=z^-$, $z\in X$, then the set $W=(\{x\}\times (z^-, \to))^+\cap ((x, \to)\times (\gets, z))^+$ is an open nbd of $A$. $\forall\ g\in \i^{\Gamma}_{b_m X} (G)$ if $g(x)<z$, then $g(x^+)< z^-$ and, therefore, $W\cap\i^{\Gamma}_{b_m X} (G)=\emptyset$. Hence, $A\not\in\cl (i^{\Gamma}_{b_m X} (G))$. The case  $A^+_{x^-}=y$ is examined similarly. 

\medskip

(c) If $A^+_{x^+}=z\in X$, then $(x^+, z^+)\not\in A$. $\exists\ t>x^+$ such that $(t, z^+)\in A$ by (a). Due to the remark below $\forall\ y\in (x^+, t]$ $(y, z^+)\in A$. Hence, $(x^+, t]\times\{z^+\}\in A$, and $(x^+, z^+)\in A$. A contradiction is obtained. All the other cases are examined similarly. 

\medskip

Take $A\in\CL (b_m X\times b_m X)$ such that $(A, {\leq_{\times}}|_A)$ is a linearly ordered compacta, $\pr_1 (A)=\pr_2 (A)=b_m X$ and conditions (a) --- (c) are valid. Any nbd of $A$ in $(2^{b_m X\times b_m X}, \tau_F)$ is of the form $[V_1, \ldots, V_{k}]$. Additionally, since $\dim A=0$, $b_m X\times b_m X$ is a product,  one can consider (see, \cite{KozlovSorin2025}) that $V_i=O_i\times U_i$, $O_i$, $U_i$ are elements of the cover 
$$u=\{(\gets, x_1),\ \{x_1\},\ (x_1, x_2),\ \{x_2\},\ \ldots,\ (x_{n-1}, x_n),\ \{x_n\},\ (x_n, \to)\},$$ $x_1,\ldots, x_n\in X$, $x_1<\ldots<x_n$, $i=1,\ldots, k$, and 
$V_1=(\gets, x_1)\times (\gets, x_1)$, $V_k=(x_n, \to)\times (x_n, \to)$, and $V_i<V_{i+1}$ (order on $\{V_1, \ldots, V_{k}\}$) if either (i) $O_i=O_{i+1}$, $U_i<U_{i+1}$ are adjacent elements of the cover $u$, or (ii) $O_i<O_{i+1}$ are adjacent elements of the cover $u$, $U_i=U_{i+1}$, or (iii) $O_i<O_{i+1}$,  $U_i<U_{i+1}$ are adjacent elements of the cover $u$, $i=1,\ldots, k-1$. 

{\it The construction of $g\in i^{\Gamma}_{b_m X} (G)\cap [V_1, \ldots, V_{k}]$.} By (b) $y_i=A\cap\pr_1^{-1} (x_i)$, $y_i\in U_i$, $i=1,\ldots, n$, are correctly defined,  and $y_1\leq\ldots\leq y_n$. Due to invariance of $b_{b_m X} G=\cl\ \i^{\Gamma}_{b_m X} (G)$ with respect to the self-inverse map $S_{b_m X}$ correspondent to the symmetry $s_{b_m X}$ on $b_m X\times b_m X$, if $y_i=y_{i+1}$ for some $i=1,\ldots, n-1$, then $y_i\in b_m X\setminus X$. Choose $z_1,\ldots, z_n\in X$ such that $z_i\in U_i$, $i=1,\ldots, n$, $z_1<\ldots<z_n$. Ultrahomogeneity of $X$ yields the existence of $g\in G$ such that $g(x_i)=z_i$, $i=1,\ldots, n$. 
Obviously $\i^{\Gamma}_{c_m X}  (g)\in [V_1, \ldots, V_n]$ and, hence, $A\in\cl\ \i^{\Gamma}_{c_m X} (G)$. 
\end{proof}

\begin{rem}
{\rm Due to invariance of $b_{b_m X} G=\cl\ \i^{\Gamma}_{b_m X} (G)$ with respect to the self-inverse map $S_{b_m X}$ (correspondent to the 
symmetry $s_{b_m X\times b_m X}$ on $b_m X\times b_m X$), for $A\in b_{b_m X} G$, and the sets 
$${\bar A}^-_t=\inf\{\tau\ |\ (\tau, t)\in A\},\ {\bar A}^+_t=\sup\{\tau\ |\ (\tau, t)\in A\},\ \tau\in b_m X,$$
the conditions analogous to (a) --- (c) are valid.}
\end{rem}

\medskip

{\it Graph compactification $b_{b^{\updownarrow}X} G$.}

By Theorem~\ref{mapcompVel} the graph compactification $b_{b^{\updownarrow}X} G$ is the image of $b_{b_m X} G$ under the map ${F}^H$, where $F=f\times f$, $f: b_m X\to b^{\updownarrow}X$ is the map of compactifications of $X$.

The map $F:  b_m X\times b_m X\to b^{\updownarrow}X\times b^{\updownarrow}X$ induces an equivalence relation on $b_m X\times b_m X$:
$$(x_1, y_1)\sim_F (x_2, y_2)\ \Longleftrightarrow\ \left\{\begin{array}{l} x_1=x_2,\ y_1=y_2, \\ 
\mbox{or}\ x_1=x_2,\ y_1=y^-, y_2=y^+, y\in X, \\ 
\mbox{or}\ x_1=x^-, x_2=x^+,\ x\in X,\ y_1=y_2, \\ 
\mbox{or}\ x_1=x^-, x_2=x^+,\ x\in X,\ y_1=y^-, y_2=y^+, y\in X, \\ 
\mbox{or}\  x_1=x^-, x_2=x^+,\ x\in X,\ y_1=y^+, y_2=y^-, y\in X. \\
\end{array}\right.$$

The equivalence classes are either one-point ($(x, y)$, $x, y\in X\cup\Gamma$), two-point ($(x, y)$, either $x\in X\cup\Gamma$, $y\in b_m X\setminus ( X\cup\Gamma)$, or $y\in X\cup\Gamma$, $x\in b_m X\setminus ( X\cup\Gamma)$), or four-point sets ($(x, y)$, $x\in b_m X\setminus ( X\cup\Gamma)$, $y\in b_m X\setminus ( X\cup\Gamma)$). 

The straightforward checking with the use of Proposition~\ref{graphdescrlotbeta} and density of $X$ in $b_m X$ proves the following proposition.

\begin{pro}\label{graphdescrlotbupdownarrow}
The graph compactification $b_{b^{\updownarrow}X} G$ is the quotient space of $b_{b_m X} G$ with respect to the equivalence relation
$$A\sim_{\updownarrow} B\ \Longleftrightarrow\ \left\{
\begin{array}{ll}
\forall\ x\in X & \mbox{either}\ A^-_x=A^+_x=B^-_x=B^+_x, \\
& \mbox{or}\ \exists\ y\in X\ \mbox{such that}\ A^-_x=A^+_x=y^-, B^-_x=B^+_x=y^+ \\
& (\mbox{or}\ A^-_x=A^+_x=y^+, B^-_x=B^+_x=y^-) \\
& \mbox{and}\ A^-_{x^-}=B^-_{x^-},\  A^+_{x^+}=B^+_{x^+}, \\
\forall\ v\in\Gamma &  A^-_v=B^-_v,\ A^+_v=B^+_v. \\
\end{array}\right.$$
\end{pro}

\medskip

{\it Graph compactification $b_{b^{\leftrightarrow}X} G$.}

\begin{lem}\label{lbetabchains}
For $K=\{\inf, \sup\}$ the condition {\rm(ER)} of {\rm Theorem~\ref{restrmap}} for $b_{b_m X} G$ is valid.
\end{lem}

\begin{proof}
Take $A, B\in b_{b_m X} G$ such that $A\cap (Y\times Y)=B\cap (Y\times Y)\ne\emptyset$, $Y=b_m X\setminus K$.

For the projections of $A$ (or $B$) on the factors there are four possible cases (take into consideration Proposition~\ref{graphdescrlotbeta}):
\begin{itemize}
\item[{\rm (1)}] $\pr_1 (A\cap (Y\times Y))=Y$, 

\item[{\rm (2)}] $\pr_1 (A\cap (Y\times Y))=(\gets, b)$, $b\in X^-\cup (\Gamma\setminus\{\inf, \sup\})$, or $(\gets, b]$, $b\in X^+\cup  (\Gamma\setminus\{\inf, \sup\})$,

\item[{\rm (3)}] $\pr_1 (A\cap (Y\times Y))=(a, \to)$, $a\in X^+\cup  (\Gamma\setminus\{\inf, \sup\})$, or $[a, \to)$, $a\in X^-\cup  (\Gamma\setminus\{\inf, \sup\})$,
 
\item[{\rm (4)}] $\pr_1 (A\cap (Y\times Y))=[a, b]$, $a\in X^-\cup  (\Gamma\setminus\{\inf, \sup\})$,  $b\in X^+\cup  (\Gamma\setminus\{\inf, \sup\})$, 
or $(a, b]$, $a\in X^+\cup  (\Gamma\setminus\{\inf, \sup\})$, $b\in X^+\cup  (\Gamma\setminus\{\inf, \sup\})$, or $[a, b)$, $a\in X^-\cup  (\Gamma\setminus\{\inf, \sup\})$, $b\in X^-\cup  (\Gamma\setminus\{\inf, \sup\})$, or $(a, b)$, $a\in X^+\cup  (\Gamma\setminus\{\inf, \sup\})$, $b\in X^-\cup  (\Gamma\setminus\{\inf, \sup\})$, $a\leq b$. 
\end{itemize}

Let $a'$ be the least lower bound and $b'$ is the greatest upper bound of $\pr_2 (A\cap (Y\times Y))$. 

In the case (1) $A=(\{\inf\}\times (\gets, a'])\cup (A\cap (Y\times Y)\cup (\{\sup\}\times [b', \to))$. 

In the case (2) $b'=\sup$ and $A=(\{\inf\}\times (\gets, a'])\cup (A\cap (Y\times Y)\cup ([b, \to)\times\{\sup\})$. 

In the case (3) $a'=\inf$ and $A=((\gets, a]\times\{\inf\})\cup (A\cap (Y\times Y)\cup  (\{\sup\}\times [b', \sup))$. 

In the case (4)  $a'=\inf$,  $b'=\sup$ and $A=((\inf, a]\times\{\inf\})\cup (A\cap (Y\times Y)\cup  ([b, \sup)\times\{\sup\})$. The same is true for $B$ and $A=B$.
\end{proof}

Note that $A_{\inf}=(b_m X\times\{\inf\})\cup (\{\sup\}\times b_m X)$ and $A_{\sup}=(\{\inf\}\times b_m X)\cup(b_m X\times\{\sup\})\}$ are the only sets $R\in  b_{b_m X} G$ such that $R\cap (Y\times Y)=\emptyset$.

\begin{pro}\label{leftrightarrow}
$$b_{b^{\leftrightarrow}X} G=b_{b_m X} G/0_K,\ 0_K=\{A_{\inf}, A_{\sup}\}.$$
\end{pro}

\begin{proof}
From Lemma~\ref{lbetabchains} by Remark~\ref{corrcond} the condition (EPM) of Theorem~\ref{mapcompVel} holds for $b_{b_m X} G$ and the map of compactifications $b_{b_m X} G$, $b_{b^{\leftrightarrow}X} G$ which is elementary. The rest follows from Theorem~\ref{mapcompVel}. 
\end{proof}

\medskip

{\it Graph compactification  $b_{b^{+}X} G$.} Let $\pr_1$ (respectively $\pr_2$) be projection of $b^{\updownarrow}X\times b^{\updownarrow}X$ onto the first  (respectively the second) factor. 

For $R\in b_{b^{\updownarrow}X} G$ the following hold:
\begin{itemize}
\item[(a)] $\{R\in b_{b^{\updownarrow}X} G\ |\ R\cap (Z\times Z)=\emptyset\}=\{F^H(A_{\inf}),\ F^H(A_{\sup})\}$, 
\item[(b)] $\{R\in b_{b^{\updownarrow}X} G\ |\ |\pr_1 (R\cap (Z\times Z))|=1\}=\{F^H(A_t=((\gets, t]\times\{\inf\})\cup (\{t\}\times b_m X)\cup ([t, \to)\times\{\sup\})),\ t\in Z\setminus X\}$. 
\end{itemize}

\begin{lem}\label{lbetabchains1}
For $K=\{\inf, \sup\}$ the condition {\rm(ER)} of {\rm Theorem~\ref{restrmap}} for $b_{b^{\updownarrow}X} G$ 
is valid for $R\in b_{b^{\updownarrow}X} G$ such that $R\ne F^H(A_{t})$, $t\in X^-\cup X^+$. 
\end{lem}

\begin{proof}
From the definition of map $F^H$ if $R\in b_{b^{\updownarrow}X} G$, then $\forall\ A\in  b_{b_m X} G$, $F^H (A)=R$, one has 
$$\pr_1  (R\cap (Z\times Z))=f(\pr'_1 (A\cap (Y\times Y))),\ \mbox{where}\ \pr'_1\ \mbox{is the projection of}\ b_m X\times b_m X\ \mbox{onto the first   factor}.$$ 

The case $\pr_1 (R\cap (Z\times Z))=\{t\}$. If $t\in f(\Gamma\setminus\{\inf, \sup\})$, then $f^{-1}(t)=v\in (\Gamma\setminus\{\inf, \sup\})$ and $(F^H)^{-1} R=((\gets, v]\times\{\inf\})\cup (\{v\}\times b_m X)\cup ([v, \to)\times\{\sup\})$. Therefore, the  condition {\rm(ER)} for $R$ is valid. 

If $t\in f(X^-\cup X^+)$, then $f^{-1}(t)=\{x^-, x^+\}$ for some $x\in X$, and $(F^H)^{-1} R=\{A_{x^-}\cup A_{x^+}\}$. Since $F^H (A_{x^-})\ne F^H (A_{x^+})$, the condition {\rm(ER)} for $R$ is not valid. 

By Proposition~\ref{graphdescrlotbeta} $t$ can't be in $f(X)$.

The case $|\pr_1 (R\cap (Z\times Z))=R_1|>1$. $f^{-1} R_1$ is an interval in $Y$ and by Lemma~\ref{lbetabchains} for any $A\in (F^H)^{-1} R$ the condition {\rm(ER)} is valid. Moreover, due to the construction in Lemma~\ref{lbetabchains} $A\setminus (Y\times Y)=B\setminus (Y\times Y)$ for all $A, B\in (F^H)^{-1} R$. 
\end{proof}

By Theorem~\ref{mapcompVel} the graph compactification $b_{b^{+}X} G$ is the image of $b_{b^{\updownarrow}X} G$ under the map induced by an   elementary map of compactifications $b^{\updownarrow}X$, $b^{+}X$ of $X$ (identification of points $\inf$ and $\sup$). 

\begin{cor}\label{corxxx}
The graph compactification $b_{b^{+}X} G$ is the quotient space of $b_{b^{\updownarrow} X} G$ with respect to the equivalence relation 
which two-points sets equivalence classes are $\{F^H(A_{\inf}),\ F^H(A_{\sup})\}$, $\{F^H(A_{x^-}), F^H(A_{x^+})\}$, $x\in X$, all other equivalence classes are one-point sets. 
\end{cor}

\medskip

{\it Graph compactification  $b_{\alpha X} G$.} By Theorem~\ref{mapcompVel}  the graph compactification $b_{\alpha X} G$ is the image of $b_{b_m X} G$  under the map induced by an elementary map of compactifications $b_m X$, $\alpha X$ of $X$.

\begin{pro}\label{alphaX}
The graph compactification $b_{\alpha X} G$ is the quotient space of $b_{b_m X} G$ with respect to the equivalence relation
$$A\sim_{\alpha} B\ \Longleftrightarrow\ \forall\ x\in X\ \mbox{if}\ A^-_x(=A^+_x),\ B^-_x(=B^+_x)\in X,\ \mbox{then}\ A^-_x=B^-_x$$
and the condition {\rm(PM)} of {\rm Theorem~\ref{mapcompVel}} is valid.  
\end{pro}

\begin{proof} 
It is easy to check that $\sim_{\alpha}$ is an equivalence relation on $b_{b_m X} G$, $A\sim_{\alpha} B$ iff the images of $A, B\in b_{b_m X} G$ under the map of compactifications $b_{b_m X} G$, $b_{\alpha X} G$ of $G$ conicide. By Lemma~\ref{lbetabchains} the condition (PM) of Theorem~\ref{mapcompVel} is valid.  
\end{proof}

\begin{thm}\label{Prodescrlotbchain} 
\begin{itemize}
\item[(a)] The map of compactifications $b_{b_m X} G$, $b_{b^{\leftrightarrow}X} G$ of $G$ is  elemantary, 
$$b_{b^{\leftrightarrow}X} G=b_{b_m X} G/0_K,\ 0_K=\{A_{\inf}, A_{\sup}\}.$$
\item[(b)] $b_{Y^-}G= b_{Y^+}G=b_{b_m X} G$, 
\item[(c)] $b_{Z^-}G= b_{Z^+}G= b_{b^{\updownarrow}X} G$, 
\item[(d)] $b_{Y} G=b_{b^{\leftrightarrow}X} G$,
\item[(e)] $b_{Z} G=b_{b^{+}X} G$,
\item[(f)] $b_{X} G=b_{\alpha X} G$.
\end{itemize}
\end{thm}

\begin{proof}
(a)  is Proposition~\ref{leftrightarrow}. (b) is a consequence of Lemma~\ref{lbetabchains}. (c) follows from Lemma~\ref{lbetabchains1} and Theorem~\ref{restrmap}. (d) follow from Lemma~\ref{lbetabchains} and Theorem~\ref{restrmap}. (e) follows from Corollary~\ref{corxxx} and Theorem~\ref{restrmap}. (f)  follows from Theorem~\ref{suffcond}. 
\end{proof}


\subsection{Order on compactifications of $(\aut (X), \tau_{\partial})$.}

\begin{thm}\label{thmOrderchain} Let $X$ be an ultrahomogeneous chain, $G=(\aut (X), \tau_{\partial})$.
\begin{itemize}
\item[I.] $$\begin{array}{ccccccc}
& & b_{b^{\leftrightarrow}X} G  &  & &  &  \\
& \nearrow & &  \searrow    & & &  \\
b_r G = b_{b_m X} G &   &   &  & b_{b^{+}X} G  & \rightarrow  &  b_{\alpha X} G.  \\
 & \searrow & b_{b^{\updownarrow}X} G &  \nearrow  &   & &  \\
\end{array}$$
\item[II.] 
\begin{itemize}
\item[{\rm(a)}]  $b_{b_m X} G\leq e_{b_m X} G$. $b_{b_m X} G=e_{b_m X} G$ iff $\Gamma=\{\inf, \sup\}$.
\item[{\rm(b)}]  $b_{b^{\leftrightarrow}X} G$ and $e_{b^{\leftrightarrow}X} G$ are incomparable if $\Gamma\ne\{\inf, \sup\}$. If $\Gamma=\{\inf, \sup\}$, then $b_{b^{\leftrightarrow}X} G>e_{b^{\leftrightarrow}X} G$.
\item[{\rm(c)}] $b_{b^{\updownarrow}X} G$ and $e_{b^{\updownarrow}X} G$  are incomparable if $\Gamma\ne\{\inf, \sup\}$.  If $\Gamma=\{\inf, \sup\}$, then  $b_{b^{\updownarrow}X} G>e_{b^{\updownarrow}X} G$.
\item[{\rm(d)}] $b_{b^{+}X} G$ and $e_{b^{+}X} G$  are incomparable if $\Gamma\ne\{\inf, \sup\}$. If $\Gamma=\{\inf, \sup\}$, then  $b_{b^{+}X} G>e_{b^{+}X} G$. 

\item[{\rm(e)}] $b_{\alpha X} G = e_{\alpha X} G$. 
\end{itemize}
\item[III.]  $b_{\alpha X} G=e_{\alpha X} G=e_{\rm B} G$, is a sim-compactification and a {\rm WAP}-compactification of $G$. $G$ is not a {\rm WAP} group.
\end{itemize}
\end{thm}

\begin{proof}
Inequalities in I follow from Propositions~\ref{graphdescrlotbupdownarrow}, \ref{leftrightarrow}, \ref{alphaX} and Corollary~\ref{corxxx}. 

The proof of the equality $b_{b_m X} G=b_r G$ is the same as in Theorem~\ref{thmOrderLOTS} below. 

II. (a) The proof of the inequality $b_{b_m X} G=b_r G\leq e_{b_m X} G$ is the same as in Theorem~\ref{thmOrderLOTS} below.  $b_{b_m X} G=b_r G= e_{b_m X} G$ iff $\Gamma=\{\inf, \sup\}$~\cite[Corollary 5.16]{KozlovSorin2025}. 

The first statement of (b) is analogous to the corresponding statement in Theorem~\ref{thmOrderLOTS} below. The second is due Proposition~\ref{Prodescrlotbchain} and the description of  $e_{b^{\leftrightarrow}X} G$ in~\cite{KozlovSorin2025}. 

(c). We have the following diagram
$$\begin{array}{ccc}
b_{b_m X} G  &  \longleftarrow & e_{b_m X} G  \\
\downarrow & & \downarrow \\
b_{b^{\updownarrow}X} G &  & e_{b^{\updownarrow}X} G \\
\end{array}$$
where arrows are maps of compactifications. 

Elements 
$$f^{x^-}(t)=\left\{
\begin{array}{ll}
\inf, & t=\inf, \\
x^-, & \inf< t<\sup, \\
\sup, & t=\sup,  \\
\end{array}\right. 
f^{x^+}(t)=\left\{
\begin{array}{ll}
\inf, & t=\inf, \\
x^+, & \inf< t<\sup, \\
\sup, & t=\sup,  \\
\end{array}\right. 
$$
where $x\in X$ are in $e_{b_m X} G$ and their images in $e_{b^{\updownarrow}X} G$ coincide~\cite[Lemma 3.7]{KozlovSorin2025} (the equivalence relation for the map $\Phi$ of compactifications $e_{b_m X} G$ and $e_{b^{\updownarrow}X} G$  induced by $f: b_m X\to b^{\updownarrow}X$ is
$$g\sim_{\Phi} h\ \Longleftrightarrow\ g(x)\sim_{f} h(x),\ x\in b_m X).$$

Elements 
$A^{x^-}=\{\inf\}\times [\inf, x^-]\cup b_m X\times\{x^-\}\cup \{\sup\}\times [x^-, \sup]$ and $A^{x^+}=\{\inf\}\times [\inf, x^+]\cup b_m X\times\{x^+\}\cup \{\sup\}\times [x^+, \sup]$ are correspondent images of $f^{x^-}$ and $f^{x^+}$ under the map of compactifications $e_{b_m X} G$ and $b_{b_m X} G$. And images of $A^{x^-}$ and $A^{x^+}$ under the map of compactifications $b_{b_m X} G$ and $b_{b^{\updownarrow}X} G$ are distinct by Proposition~\ref{graphdescrlotbupdownarrow}. Hence, the unequality $b_{b^{\updownarrow}X} G\leq e_{b^{\updownarrow}X} G$ doesn't hold. 

If $\Gamma\ne\{\inf, \sup\}$, take $v\in\Gamma\setminus\{\inf, \sup\}$ and consider the elements 
$$f_v^{\tau}(t)=\left\{\begin{array}{ll}
\inf, & \tau<v, \\
\tau, &  \tau=v, \\
\sup, &  \tau> v,  \\
\end{array}\right.$$
of $e_{b_m X} G$, where $\tau\in b_m X\setminus X$. The image of the set $\{f_v^{\tau}\ |\ \tau\in  b_m X\}$ is a one-pont set under the map of compactifications $e_{b_m X} G$ and $b_{b_m X} G$. Therefore, if $b_{b^{\updownarrow}X} G> e_{b^{\updownarrow}X} G$, then its image in $ e_{b^{\updownarrow}X} G$ is a one-point set. Since this doesn't hold the unequality $b_{b^{\updownarrow}X} G\geq e_{b^{\updownarrow}X} G$ doesn't hold. 

If $\Gamma=\{\inf, \sup\}$, then the equivalence relation on $e_{b_m X} G$ which determines $e_{b^{\updownarrow}X} G$ is stronger than the equivalence relation on $b_{b_m X} G=e_{b_m X} G$ which determines $b_{b^{\updownarrow}X}$ (elements in $e_{b_m X} G$ are uniquely determined in points of $X$). 
Hence, $b_{b^{\updownarrow}X} G\geq e_{b^{\updownarrow}X} G$.

The images of elements 
$$f_{x}^{y^- z^-}(t)=\left\{
\begin{array}{ll}
\inf, & t<x^-, \\
y^-, & x^-\leq t\leq x^+, \\
z^-, & x^+< t<\sup, \\
\sup, & t=\sup,  \\
\end{array}\right. 
f_{x}^{y^+ z^+}(t)=\left\{
\begin{array}{ll}
\inf, & t<x^-, \\
y^+, & x^-\leq t\leq x^+, \\
z^+, & x^+< t<\sup, \\
\sup, & t=\sup,  \\
\end{array}\right. 
$$
where $x, y, z\in X$ coincide under the map of compactifications $e_{b_m X} G$ and $e_{b^{\updownarrow}X} G$. The corresponding elements to $f_{x}^{y^- z^-}$ and $f_{x}^{y^+ z^+}$ in  $b_{b_m X} G$ are $A_{x}^{y^- z^-}=[\inf, x^-]\times\{\inf\}\cup\{x^-\}\times [\inf, y^-]\cup [x^-, x^+]\times\{y^-\}\cup\{x^+\}\times [y^-, z^-]\cup[x^+, \sup]\times\{z^-\}\cup\{\sup\}\times [z^-, \sup]$ and $A_{x}^{y^+ z^+}=[\inf, x^-]\times\{\inf\}\cup\{x^-\}\times [\inf, y^+]\cup [x^-, x^+]\times\{y^+\}\cup\{x^+\}\times [y^+, z^+]\cup[x^+, \sup]\times\{z^+\}\cup\{\sup\}\times [z^+, \sup]$. The images of  $A_{x}^{y^- z^-}$ and  $A_{x}^{y^+ z^+}$ under  the map of compactifications $b_{b_m X} G$ and $b_{b^{\updownarrow}X} G$ are distinct. 
Hence, $b_{b^{\updownarrow}X} G> e_{b^{\updownarrow}X} G$.

(d). Follows from (c) and Corollary~\ref{corxxx}.

(e). The equivalence relation for the map $\Psi$ of compactifications $e_{b_m X} G$ and $e_{\alpha X} G$ induced by $f': b_m X\to \alpha X$ is
$$g\sim_{\Psi} h\ \Longleftrightarrow\ g(x)\sim_{f'} h(x),\ x\in b_m X\ \Longleftrightarrow\ \forall\ x\in X\ \mbox{either}\ g(x)=h(x)\in X,\ \mbox{or}\ g(x), h(x)\in b_m X\setminus X$$
and coincides with composition of the equivalence relation unduced by the map of compactifications $e_{b_m X} G$ and $b_{b_m X} G$ and $\sim_{\alpha}$ on $b_{b_m X} G$ from Proposition~\ref{alphaX}.

\medskip

III. Since $e_{\alpha X} G$ is an sim-compactification of $G$~\cite{KozlovSorin2025} and $b_r G\leq e_{b_m X} G$, for the WAP-compactification ${\rm WAP}\, G$ of $G$ the following holds: 
$$e_{b_m X} G\geq {\rm WAP}\, G\geq e_{\alpha X} G,$$
and the maps $F:e_{b_m X} G\to {\rm WAP}\, G$, $T: e_{b_m X} G\to e_{ \alpha X} G$ which are homomorphisms of semigroups  are defined~\cite{KozlovLeiderman2025}. $0_{e_{\alpha X} G}$ (the map ${\rm const}_{\infty}:e_{\alpha X} G\to\{\infty\}$) is the zero (ideal) of $e_{\alpha X} G$.

\medskip

A. {\it $F(T^{-1}(0_{e_{\alpha X} G})=\{f\in e_{b_m X} G\ |\ f\in T^{-1} ({\rm const}_{\infty})\})$ is a zero $0$ {\rm(}ideal{\rm)} of ${\rm WAP}\, G$.} 

$\forall\ f\in I^m_{\leftrightarrow}$ is a right-zero (or left ideal) of $e_{b_m X} G$. Therefore, $F(f)$ is a right-zero in ${\rm WAP}\, G$.  Moreover, since $f_{\inf}$, $f_{\inf}(\inf)=\inf$, $f_{\inf}(t)=\sup$, $t>\inf$, is a relative left-zero for homeomorphisms from $G$ and ${\rm WAP}\, G$ is a semitopological semigroup, $F(f_{\inf})=0$. Hence, $\forall\ f\in I^m_{\leftrightarrow}$ $F(f)=F(f_{\inf})=0$.

\medskip

1. {\it For elementary maps $f^{\tau}_{x_1^+, x_2^+}$, $F(f^{\tau}_{x_1^+, x_2^+})=0$.}

Elementary maps 
$$f^{\tau}_{x_1^+, x_2^+}(t)=\left\{
\begin{array}{ll}
\inf, & t\leq x_1^+, \\
\tau, & x_1^+< t\leq x_2^+, \\
\sup, & x_2^+<t,  \\
\end{array}
\right. x_1<x_2\in X,\ \tau\in b_m X\setminus X,$$
are in $T^{-1}(0_{e_{\alpha X} G})$. 
The set of maps $F_{x_1^+, x_2^+}=\{f^{\tau}_{x_1^+, x_2^+}\ |\ \tau\in b_m X\setminus X\}$ is a closed subsemigroup of $e_{b_m X} G$, 
and maps 
$$f_x(t)=\left\{
\begin{array}{ll}
\inf, & t\leq x^+, \\
\sup, & x<t,  \\
\end{array}
\right. x\in X,$$
from $I^m_{\leftrightarrow}$ are in $\bigcup\limits_{x_1<x_2 \in X}F_{x_1^+, x_2^+}$. 
$$f^{\tau}_{x_1^+, x_2^+}\circ f^{\tau'}_{x_1^+, x_2^+}=\left\{
\begin{array}{ll}
f_{x_2}, & \tau'\leq x_1^+, \\
f^{\tau}_{x_1^+, x_2^+}, & x_1^+<\tau'\leq x_2^+, \\
f_{x_1}, & \tau'>x_2^+, \\
\end{array}
\right. \tau,\ \tau'\in b_m X\setminus X.$$
$$F(f^{\tau}_{x_1^+, x_2^+}\circ f^{\tau'}_{x_1^+, x_2^+})=F(f^{\tau}_{x_1^+, x_2^+}) F(f^{\tau'}_{x_1^+, x_2^+})=\left\{
\begin{array}{ll}
 0, & \tau'\leq x_1^+, \tau'>x_2^+, \\
F(f^{\tau}_{x_1^+, x_2^+}), &  x_1^+<\tau'\leq x_2^+. \\
\end{array}
\right. $$
Further, any nbd of $f^{x_2^+}_{x_1^+, x_2^+}=f_{x_2}$ contains an element $f^{\tau'}_{x_1^+, x_2^+}$ such that $F(f^{\tau}_{x_1^+, x_2^+}) F(f^{\tau'}_{x_1^+, x_2^+})=0$. Hence, since ${\rm WAP}\, G$ is a semitopological semigroup,  $F(f^{\tau}_{x_1^+, x_2^+})=0$. 

\medskip

1 (a). The same aurguments show that for the maps $f\in T^{-1}(0_{e_{\alpha X} G})$ of the form 
$$f^{\tau}_{x_1^-, x_2^+}(t)=\left\{
\begin{array}{ll}
\inf, & t< x_1^-, \\
\tau, & x_1^+\leq t\leq x_2^+, \\
\sup, & x_2^+<t,  \\
\end{array}
\right. x_1\leq x_2\in X,\ \tau\in b_m X\setminus X,$$ or
$$f^{\tau}_{x_1^+, x_2^-}(t)=\left\{
\begin{array}{ll}
\inf, & t\leq x_1^+, \\
\tau, & x_1^+< t< x_2^+, \\
\sup, & x_2^-\leq t,  \\
\end{array}
\right. x_1< x_2\in X,\ \tau\in b_m X\setminus X,$$
or
$$f^{\tau}_{x_1^-, x_2^-}(t)=\left\{
\begin{array}{ll}
\inf, & t< x_1^-, \\
\tau, & x_1^-\leq t< x_2^-, \\
\sup, & x_2^-\leq t,  \\
\end{array}
\right. x_1< x_2\in X,\ \tau\in b_m X\setminus X,$$
$F(f)=0$. 

\medskip

2. {\it For step maps $f^{s}_{\sigma}$, $F(f^{s}_{\sigma})=0$.}

Take $\sigma=\{x_1, x_2,\ldots, x_n\}$, $x_1< x_2<\ldots<x_n$, $x_1,\ldots, x_n\in X$, $s=\{\tau_1,\ldots, \tau_{n-1}\}$, $\tau_1\leq \tau_2\leq\ldots\leq \tau_{n-1}$, $\tau_1,\ldots, \tau_{n-1}\in b_m X\setminus X$.
$$f^{s}_{\sigma}(t)=\left\{
\begin{array}{ll}
\inf, & t\leq x^+_1, \\
\tau_k, & x_k^+< t\leq x_{k+1}^+,\ k=1,\ldots, n-1, \\
\sup, & x_n^+<t.  \\
\end{array}
\right.$$
The set of maps $F_{\sigma}=\{f^{s}_{\sigma}\ |\ s=\{\tau_1,\ldots, \tau_{n-1}\},\ \tau_1\leq \tau_2\leq\ldots\leq \tau_{n-1},\ \tau_1,\ldots, \tau_{n-1}\in b_m X\setminus X\}$ is a closed subsemigroup of $e_{b_m X} G$. If $\sigma'\subset\sigma$, then  $F_{\sigma'}\subset F_{\sigma}$. If $|\sigma|=2$, then  $F_{\sigma}$ is a set of elementary maps and $F(F_{\sigma})=0$.
 
Using induction on cardinality of $\sigma$ we shall prove that $\forall\ \sigma$ $F(F_{\sigma})=0$. Assume that for $\forall\ \sigma'$, $|\sigma'|<n$ $F(F_{\sigma'})=0$. Take $\sigma$, $|\sigma|=n$, and $f^{s}_{\sigma}\in F_{\sigma}$, $\sigma=\{x_1, \ldots, x_n\}$,  $s=\{\tau_1, \ldots, \tau_{n-1}\}$. Put $S=\{\{\tau, x_3^+, \ldots, x_{n}^+\}\ |\ \tau\in b_m X\setminus X,\ \tau\leq x_3^+\}$. 
$$f^{s}_{\sigma}\circ f^{s'}_{\sigma}=\left\{
\begin{array}{ll}
f_{\sigma\setminus\{x_1\}}^{s\setminus\{\tau_1\}}, & \tau\leq x_1^+, \\
f^{s}_{\sigma} & x_1^+<\tau\leq x_2^+, \\
f_{\sigma\setminus\{x_2\}}^{s\setminus\{\tau_1\}}, & x_2^+<\tau\leq x_3^+, \\
\end{array}
\right. s'\in S.$$
$$F(f^{s}_{\sigma}\circ f^{s'}_{\sigma})=F(f^{s}_{\sigma}) F(f^{s'}_{\sigma})=\left\{
\begin{array}{lll}
0, & \tau\leq x_1^+,\   x_2^+<\tau\leq x_3^+ & \mbox{by inductive assumption},\\
F(f^{s}_{\sigma}), & x_1^+<\tau\leq x_2^+. & \\
\end{array}
\right. $$
Further, any nbd of $f^{s'}_{\sigma}$, $s'=\{x_2^+, x_3^+, \ldots, x_{n}^+\}$, contains an element $f^{s''}_{\sigma}$ such that $F(f^{s}_{\sigma}) F(f^{s''}_{\sigma})=0$. Hence, since ${\rm WAP}\, G$ is a semitopological semigroup,  $F(f^{s}_{\sigma})=0$. 

\medskip

2 (a). The same aurguments show that for the maps $f$ of the form 
$f^{s}_{\sigma}(t)$, $\sigma=\{t_1, \ldots, t_n\}$, $t_k\in b_m X\setminus X$, $t_1<t_2<\ldots<t_n$ (which steps are defined with respect to the rule $f(x^-)=f(x)=f(x^+)$), such that $f(t)\in b_m X\setminus X$, $t\in X$, $F(f)=0$. 

\medskip

Since the "step" maps are dense in the set of maps such that $f(t)\in b_m X\setminus X$, $t\in X$, from 2 (a) it follows that   $F(T^{-1}(0_{e_{\alpha X} G}))=0$. 

\medskip

B. {\it $F(T^{-1}(id_{\sigma}))$ is an idempotent for idempotent $id_{\sigma}$ in $e_{\alpha X} G$.}

1. For idempotents 
$$id_x(t)=\left\{
\begin{array}{ll}
x, & t=x, \\
\infty,  & t\ne x, \\
\end{array}
\right.\ x\in X,$$
in $e_{\alpha X} G$, $T^{-1}(id_{x})$ is a set of monotone maps $f$ on $b_m X$ such that $f(t)\in b_m X\setminus X$, $t<x^-$, $t>x^+$, $f(x^-)=x^-$,  $f(x)=x$,  $f(x^+)=x^+$.   
$T^{-1}(id_{x})$ is a closed subsemigroup of $e_{b_m X} G$. 

The maps 
$$h_{t_1, t_2}(t)=\left\{
\begin{array}{ll}
-\infty, & t< (\mbox{or}\ \leq)\ t_1, \\
x^-, & t_1\leq (\mbox{or}\ <)\ t\leq x^-, \\
x, & t=x, \\
x^+, & x^+\leq t < (\mbox{or}\ \leq)\ t_2, \\
+\infty,  & t_2\leq   (\mbox{or}\ <)\  t, \\
\end{array}
\right.\ t_1, t_2\in b_m X,\ t_1\leq x^-, x^+\leq t_2,\ x\in X,$$
are right-zeros (or left ideals) of $T^{-1}(id_{x})$. Therefore, $F(h_{t_1, t_2})$ are right-zeros in $F(T^{-1}(id_{x}))$. Moreover, since $h_{x^-, x^+}$ is a relative left-zero for homeomorphisms $g\in\St_x$, $T^{-1}(id_{x})$ belongs to the closure of $\St_x$ in $e_{b_m X} G$ and $F(\cl\ 
\St_x)$ is a semitopological semigroup, $F(h_{x^-, x^+})$ is a zero in $F(T^{-1}(id_{x}))$. Hence, $\forall\ h_{t_1, t_2}$ $F(h_{t_1, t_2})=F(h_{x^-, x^+})$.  

\medskip 

1 (a). As in A (items 1, 1 (a)) one can show that for the maps 
$$h^{\tau_1, \tau_2}_{t_1, t_2, t_3, t_4}(t)=\left\{
\begin{array}{ll}
-\infty, & t< (\mbox{or}\ \leq)\ t_1, \\
\tau_1, & t_1\leq t < (\mbox{or}\ \leq)\ t_2, \\
x^-, & t_1\leq (\mbox{or}\ <)\ t\leq x^-, \\
x, & t=x, \\
x^+, & x^+\leq t < (\mbox{or}\ \leq)\ t_3, \\
\tau_2, & t_3\leq t < (\mbox{or}\ \leq)\ t_4, \\
+\infty,  & t_4\leq   (\mbox{or}\ <)\  t, \\
\end{array}
\right.\ t_1, t_2, t_3, t_4\in b_m X,\ t_1< t_2\leq x^-, x^+\leq t_3< t_4,\ x\in X,$$
$F(h^{\tau_1, \tau_2}_{t_1, t_2, t_3, t_4})$ is a one-point set. 

\medskip 

2 (a). As in A (items 2, 2 (a)) one can show that for the maps $f\in T^{-1}(id_x)$ which are finite number of "steps" on $(\gets, x^-]$ and $[x^+, \to)$ $F(f)$  is a one-point set.  

\medskip 

Since "step" maps are dense in $T^{-1}(id_{x})$, from 2 (a) it follows that $F(T^{-1}(id_{x}))$ is an idempotent.   

\medskip

The same treatment as in B (items 1, 1 (a), 2 (a)) allows one to show that for any idempotent 
$$id_{\sigma}(t)=\left\{
\begin{array}{ll}
t, & t\in\sigma, \\
\infty,  & t\ne\sigma, \\
\end{array}
\right.\ \sigma=\{x_1, \ldots, x_n\},\ x_1,\ldots, x_n\in X,$$
in $e_{\alpha X} G$, $F(T^{-1}(id_{\sigma}))$ is an idempotent.   
\medskip

C.  {\it $F(T^{-1}(f))$ is a one-point set for a finite partial map $f$ in $e_{\alpha X} G$.}

Since $X$ is ultrahomogeneous, for any finite partial map $f$ in $e_{\alpha X} G$ with finite domain ($f(x_k)=y_k$, $k=1,\ldots, n$, $f(x)=\infty$, $x\not\in\sigma=\{x_1, \ldots, x_n\}$) $\exists\ g\in G$ such that $gf=id_{\sigma}$. The map $F$ is a $G$-map. Therefore, $g (T^{-1} (f))=T^{-1} (id_{\sigma})$ and $F (T^{-1} (f))$ is a one-point set.

\medskip

D. {\it $F(T^{-1}(f))$ is a one-point set for a partial map $f$ in $e_{\alpha X} G$.}

Take a partial map $f\in e_{\alpha X} G$. Assume that $h, g\in T^{-1}(f)$ but $F(h)\ne F(g)$. Take disjoint nbd $O$ and $V$ of $F(h)$ and $F(g)$ respectively. 
Then $O_h=F^{-1} O$ and $V_g=F^{-1} V$ are disjoint nbds of $h$ and $g$ respectively. Without loss of generality, we can assume that they are of the form 
$$O_h=\{f'\in e_{b_m X} G\ |\ f' (t)\in O_t, t=x_1,\ldots, x_n\},\  V_h=\{f'\in e_{b_m X} G\ |\ f' (t)\in V_t, t=x_1,\ldots, x_n\},$$
$x_1<\ldots< x_n$. Let $Z=\{x_{m_1},\ldots, x_{m_l}\}$ be those points in which $h(x_{m_j})=g(x_{m_j})=y\in X$, $j=1,\ldots, l$. If $Z=\emptyset$, then take any $h'(x_k)\in O_k$, $g'(x_k)\in V_k$, $h'(x_k), g'(x_k)\in b_m X\setminus X$, and extend $h', g'$ to the step map on $b_m X$ with values $\inf, h'(x_k), k=1,\ldots, n,\ \sup$ and  $\inf, g'(x_k), k=1,\ldots, n,\ \sup$ respectively. Then $h'\in O_h$, $g'\in V_g$, $F(h')=F(g')$ and a contradiction with the choice of disjoint nbds of $F(h)$ and $F(g)$ is obtained. 

If $Z\ne\emptyset$, then as above one can define maps $h', g'\in b_{b_m X} G$ such that $h'(x_{m_j})=g'(x_{m_j})$, $j=1,\ldots, l$,  $h'(x_k)\in O_k$, $g'(x_k)\in V_k$, $h'(x_k), g'(x_k)\in b_m X\setminus X$ for $x_k\not\in Z$, with values  $\inf, h'(x_k), k=1,\ldots, n,\ \sup$ and  $\inf, g'(x_k), k=1,\ldots, n,\ \sup$ respectively. Then $h'\in O_h$, $g'\in V_g$, $F(h')=F(g')$ and a contradiction with the choice of disjoint nbds of of $F(h)$ and $F(g)$ is obtained. 

\medskip

By the above resonings and description of $e_{\alpha X} G$ one has ${\rm WAP}\, G=e_{\alpha X} G$.  

\medskip

The equality $e_{\alpha X} G={\rm WAP}\, G$ and the inequalities $e_{\alpha X} G\leq e_B G\leq {\rm WAP}\, G$ imply 
$$e_{\alpha X} G=e_B G= {\rm WAP}\, G.$$
\end{proof}

\begin{cor}\label{autchainWAP}
Let $X$ be an ultrahomogeneous chain, $G=(\aut (X), \tau_{\partial})$. ${\rm WAP}\, G$ is the closure of $G$ in $b_r  {\rm U}(\ell^2(X))={\rm WAP}\,  {\rm U}(\ell^2(X))$.
\end{cor}

\begin{rem}
{\rm  The group $\aut (\mathbb Q)$ (in the topology of pointwise convergence from its $\tau_p$-representation in a discrete chain $\mathbb Q$) is a non-archimedean Roelcke precompact Polish group, its Roelcke-compactification $b_r \aut (\mathbb Q)$ is not a compact semigroup. Further, the Hilbert compactification of $\aut (\mathbb Q)$ is a semitopological inverse monoid with continuous inverse, is a WAP-compactification of $\aut (\mathbb Q)$ ($\aut (\mathbb Q)$ is {\it Eberlein group} (see definition in~\cite{GlasnerMegr})) and all its factors are Hilbert--representable~\cite {BITsankov}.

The results of section~\ref{ultrchain} hold for subgroups of $(\aut (X), \tau_{\partial})$ which {\it acts ultratransitively}~\cite{KozlovSorin2025}. These subgroups are dense subgroups of  $(\aut (X), \tau_{\partial})$. 

The  WAP-compactification of $(\aut ({\bf F}), \tau_{\partial})$ coincides with WAP-compactification of $\aut (\mathbb Q)$, where $F$ is a Thompson group~\cite{Canon} (take into account~\cite[Remark 5.3]{KozlovSorin2025} and Corollary~\ref{heredsm}).}
\end{rem}


\section{Automorphism group of ultrahomogeneous LOTS}\label{ultLOT}

Let $X$ be an ultrahomogeneous chain, $\aut X$ be automorphism group of $X$. If $X$ is equipped with the topology of linear order $\tau$ ($(X, \tau)$ is a LOTS), then the t.p.c $\tau_p$ is an admissible group topology on $\aut X$~\cite{Ovch} or~\cite{Sorin}. Hence, $G=(\aut X, \tau_p)$ is $\tau_p$-representable in $X$. The maximal equiuniformity $\mathcal U^{\max}_X$ on $X$ is totally bounded. The maximal $G$-compactification $\beta_G X$ (completion of $(X, \mathcal U^{\max}_X)$) is the least linearly ordered compactification of $X$. It is obtained by replacing every gap (also improper) in $X$ by a point with the natural extension of the order~\cite[Lemma 1.13]{Sorin2}. Let $c X$ be the compactification of $X$ obtained by identifying points $\{\inf\}$ and  $\{\sup\}$ of $\beta_G X$. 

\begin{lem}~\cite{KozlovSorin2025}\label{gluinf}
The $G$-space $(G=(\aut X, \tau_p), X=(X, \tau), \curvearrowright )$ has only two $G$-compatifications $c_m X=\beta_G X$ and $c X$. 
\end{lem}

There are three locally compact non-compact $G$-extensions of $X$. Indeed, $Y^+=c_m X\setminus\{\inf\}$, $Y^-=c_m X\setminus\{\sup\}$ and $Y=c_m X\setminus\{\inf, \sup\}$ are the only invariant locally compact non-compact subsets of $c_m X$; $Y=c X\setminus\{\infty\}$ is the only invariant locally compact  non-compact subset of $c X$. All of them are nonequivalent locally compact $G$-extensions and they are locally connected. Hence, the group $G=(\aut X, \tau_p)$ is $\tau_p$-representable in all $Z\in\mathbb{GLC} (X)=\{Y, Y^-, Y^+, c_m X, c X\}$ by Lemma~\ref{lemcompemb}. 

The Ellis compactifications $e_{c_m X} G$ and $e_{c X} G$ of  $G$ correspondent to $G$-compactifications $c_m X$ and $c X$ of $X$ are described in~\cite[Theorem 5.8,\ Corollary 5.10,\ Proposition 5.11]{KozlovSorin2025}.  $e_{c_m X} G$ is the set of monotone self-maps $f$ of $c_m X$ {\rm(}$x<y\Longrightarrow f(x)\leq f(y)${\rm)} in the topology of pointwise convergence such that $f(\inf)=\inf$,  $f(\sup)=\sup$, 

\noindent $e_{c X} G$ is the set of all self-maps $f$ of $c X$ such that  $f(\infty)=\infty$,  the restriction of $f$ to $f^{-1} (c X\setminus\{\infty\})$ is a monotone map of $f^{-1} (c X\setminus\{\infty\})$ to $c X\setminus\{\infty\}$ {\rm(}with the linear order induced on $Y$ as a subset of $c_m X${\rm)}. 

\noindent $e_{c X} G=e_{c_m X} G/I$ ($e_{c X} G$ is a Rees quotient of $e_{c_m X} G$, where ideal 
$$I=\{f^-_y=\left\{
\begin{array}{ll}
\inf, & x\leq y, \\
\sup, & y<x,  \\
\end{array}
\right. 
 y\in Y^-\ \mbox{and}\ 
f^+_z=\left\{
\begin{array}{ll}
\inf, & x<z, \\
\sup & z\leq x, \\
\end{array}
\right. z\in Y^+\}.)
$$

$G$ is Roelcke precompact, $b_r G<e_{c_m X} G$, compactifications $b_r G$ and $e_{c X} G$ are incomparable~\cite{KozlovSorin2025}. 


\subsection{Graph  compactifications of $(\aut X, \tau_p)$} 

The orders on locally compact (and locally connected) extensions of $X$ and corresponding, correctly defined, graph compactifications of $G=(\aut X, \tau_p)$ are the following:
$$\begin{array}{ccc}
   &   c_m X    &  \\
\ \ \swarrow      &  \downarrow        &  \searrow  \\
Y^-   &   c X   &  \ \ Y^+ \\
\ \ \searrow      &  \downarrow        &  \swarrow  \\
 &   Y  &   \\
\end{array}\quad 
\begin{array}{ccc}
   &   b_{c_m X} G    &  \\
\ \ \swarrow      &  \downarrow        &  \searrow  \\
b_{Y^-}  G   &   b_{c X}  G   &  \ \ b_{Y^+}  G \\
\ \ \searrow      &  \downarrow        &  \swarrow  \\
 &   b_Y  G.  &   \\
\end{array}
$$

Let $\i^{\Gamma}_{c_m X}: G\to (2^{c_m X\times c_m X}, \tau_V)$, $\i^{\Gamma}_{c_m X} (g)=\{(x, g(x))\ |\ x\in c_m X\}$, and $\i^{\Gamma}_{c X}: G\to (2^{c X\times c X}, \tau_V)$, $\i^{\Gamma}_{c X} (g)=\{(x, g(x))\ |\ x\in c X\}$, $g\in G$, be embeddings of $G$. $b_{c_m X} G$ and  $b_{c X} G$ are the correspondent graph compactifications (closures of images of $G$ in $(2^{c_m X\times c_m X}, \tau_V)$ and $(2^{c X\times c X}, \tau_V)$ respectively).  $\pr_1$ and $\pr_2$ are projections of $c_m X\times c_m X$ (or $c X\times c X$) on the first and the second factor respectively, $(c_m X\times c_m X, \leq_{\times})$ is a {\it cardinal product} ($(x, y)\leq (x', y')$ iff $x\leq x'$\ \&\ $y\leq y'$)~\cite[Ch.~1, item 7]{Birkhoff}. 

\begin{pro}\label{Prodescrlotbeta}
Let $X$ be an ultrahomogeneous {\rm LOTS}, $G=(\aut X, \tau_p)$. Then $b_{c_m X} G$ is the set of linearly ordered {\rm(}the restriction of $\leq_{\times}${\rm)} continua $A\in\CL (c_m X\times c_m X)$ such that $\pr_1 (A)=\pr_2 (A)=c_m X$.
\end{pro}

\begin{proof}
Take $A\in b_{c_m X} G=\cl\ \i^{\Gamma}_{c_m X} (G)$. $\pr_1 (A)=\pr_2 (A)=c_m X$ by Lemma~\ref{proj}. 

If there exist $(x, y), (x', y')\in A$ such that $x<x'$, $y>y'$, then take  $x<t<x'$, $y'<\tau<y$. The set $W=(\pr_1^{-1}(\gets, t)\cap\pr_2^{-1}(\tau, \to))^-\cap (\pr_1^{-1}(t, \to)\cap\pr_2^{-1}(\gets, \tau))^-$ is an open nbd of $A$ and $W\cap \i^{\Gamma}_{c_m X} (G)=\emptyset$. Hence, $A\not\in\cl (\i^{\Gamma}_{c_m X} (G))$. The case of $x>x'$, $y<y'$ is verified similarly. Therefore, ${\leq_{\times}}|_A$ is a linear order on $A$. 

Since $A\in\CL (c_m X\times c_m X)$ there no jumps in $A$. Indeed, if there is a jump $(x, y)<(x', y')$ in $A$, then either $x<x'$ or $y<y'$. If  $x<x'$ there exists $x<t<x'$ such $(\{t\}\times c_m X)\cap A=\emptyset$ and, hence, $\pr_1 (A)\ne c_m X$. Therefore, $A$ is a continuum. The case $y<y'$ is verified similarly. 

\medskip

Take $A\in\CL (c_m X\times c_m X)$ such that $(A, {\leq_{\times}}|_A)$ is a linearly ordered continua, $\pr_1 (A)=\pr_2 (A)=c_m X$. Any nbd of $A$ in $(2^{c_m X\times c_m X}, \tau_V)$ is of the form $[V_1, \ldots, V_n]$ and, additionally, since $\dim A=1$, and $c_m X\times c_m X$ is a product, one can consider that $V_i=O_i\times U_i$, $O_i=(a^-_i, a^+_i)$, $U_i=(b^-_i, b^+_i)$ are open interval of $c_m X$, $i=2,\ldots, n-1$, ($V_1=O_1\times U_1$, $O_1=[a^-_1=\inf, a^+_1)$, $U_1=[b^-_1=\inf, b^+_1)$, $V_n=O_n\times U_n$, $O_n=(a^-_1, a^+_1=\sup]$, $U_n=(b^-_n, b^+_n=\sup]$) and $V_i\cap V_{i+2}=\emptyset$, $i=1,\ldots, n-2$ (the order of the family $\{V_1, \ldots, V_n\}$ is one and since $A$ is a continuum, $V_i\cap V_{i+1}\ne\emptyset$, $i=1,\ldots, n-1$). Since ${\leq_{\times}}|_A$ is a linear order on $A$, one can also assume that $a^-_i< a^-_{i+1}$,  $a^+_i\leq a^+_{i+1}$, $b^-_i< b^-_{i+1}$, $b^+_i\leq b^+_{i+1}$, $i=1,\ldots, n-1$.

Take $n-1$ points $(t_i, \tau_i)\in (X\times X)\cap (V_i\cap V_{i+1})$ such that $t_i<t_{i+1}$,  $\tau_i<\tau_{i+1}$, $i=1, \ldots, n-1$. Ultratrahomogeneity of $X$ yields that there exists $g\in G$ such that $g(t_i)=\tau_i$, $i=1, \ldots, n-1$, ($g(\inf)=\inf$, $g(\sup)=\sup$). Since $\i^{\Gamma}_{c_m X}  (g)\in [V_1, \ldots, V_n]$, $A\in\cl\ \i^{\Gamma}_{c_m X} (G)$. 
\end{proof}

\begin{rem}
{\rm $b_{c_m X} G$ is not a semigroup. Indeed, $A=(c_m X\times\{\inf\})\cup (\{\sup\}\times c_m X)$, $B=(\{\inf\}\times c_m X)\cup(c_m X\times\{\sup\})\in b_{c_m X} G$, $BA=c_m X\times c_m X\not\in b_{c_m X} G$.

The description of $b_{c_m X} G$ for $X=(0, 1)$ (the maximal $G$-compactification of $X$ is $[0, 1]$) is given in~\cite{usp2001} (see also~\cite{GlasnerMegr2008}).}
\end{rem}

By Theorem~\ref{mapcompVel} the graph compactification $b_{c X} G$ is the image of $b_{c_m X} G$ under the map $F^H$, $F=f\times f$, where $f: c_m X\to c X$ is an elementary map of compactifications of $X$ (identification of points $\inf$ and $\sup$).

\begin{lem}\label{lbetab}
For $K=\{\inf, \sup\}$ the condition {\rm(ER)} of {\rm Theorem~\ref{restrmap}} is valid.
\end{lem}

\begin{proof}
Take $A, B\in b_{c_m X} G$ such that $A\cap (Y\times Y)=B\cap (Y\times Y)\ne\emptyset$, $Y=c_m X\setminus K$.

For the projections of $A$ (or $B$) there are four possible cases:
\begin{itemize}
\item[{\rm (1)}] $\pr_1 (A\cap (Y\times Y))=Y$, 

\item[{\rm (2)}] $\pr_1 (A\cap (Y\times Y))=(\gets, b)$ or $(\gets, b]$, $b\in Y$,  

\item[{\rm (3)}] $\pr_1 (A\cap (Y\times Y))=(a, \to)$ or $[a, \to)$, $a\in Y$, 

\item[{\rm (4)}] $\pr_1 (A\cap (Y\times Y))=[a, b]$ or $(a, b]$ or $[a, b)$ or $(a, b)$, $a, b\in Y$, $a\leq b$. 
\end{itemize}

Let $a'$ be the least lower bound and $b'$ is the greatest upper bound of $\pr_2 (A\cap (Y\times Y))$. 

In the case (1) $A=(\{\inf\}\times (\gets, a'])\cup (A\cap (Y\times Y)\cup (\{\sup\}\times [b', \to))$. 

In the case (2) $b'=\sup$ and $A=(\{\inf\}\times (\gets, a'])\cup (A\cap (Y\times Y)\cup ([b, \to)\times\{\sup\})$. 

In the case (3) $a'=\inf$ and $A=((\gets, a]\times\{\inf\})\cup (A\cap (Y\times Y)\cup  (\{\sup\}\times [b', \sup))$. 

In the case (4) $A=((\inf, a]\times\{\inf\})\cup (A\cap (Y\times Y)\cup  ([b, \sup)\times\{\sup\})$. The same is true for $B$ and $A=B$.

It remains to note that $A=(b_m X\times\{\inf\})\cup (\{\sup\}\times b_m X)$ and $B=(\{\inf\}\times b_m X)\cup(b_m X\times\{\sup\})\}$ are the only sets $R\in  b_{b_m X} G$ such that $R\cap (Y\times Y)=\emptyset$.
\end{proof}

\begin{pro}\label{Prodescrlotb} 
Let $X$ be an ultrahomogeneous {\rm LOTS}, $G=(\aut X, \tau_p)$. 
\begin{itemize}
\item[(a)]  $F^H$ is an elemantary map of $G$-compactifications of $G$, 
$$b_{c X} G=b_{c_m X} G/0_K,$$
$0_K=\{A=(c_m X\times\{\inf\})\cup (\{\sup\}\times c_m X),\ B=(\{\inf\}\times c_m X)\cup(c_m X\times\{\sup\})\}$.  
\item[(b)] $b_{Y^-} G= b_{Y^+} G=b_{c_m X} G$, 
\item[(c)] $b_{Y} G=b_{\alpha Y} G=b_{c X} G$. 
\end{itemize}
\end{pro}

\begin{proof}
From Lemma~\ref{lbetab} $A=(c_m X\times\{\inf\})\cup (\{\sup\}\times c_m X)$ and $B=(\{\inf\}\times c_m X)\cup(c_m X\times\{\sup\})$ are the only $R\in  b_{c_m X} G$  such that $R\cap (Y\times Y)=\emptyset$ and $F^H (A)=F^H (B)$.

Theorems~\ref{mapcompVel},\ \ref{restrmap} and Remark~\ref{corrcond} finish the proof.
\end{proof}

\begin{rem}\label{rem2}
{\rm $b_{c X} G$ is not a semigroup. Indeed, for $A=(\{\infty\}\times c X)\cup (c X\times \{\infty\})$ 
$AA=c X\times c X\not\in b_{c X} G$.

Item (c) of Proposition~\ref{Prodescrlotb} serves as an example  when item (b) of Theorem~\ref{partialcases} holds, but item (a) doesn't hold.}
\end{rem}


\subsection{Order on $G$-compactifications of $(\aut X, \tau_p)$.}

\begin{thm}\label{thmOrderLOTS} Let $X$ be an ultrahomogeneous {\rm LOTS}, $G=(\aut X, \tau_p)$. 
$$\begin{array}{rcc}
b_r G = b_{c_m X} G & \longleftarrow  & e_{c_m X} G\\
  & &  \\
  \downarrow \quad\quad\quad\quad\quad & \not\searrow\ \not\nwarrow              &  \downarrow \\
  & &  \\
  b_{c X} G= b_{c_m X} G /0_K    &   \not\longleftrightarrow          &    e_{c X} G= e_{c_m X} G/I. \\
\end{array}
$$
\end{thm}

\begin{proof} {\it The proof that $b_{c_m X} G<e_{c_m X} G$}. Since $b_{c_m X}  G\leq b_r  G<e_{c_m X} G$ (Corollary~\ref{closure} and~\cite[Theorem 5.8]{KozlovSorin2025}) one has $b_{c_m X} G< e_{c_m X} G$. For the description of the map of compactifications $\Theta: e_{c_m X} G\to b_{c_m X} G$ we introduce the equivalence relation $\sim_{\Theta}$ on  $e_{c_m X} G$.

Since every subset of a compact LOTS has the least upper bound (abbreviation LUB) and the greatest lower bound (abbreviation GLB), $\forall\ f\in e_{c_m X} G$ and $\forall\ x\in c_m X$ there exist $x_f^-={\rm LUB}\{f(t)\ |\ t<x\}$ ($\inf_f^-=\inf$) and $x_f^+={\rm GLB}\{f(t)\ |\ t>x\}$  ($\sup_f^+=\sup$)  and $x_f^-\leq f(x)\leq x_f^+\leq y_f^-$ if $x<y$ ($f$ is monotone). Moreover, $x_f^-=x_f^+$ iff $f$ is continuous at $x$. Evidently, 
$$f\sim_{\Theta} h\ \Longleftrightarrow x_f^+=x_h^+,\  x_f^-=x_h^-,\ \forall\ x\in c_m X$$
is an equivalence relation on $e_{c_m X} G$. If $f\in \i^{\Gamma}_{c_m X}(G)$, then $[f]=f$. Every equivalence class $[f]$, $f\in e_{c_m X} G$,  can be identified with a subset 
$A_f=\bigcup\{\{x\}\times [x^-_f, x^+_f]\ |\ x\in c_m X\}\in\CL\ (c_m X\times c_m X)$. Indeed, firstly, $A_f$ is a closed set. If $(x, y)\not\in A_f$, $y<x_f^-$, $x\ne\inf$, then $\exists$ $a<x$ and $y<b<x_f^-$ such that $b<f(a)$. Since $f$ is monotone, $((a, \to)\times (\gets, b))\cap A_f=\emptyset$ and $(a, \to)\times (\gets, b)$ is a nbd of $(x, y)$. The cases $y>x_f^+$ and $x=\inf$ or $\sup$ are examined similarly. Secondly, the restriction of order $\leq_{\times}$ to  $A_f$ is a linear order on $A_f$. Thirdly, $\pr_1 (A_f)=c_m X$ and $\pr_2 (A_f)=c_m X$ since $\forall\ y\in c_m X$ $\exists\ x\in c_m X$ such that $x_f^-\leq y\leq x_f^+$.  The  identification of equivalent classes of $e_{c_m X} G$ and $b_{c_m X} G$ is a bijection. Injectivity is evident and $A\in b_{c_m X} G$ is identified with $[f]$, where $f\in e_{c_m X} G$ is such that  $f(x)$ is an arbitrary point of $A\cap (\{x\}\times c_m X)$, $x\in c_m X$.

Let $\Theta: e_{c_m X} G\to b_{c_m X} G$ be the composition of the quotient map of $e_{c_m X} G$ onto $e_{c_m X} G/\sim_{\Theta}$ and the described identification of $e_{c_m X} G/\sim_{\Theta}$ and $b_{c_m X} G$. It remains to show that $\Theta$ is continuous. $\forall\ f\in e_{c_m X} G$ and a nbd  $(U\times V)^-$, where $U=(a, b)$ and $V=(c, d)$ are open in $c_m X$, of $A_f=\Theta (f)$ from the subbase of Vietoris topology, take arbitrary  $(x, y)\in A_f\cap (U\times V)$. 

If $f(x)\in V$, then $[x, V]$ is a nbd of $f$ and $\Theta ([x, V])\subset (U\times V)^-$. Otherwise, if $x^-_f<d\leq f(x)\leq x^+_f$, then take $a<t<x<t'<b$, $y<z<z'<d$ such that $W=[t, (\gets, z)]\cap [t', (z', \to)]$ is a nbd of $f$. Then $\Theta (W)\subset  (U\times V)^-$. The case $x^-_f\leq f(x)\leq c$ is examined similarly. 

Now let a nbd of $A_f$  from the subbase be of the form $((c_m X\times c_m X)\setminus K)^+$, where $K$ is compact. Without loss of generality, one can assume that $K=K_1\times K_2$, where $K_1=[a, b],\ K_2=[c, d]$ are compact subsets of $c_m X$. If $f(a)>d$, then $W=[a, (d, \to)]$ is a nbd of $f$ and $\Theta (W)\subset ((c_m X\times c_m X)\setminus K)^+$. If $f(b)<c$, then $W=[b, (\gets, c)]$ is a nbd of $f$ and $\Theta (W)\subset ((c_m X\times c_m X)\setminus K)^+$. 

\medskip

{\it The proof that $b_{c X} G$ and $e_{c X} G$ are incomparable}. If $e_{c X} G\geq b_{c X} G$, then the following diagram of maps of compactifications is commutative
$$\begin{array}{rcl}
 b_{c X} G  &  \stackrel{F^H\circ\Theta} {\longleftarrow}    &  e_{c_m X} G \\
 \nwarrow      &          &  \swarrow  \\
  &    e_{c X} G.   &     \\
\end{array}
$$
The map $e_{c_m X} G\to e_{c X} G$ of compactifications sends ideal $I$ to the one-point set. The map $F^H\circ\Theta$ sends $I$ to the set $\{(c X\times\{\infty\})\cup (\{x\}\times c X)\ |\ x\in c X\}$. Hence, $e_{c X} G\not\geq b_{c X} G$. 

If $e_{c X} G< b_{c X} G$, then the following diagram of maps of compactifications is commutative
$$\begin{array}{rcl}
 b_{c X} G   &  \stackrel{F^H\circ\Theta} {\longleftarrow}   &    e_{c_m X} G \\
 \searrow      &          &  \swarrow  \\
  &    e_{c X} G.  &     \\
\end{array}
$$
For arbitrary $y\in X$ the map $F^H\circ\Theta$ sends points
$$f_a=f_a(x)=\left\{
\begin{array}{ll}
\inf, & x<y, \\
a, & x=y, \\
\sup, & y<x,  \\
\end{array}
\right. a\in c_m X, 
$$
of $ e_{c_m X} G$ to the one-point set  $(c X\times\{\infty\})\cup (\{y\}\times c X)$. The map $e_{c_m X} G\to e_{c X} G$ of compactifications sends $f_a$ to the point  
$$f'_a=f'_a(x)=\left\{
\begin{array}{ll}
\infty, & x\ne y, \\
a, & x=y. \\
\end{array}
\right. 
$$
Hence, images of $f_a$ and $f_b$ are distinct if $a\ne b$ and $b_{c X} G \not\geq e_{c X} G$. 

\medskip

{\it The proof that $b_{c_m X} G$ and $e_{c X} G$ are incomparable}. The previous resonings show that $b_{c_m X} G \not\leq e_{c X} G$. The map $\Theta: e_{c_m X} G\to b_{c_m X} G$ sends points $f_a$, $a\in c_m X$,  to the one-point set  $((\gets, y)\times\{\inf\})\cup (\{y\}\times c_m X)\cup ((y, \to)\times\{\sup\})$. The map $e_{c_m X} G\to e_{c X} G$ of compactifications sends points $f_a$ to the points $f'_a$, $a\in X$. 
Hence, $b_{c_m X} G \not\geq e_{c X} G$. 

\medskip

{\it The proof that $b_{c_m X}=b_r G$}. The inequality $b_{c_m X} G\leq b_r G$ follows from Corollary~\ref{closure}. To prove the inequality $b_{c_m X} G\geq b_r G$ we verify the fullfilment of condition $(\star)$ of Theorem~\ref{suffcond}. 

Fix an entourage ${\rm U}$ from the unique uniformity on $c_m X$. Let 
$$v=\{J^a_1=[a_0=\inf, a_1), J^a_{n}=(a_{n-1}, a_{n}=\sup]\}\bigcup\{J^a_k=(a_k, a_{k+1})\ |\ k=1,\ldots, n-2\}\bigcup$$
$$\{J^b_k=(b_{k}, b_{k+1})\ |\ k=1,\ldots, n-1\}$$ 
$$\mbox{where}\ a_k\in X,\ k=1,\ldots, n-1,\ b_k\in X,\ k=1,\ldots, n,\ a_{k-1}<b_{k}<a_{k},\ \ k=1,\ldots, n.$$
be the cover of $c_m X$ by intevals (points $a_0$ and $a_n$ are included in $J^a_1$ and $J^a_n$ respectively) which corresponds to the entourage ${\rm V}\subset {\rm U}$ (this can be done using star refinement of covers and one-dimensionality of LOTS $c_m X$). 

The following equality $\st ((x, y), v\times v)=\st (x, v)\times\st (y, v)$, $(x, y)\in c_m X\times c_m X$, holds. 

Let us show that if $(i^{\Gamma}_{c_m X} (f), i^{\Gamma}_{c_m X} (h))\in 2^{V\times V}$, $f, g\in G$ ($i^{\Gamma}_{c_m X} (f)\in\st (i^{\Gamma}_{c_m X} (h), v\times v)$ and $i^{\Gamma}_{c_m X} (h)\in\st (i^{\Gamma}_{c_m X} (f), v\times v)$), then $\exists\ g\in G$ such that $(f(x), g(x))\in {\rm U},\  (g^{-1}(x), h^{-1}(x))\in {\rm U}\ \forall\ x\in X$.
This yields by Theorem~\ref{suffcond} that the restriction of the unique uniformity on $2^{c_m X\times c_m X}$ to $i^{\Gamma}_{c_m X} (G)$ is greater than  the Roelcke uniformity on $G$. 

Let us note that $\st (a_k, v)=J^b_k$, $k=1,\ldots, n-1$; $\st (b_k, v)=J^a_k$, $k=1,\ldots, n$; 

$\st (x, v)=J^a_1$ if $x\in [a_0, b_1)$; $\st (x, v)=[a_0, b_2)$ if $x\in (b_1, a_1)$; 

$\st (x, v)=(a_{k-1}, b_{k+1})$ if $x\in (b_k, a_k)$, $k=2,\ldots, n-1$;  

$\st (x, v)=(b_{k}, a_{k+1})$ if $x\in (a_k, b_{k+1})$, $k=1,\ldots, n-2$;  

$\st (x, v)=(b_{n-1}, a_n]$ if $x\in (a_{n-1}, b_n)$; $\st (x, v)=J^a_n$ if $x\in (b_n, a_n]$. 

\medskip

{\it Claim A.} If $(i^{\Gamma}_{c_m X} (f), i^{\Gamma}_{c_m X} (h))\in 2^{V\times V}$, $f, g\in G$, then $\forall\ x\in c_m X$ the following hold 
$$\forall\ y\in (c_m X\setminus (\st (x, v)\cup [x, a_n]))\ h(y)\in ([a_0, f(x)]\cup\st (f(x), v)),$$
$$\forall\ y\in (c_m X\setminus ([a_0, x]\cup\st (x, v)))\ h(y)\in (\st (f(x), v)\cup [f(x), a_n]).$$

{\it Claim B.} Let $f\in G$, $[a, b]$ is either segment $[a_i, b_{i+1}]$ or $[b_i, a_i]$, $y'>y$, $y\in\st (f(a), v)$ and $y'\in\st (f(b), v)$. 
Then $\exists\ \varphi\in G$ such that $\varphi (a)=y$, $\varphi (b)=y'$ and $\varphi (x)\in\st (f(x), v)$, $x\in [a, b]$.

{\it Proof.} Let $T=\{t\in (a, b)\ |\ f(t)=a_k\vee b_m,\ k=1,\ldots, n-1,\ m=1, \dots, n\}$. If $T=\emptyset$, then any $\varphi\in G$,  $\varphi (a)=y$, $\varphi (b)=y'$, satisfies the condition of the Claim. 

Otherwise, order $T=\{t_1,\ldots, t_s\}$. Put $t_0=a$, $t_{s+1}=b$ and take $\tau_i\in\st (f(t_i), v)$, $i=1,\ldots, s$,  such that $y<\tau_1<\ldots<t_s<y'$.

Let $\varphi\in G$ sends points $t_0< t_1<\ldots< t_s<t_{s+1}$ to the points $y< \tau_1<\ldots< \tau_s< y'$ respectively. Since $\st (f(x), v)\supset (\st (f(t_i), v)\cup\st (f(t_{i+1}), v))$, $x\in (t_i, t_{i+1})$, $i=1,\ldots, s$, $\varphi$ satisfies the condition of the Claim. $\Box$

\medskip

The construction of $g$ will be done inductively on the end points of intervals of the cover $v$. Without loss of generality let $h(b_1)\geq f(b_1)$. 

On $[a_0, b_1]$ put $g|_{[a_0, b_1]}=f|_{[a_0, b_1]}$. Evidently, $(f(x), g(x))\in V$, $x\in [a_0, b_1]$ and $(g^{-1}(x), h^{-1}(x))\in V$, $x\in [a_0, f(b_1)]$.

\medskip

{\it Base of induction}: $g|_{[a_0, a_1]}$. By Claim A $h(a_1)\in (\st (f(b_1), v)\cup [f(b_1), a_n])$. 

\medskip

(A) If $f(a_1)\leq h(b_1)$, then by Claim A $h(b_1)\in\st (f(a_1), v)$. Take: 
$$y\in X\ \mbox{such that}\ y\in\st (f(a_1), v),\ h(b_1)<y<h(a_1).$$
By Claim B $\exists\ \varphi\in G$ such that $\varphi ([b_1, a_1])=[f(b_1), y]$, $\varphi (x)\in\st (f(x), v)$, $x\in [b_1, a_1]$.  
Let $g|_{[a_0, a_1]}$ be a combination of maps $f|_{[a_0, b_1]}$ and $\varphi|_{[b_1, a_1]}$. 
$$f(a_1)<g(a_1)<h(a_1),\ g(a_1)\in\st (f(a_1), v).$$
Evidently, $(f(x), g(x))\in V$, $x\in [a_0, a_1]$ and $(g^{-1}(x), h^{-1}(x))\in V$, $x\in [a_0, g(a_1)]$.

\medskip

(B) If $h(b_1)<f(a_1)\leq h(a_1)$. 

Put $g|_{[a_0, a_1]}=f|_{[a_0, a_1]}$. 
$$f(a_1)=g(a_1)\leq h(a_1).$$
Evidently, $(f(x), g(x))\in V$, $x\in [a_0, a_1]$ and $(g^{-1}(x), h^{-1}(x))\in V$, $x\in [a_0, g(a_1)]$.

\medskip

(C) If $h(a_1)< f(a_1)\leq h(b_2)$.

Put $g|_{[a_0, a_1]}=f|_{[a_0, a_1]}$. 
$$h(a_1)<g(a_1)=f(a_1)\leq h(b_2).$$
Evidently, $(f(x), g(x))\in V$, $x\in [a_0, a_1]$ and $(g^{-1}(x), h^{-1}(x))\in V$, $x\in [a_0, h(a_1)]$.

\medskip

(D) If $h(b_2)<f(a_1)$, then by Claim A $h(b_2)\in\st (f(a_1), v)$. Take 
$$y\in X\ \mbox{such that}\ y\in\st (f(a_1), v),\ h(a_1)<y<h(b_2).$$
By Claim B $\exists\ \varphi\in G$ such that $\varphi ([b_1, a_1])=[f(b_1), y]$, $\varphi (x)\in\st (f(x), v)$, $x\in [b_1, a_1]$.  
Let $g|_{[a_0, a_1]}$ be a combination of maps $f|_{[a_0, b_1]}$ and $\varphi|_{[b_1, a_1]}$. 
$$h(a_1)<g(a_1)<f(a_1),\ g(a_1)\in\st (f(a_1), v).$$
Evidently, $(f(x), g(x))\in V$, $x\in [a_0, a_1]$ and $(g^{-1}(x), h^{-1}(x))\in V$, $x\in [a_0, h(a_1)]$.

\medskip

The map $g|_{[a_0, a_1]}$ is such that $(f(x), g(x))\in V,\ x\in [a_0, a_1]$, 
$$f(a_1)\leq g(a_1)\leq h(a_1)\ \mbox{and}\ (g^{-1}(x), h^{-1}(x))\in V,\ x\in [a_0, g(a_1)],\ \mbox{or}$$
$$h(a_1)\leq g(a_1)\leq f(a_1)\ \mbox{and}\ (g^{-1}(x), h^{-1}(x))\in V,\ x\in [a_0, h(a_1)].$$

\medskip

{\it Step of induction.} Let $g|_{[a_0, c]}$ be such that $(f(x), g(x))\in V$, $x\in [a_0, c]$,
$$f(c)\leq g(c)\leq h(c)\ \mbox{and}\ (g^{-1}(x), h^{-1}(x))\in V,\ x\in [a_0, g(c)],\ \mbox{or}\leqno{\rm (I)}$$
$$h(c)\leq g(c)\leq f(c)\ \mbox{and}\ (g^{-1}(x), h^{-1}(x))\in V,\ x\in [a_0, h(c)]\leqno{\rm (II)}$$
$c=a_k$, $k\leq n-1$ or $c=b_k$, $k\leq n-1$. Without loss of generality let $c=a_k$. 

Construction of $g|_{[a_0, b_{k+1}]}$. 

\medskip

(I)  If $f(a_k)\leq g(a_k)\leq h(a_k)$, $g(a_k)\in\st (f(a_k), v)$. 

(a) If $f(b_{k+1})\leq h(a_k)$, then by Claim A $h(a_k)\in\st (f(b_{k+1}), v)$. Take 
$$y\in X\ \mbox{such that}\  y\in\st (f(b_{k+1}), v),\ h(a_k)<y<h(b_{k+1}).$$ 
By Claim B $\exists\ \varphi\in G$ such that $\varphi ([a_k, b_{k+1}])=[g(a_k), y]$ and $\varphi (x)\in\st (f(x), v),\ x\in [a_k, b_{k+1}]$. 
Let $g|_{[a_0, b_{1+1}]}$ be a combination of maps $g|_{[a_0, a_k]}$ and $\varphi|_{[a_k, b_{k+1}]}$. 
$$f(b_{k+1})<g(b_{k+1})<h(b_{k+1}),\ y=g(b_{k+1})\in\st (f(b_{k+1}), v).$$
Evidently, $(f(x), g(x))\in V$, $x\in [a_0, b_{k+1}]$ and $(g^{-1}(x), h^{-1}(x))\in V$, $x\in [a_0, g(b_{k+1})]$.

\medskip

(b) If $h(a_k)<f(b_{k+1})\leq h(b_{k+1})$, then by Claim B
$$\exists\ \varphi\in G\ \mbox{such that}\ \varphi ([a_k, b_{k+1}])=[g(a_k), f(b_{k+1})]\ \mbox{and}\ \varphi (x)\in\st (f(x), v),\ x\in [a_k, b_{k+1}].$$ 
Let $g|_{[a_0, b_{1+1}]}$ be a combination of maps $g|_{[a_0, a_k]}$ and $\varphi|_{[a_k, b_{k+1}]}$. 
$$f(b_{k+1})=g(b_{k+1})\leq h(b_{k+1}).$$
Evidently, $(f(x), g(x))\in V$, $x\in [a_0, b_{k+1}]$ and $(g^{-1}(x), h^{-1}(x))\in V$, $x\in [a_0, g(b_{k+1})]$.

\medskip

(c) If $h(b_{k+1})<f(b_{k+1})\leq h(a_{k+1})$, then by Claim B
$$\exists\ \varphi\in G\ \mbox{such that}\ \varphi ([a_k, b_{k+1}])=[g(a_k), f(b_{k+1})]\ \mbox{and}\ \varphi (x)\in\st (f(x), v),\ x\in [a_k, b_{k+1}].$$ 
Let $g|_{[a_0, b_{1+1}]}$ be a combination of maps $g|_{[a_0, a_k]}$ and $\varphi|_{[a_k, b_{k+1}]}$. 
$$h(b_{k+1})<g(b_{k+1})=f(b_{k+1}).$$
Evidently, $(f(x), g(x))\in V$, $x\in [a_0, b_{k+1}]$ and $(g^{-1}(x), h^{-1}(x))\in V$, $x\in [a_0, h(b_{k+1})]$.

\medskip 

(d) If $h(a_{k+1})<f(b_{k+1})$, then by Claim A $h(a_{k+1})\in\st (f(b_{k+1}), v)$. Take 
$$y\in X\ \mbox{such that}\  y\in\st (f(b_{k+1}), v),\ h(a_k)\leq y<h(a_{k+1}).$$ 
By Claim B $\exists\ \varphi\in G$ such that $\varphi ([a_k, b_{k+1}])=[g(a_k), y]$ and $\varphi (x)\in\st (f(x), v),\ x\in [a_k, b_{k+1}]$. 
Let $g|_{[a_0, b_{1+1}]}$ be a combination of maps $g|_{[a_0, a_k]}$ and $\varphi|_{[a_k, b_{k+1}]}$. 
$$h(b_{k+1})\leq g(b_{k+1})<f(b_{k+1}),\ y=g(b_{k+1})\in\st (f(b_{k+1}), v).$$
Evidently, $(f(x), g(x))\in V$, $x\in [a_0, b_{k+1}]$ and $(g^{-1}(x), h^{-1}(x))\in V$, $x\in [a_0, h(b_{k+1})]$.

\medskip 

(II) $h(a_k)\leq g(a_k)\leq f(a_k)$,  $g(a_k)\in\st (f(a_k), v)$.

\medskip

(a') If $h(b_{k+1})\leq f(b_{k+1})\leq h(a_{k+1})$, then by Claim B
$$\exists\ \varphi\in G\ \mbox{such that}\ \varphi ([a_k, b_{k+1}])=[g(a_k), f(b_{k+1})]\ \mbox{and}\ \varphi (x)\in\st (f(x), v),\ x\in [a_k, b_{k+1}].$$ 
Let $g|_{[a_0, b_{1+1}]}$ be a combination of maps $g|_{[a_0, a_k]}$ and $\varphi|_{[a_k, b_{k+1}]}$. 
$$h(b_{k+1})<g(b_{k+1})=f(b_{k+1}).$$
Evidently, $(f(x), g(x))\in V$, $x\in [a_0, b_{k+1}]$ and $(g^{-1}(x), h^{-1}(x))\in V$, $x\in [a_0, h(b_{k+1})]$.

\medskip 

(b') If $h(a_{k+1})<f(b_{k+1})$, then by Claim A $h(a_{k+1})\in\st (f(b_{k+1}), v)$. Take 
$$y\in X\ \mbox{such that}\  y\in\st (f(b_{k+1}), v),\ h(a_k)\leq y<h(a_{k+1}).$$ 
By Claim B $\exists\ \varphi\in G$ such that $\varphi ([a_k, b_{k+1}])=[g(a_k), y]$ and $\varphi (x)\in\st (f(x), v),\ x\in [a_k, b_{k+1}]$. 
Let $g|_{[a_0, b_{1+1}]}$ be a combination of maps $g|_{[a_0, a_k]}$ and $\varphi|_{[a_k, b_{k+1}]}$. 
$$h(b_{k+1})\leq g(b_{k+1})<f(b_{k+1}),\ y=g(b_{k+1})\in\st (f(b_{k+1}), v).$$
Evidently, $(f(x), g(x))\in V$, $x\in [a_0, b_{k+1}]$ and $(g^{-1}(x), h^{-1}(x))\in V$, $x\in [a_0, h(b_{k+1})]$.

\medskip

(c') If $f(b_{k+1})\leq h(b_{k+1})$, then by Claim B
$$\exists\ \varphi\in G\ \mbox{such that}\ \varphi ([a_k, b_{k+1}])=[g(a_k), f(b_{k+1})]\ \mbox{and}\ \varphi (x)\in\st (f(x), v),\ x\in [a_k, b_{k+1}].$$ 
Let $g|_{[a_0, b_{1+1}]}$ be a combination of maps $g|_{[a_0, a_k]}$ and $\varphi|_{[a_k, b_{k+1}]}$. 
$$f(b_{k+1})=g(b_{k+1})\leq h(b_{k+1}).$$
Evidently, $(f(x), g(x))\in V$, $x\in [a_0, b_{k+1}]$ and $(g^{-1}(x), h^{-1}(x))\in V$, $x\in [a_0, g(b_{k+1})]$.

\medskip 

As a result of using induction $g|_{[a_0, b_n]}$ such that $(f(x), g(x))\in V$, $x\in [a_0, b_n]$,
$$f(b_n)\leq g(b_n)\leq h(b_n)\ \mbox{and}\ (g^{-1}(x), h^{-1}(x))\in V,\ x\in [a_0, g(b_n)],\ \mbox{or}$$
$$h(b_n)\leq g(b_n)\leq f(b_n)\ \mbox{and}\ (g^{-1}(x), h^{-1}(x))\in V,\ x\in [a_0, h(b_n)]$$
is constructed. 

Construction of $g|_{[b_n, a_n]}$. In both cases by Claim B
$$\exists\ \varphi\in G\ \mbox{such that}\ \varphi ([b_n, a_n])=[g(b_n), a_n]\ \mbox{and}\ \varphi (x)\in\st (f(x), v),\ x\in [b_b, a_n].$$ 
Let $g$ be a combination of maps $g|_{[a_0, b_n]}$ and $\varphi|_{[b_n, a_n]}$. 
$$h(a_n)=g(a_n)=f(a_n).$$
Evidently, $(f(x), g(x))\in V$, $x\in [a_0, a_n]$ and $(g^{-1}(x), h^{-1}(x))\in V$, $x\in [a_0, a_n]$.
\end{proof}

\begin{rem}
{\rm The equality  $b_{c_m X} \aut X=b_r \aut X$  when $X=(0, 1)$ ($c_m X=[0, 1]$) is stated in~\cite{usp2001} (see, also, \cite{GlasnerMegr2008}).

In the  same manner as above one can show that $b_{Y^+} \aut X$, $b_{Y} \aut X$, $b_{Y^-} \aut X$ are not semigroups.

The same results are true for a subgroup $G$ of $\aut X$  which acts ultrtratransitively on $X$ and the actions 
$$\St_u\curvearrowright (X\cap (\gets, u)),\ \St_u\curvearrowright (X\cap (u, \to))\ \mbox{are ultratransitive for any proper gap}\ u\ \mbox{in}\  X.$$}
\end{rem}

\begin{que}
Does there exist $b G\in\mathbb E (G)$ such that $b_{c_m X} G<bG<e_{c_m X} G$?
\end{que}


\begin{thm}\label{WAPtrivial}
Let $X$ be an ultrahomogeneous {\rm LOTS}, $G=(\aut X, \tau_p)$. 
There are no sm-compactifications of $G$. Moreover, the {\rm WAP}-compactification of $G$ is trivial. 
\end{thm}

\begin{proof}
Assume that there is an sm-compactification $b G$ of $G$. Then $b G\leq b_r G<e_{c_m X} G$ by~\cite{KozlovLeiderman2025} and Theorem~\ref{thmOrderLOTS}. Therefore, the congruence $\sim$ on $e_{c_m X} G$ which corresponds to the map of compactifications $e_{c_m X} G\to b G$ (homomorphism of monoids) is stronger than the equivalence relation $\sim_{\Theta}$ from Theorem~\ref{thmOrderLOTS}. Let us show that $\sim$ is trivial (there is the only equivalence class). For this purpose it is enough to show that for any $\sigma=\{x_1,\ldots, x_n\}$, $x_1,\ldots, x_n\in Y$, $x_1<\ldots< x_n$, $\sigma'=\{y_1,\ldots, y_{n-1}\}$, $y_1,\ldots, y_{n-1}\in Y$, $y_1<\ldots<y_{n-1}$, $n\in\mathbb N$, $n\geq 2$, 
the map
$$f_{\sigma}^{\sigma'}(t)=\left\{
\begin{array}{ll}
\inf, &  t\leq x_1\\
y_k, & x_k< t\leq x_{k+1},\ k=1,\ldots, n-1, \\
\sup, & x_n<t, \\
\end{array}
\right.$$
belongs to the equivalence class $[f_{\inf}^-]$ of $f_{\inf}^-$ (all equivalence classes are closed). If this holds, then in any nbd of any homeomorphism $g\in G$ (due to the topology of pointwise convergence on $e_{c_m X} G$) there is an element from  $[f_{\inf}^-]$. Since the equivalent class of $\sim$ is closed, $g\in [f_{\inf}^-]$.

{\it Base of induction.} $\sigma=\{x_1, x_2\}$, $\sigma'=\{y_1\}$, $x_1, x_2, y_1\in Y$.

For $y\in Y$, $z\in c_m X$ put 
$$f^z_y(t)=\left\{
\begin{array}{ll}
\inf, &  t<y,\\
z, & t=y, \\
\sup, & t>y. \\
\end{array}
\right.$$ 
$f^z_{y_1}\sim f^{z'}_{y_1}$ since $f^z_{y_1}\sim_{\Theta} f^{z'}_{y_1}$, $z, z'\in c_m X$. Further, $f^z_{y_1}\circ f^{y_1}_{\sigma}=f^{z}_{\sigma}$,  $z\in c_m X$, and $f^{z}_{\sigma}\sim f^{z'}_{\sigma}$ for all $z, z'\in c_m X$. 

$f^{\inf}_{\sigma}=f^-_{x_2}$,  $f^{\sup}_{\sigma}=f^-_{x_1}$. Therefore, $f^z_{x_1}\sim f^{z'}_{x_2}$, $z, z'\in c_m X$, $x_1, x_2\in Y$. It is easy to check that $f^-_{\inf}$ belongs to the closure of the set $\{f^z_y(t)\ |\ y\in Y,\ z\in c_m X\}$. Hence, $\{f^z_y(t)\ |\ y\in Y,\ z\in c_m X\}\subset [f^-_{\inf}]$ and  $\{f^{\sigma'}_{\sigma}(t)\ |\ \sigma'=\{y_1\},\ y_1\in c_m X,\ \sigma=\{x_1, x_2\},\ x_1, x_2\in Y\}\subset [f^-_{\inf}]$.

\medskip

{\it Step of induction.} Let $f_{\sigma}^{\sigma'}\in [f_{\inf}^-]$ $\forall\ \sigma=\{x_1,\ldots, x_n\}$, $\sigma'=\{y_1, \ldots, y_{n-1}\}$, $|\sigma|\leq n$.  

We shall examine the case $\hat\sigma=\{x_1, a, x_2, \ldots, x_n\}$, $x_1<a<x_2<\ldots<x_n$,  ${\hat\sigma}'=\{y_1, b, y_2, \ldots, y_{n-1}\}$, 
$y_1<b<y_2<\ldots<y_{n-1}$ (all the other cases can be examined analogously). Take $z\in c_m X$,  $z<b$, and the map
$$h_z(t)=\left\{
\begin{array}{ll}
\inf, &  t< y_1,\\
z, &  t=y_1,\\
b, & y_1< t\leq b, \\
y_2, & b< t\leq y_2, \\
y_k, & y_{k-1}< t\leq y_k,\ k=3,\ldots, n-1, \\
\sup, & t>y_{n-1}. \\
\end{array}
\right.$$
$h_z\sim h_{z'}$, $z, z'\in (\gets, b]$, since $h_z\sim_{\Theta} h_{z'}$. Put ${\hat\sigma}''=\{z, b, y_2, \ldots, y_{n-1}\}$. Then $h^z\circ f_{\hat\sigma}^{{\hat\sigma}'}=f_{\hat\sigma}^{{\hat\sigma}''}$,  $z\in (\gets, b]$, and $f_{\hat\sigma}^{{\hat\sigma}''}$ belong to the one  equivalence class $\forall\ z\in  (\gets, b]$. 

If $z=\inf$, then $f_{\hat\sigma}^{{\hat\sigma}''}\in  [f_{\inf}^-]$ by the inductive assumption. If $z=y_1$, then $f_{\hat\sigma}^{{\hat\sigma}''}= f_{\hat\sigma}^{{\hat\sigma}'}$. Hence, $f_{\hat\sigma}^{{\hat\sigma}'}\in  [f_{\inf}^-]$ and the WAP-compactification of $G$ is trivial. 
\end{proof}

\begin{rem}
{\rm 
\begin{itemize}
\item The  WAP-compactification of $\aut ((0, 1), \tau_p)$ is trivial~\cite{Megr2001}. From~\cite[Lemma 6.3]{Megr2001} it follows that  $\aut (\mathbb Q, \tau_p)$ has no sm-compactifications. 
\item The WAP-compactifications of $\aut (\mathbb Q, \tau_p)$ and $\aut ({\bf F}, \tau_p)$, where ${\bf F}$ is a Thompson group~\cite{Canon}, are trivial. It also follows from Corollary~\ref{heredsm}. 
\item If $G$ is a universal group for a class of groups which contains a group without sm-compactifications, then by Corollary~\ref{heredsm} $G$ has no sm-compactifications. 

The group ${\rm Hom}\, {\rm Q}$ of homeomorphisms of the Hilbert cube ${\rm Q}$ (in the c.o-t)~\cite{usp1986} and the group ${\rm Iso}(\mathbb U)$ of isometries of the Urysohn universal space (or of the Urysohn sphere)  (in the t.p.c)~\cite{usp1990} (\cite{usp2008}) are universal for the class of topological groups with a countable base. They contains, for example, the group $\aut (\mathbb Q, \tau_p)$. Hence, they have no sm-compactifications. The WAP-compactification of ${\rm Iso}\,(\mathbb U)$ is trivial~\cite{GlasnerMegr2008}. 
\end{itemize}}
\end{rem}

\begin{que}
Is the {\rm WAP}-compactification of ${\rm Hom}\, {\rm Q}$ trivial?
\end{que}


\section{Remainders of group compactifications}\label{remainders}\label{remainders}

Let $G$ be a non-locally compact topological group. Then the {\it remainder} $b G\setminus G$ of $G$ in any compactification $b G$ of $G$ is 
\begin{itemize}
\item either Lindel\"of or pseudocompact~\cite{Arh2008}, 
\item either $\sigma$-compact or Baire~\cite{Arh2009}.
\end{itemize}

\begin{rem}\label{rem61}
{\rm Let $G$ be a non-locally compact topological group. 
\begin{itemize}
\item[{\rm (a)}] Recall that a space $X$ is  {\it weakly pseudocompact} if $X$ is {\it $G_{\delta}$-dense} in some compactification $b X$ of $X$~\cite{GarMay}. 

The first dichotomy yields that if in some compactification $b G$ of $G$ the remainder $b G\setminus G$ is $G_{\delta}$-dense in $b G$, then the remainder $b G\setminus G$ is pseudocompact, and the remainder $b' G\setminus G$ of $G$ in any compactification $b' G$ of $G$ is pseudocompact. Moreover,  the remainder $b' G\setminus G$ of $G$ in any compactification $b' G$ of $G$ is not $\sigma$-compact, and, hence, has the Baire property.
\item[{\rm (b)}] The second dichotomy yields that the remainder $b G\setminus G$ of $G$ has the Baire property (equivalently, is not $\sigma$-compact) in some compactification $b G$ of $G$ iff $G$ is not \v{C}ech-complete~\cite{Arh2009}.
\end{itemize}}
\end{rem}

\begin{thm}\label{remainderstp}
Let $X$ be a discrete space.
\begin{itemize}
\item[{\rm (1)}] Every remainder of $G={\rm U}\, (\ell^2(X))$ {\rm(}in strong operator topology{\rm)}, $({\rm S}(X), \tau_p)$ or $(\aut X, \tau_{\partial})$ {\rm(}$X$ is ultrahomogeneous chain{\rm)} is Lindel\"of, $\sigma$-compact and $G$ is \v{C}ech-complete iff $X$ is countable. 
\item[{\rm (2)}] Every remainder of $G={\rm U}\, (\ell^2(X))$ {\rm(}in strong operator topology{\rm)}, $({\rm S}(X), \tau_{\partial})$ or $(\aut X, \tau_{\partial})$ {\rm(}$X$ is ultrahomogeneous chain{\rm)} is pseudocompact, Baire and $G$ is not \v{C}ech-complete iff $X$ is uncountable. 
\end{itemize}

Let $X$ be an ultrahomogeneous {\rm LOTS}. 
\begin{itemize}
\item[{\rm (1')}] Every remainder of $(\aut X, \tau_p)$ is Lindel\"of, $\sigma$-compact and  $(\aut X, \tau_p)$ is \v{C}ech-complete iff $X$ is continuously dense and separable. 
\item[{\rm (2')}]  Every remainder of $(\aut X, \tau_p)$ is Lindel\"of,  Baire and  $(\aut X, \tau_p)$ is not \v{C}ech-complete iff $X$ is not continuously dense and separable. 
\item[{\rm (3')}]  Every remainder of $(\aut X, \tau_p)$ is pseudocompact, Baire and $(\aut X, \tau_p)$ is not \v{C}ech-complete iff  $X$ is not separable. 
\end{itemize}
\end{thm}

\begin{proof}
(1) and (2). If $X$ is countable, then $G={\rm U}\, (\ell^2(X))={\rm U}\,(\ell^2)$ and $G={\rm S}\,(X)={\rm S}\,(\mathbb N)$ are Polish groups. G.~Cantor's characterization of $\mathbb Q$ implies that an ultrahomogeneous countable chain $X$ is order isomorphic to $\mathbb Q$ and $G=(\aut\, \mathbb Q, \tau_{\partial})$ is a Polish group. Hence, $G$ is \v{C}ech-complete and every remainder of $G$ is $\sigma$-compact and Lindel\"of.

\medskip

Let $X$ be uncountable. 

(a) The Roelcke compactification of ${\rm U}\, (\ell^2(X))$ is an Ellis compactification $e_{\rm B}\, {\rm U}\, (\ell^2(X))$ where ${\rm B}$ 
is the unit ball in $\ell^2(X)$ in the weak topology by results from \S\ \ref{unitarygroup}. Take any $G_{\delta}$-set $T$ in $e_{\rm B}\, {\rm U}\, (\ell^2(X))\subset {\rm B}^{\rm B}$ which contains element $g\in {\rm U}\, (\ell^2(X))$. $\exists$ a countable subset $Y\subset {\rm B}$ such that $\pr^{-1} (\pr\, T)=T$ where $\pr$ is the restriction to $e_{\rm B}\, {\rm U}\, (\ell^2(X))$ of the projection ${\rm B}^{\rm B}\to {\rm B}^Y$. A {\it closed linear span  $\overline{\rm Span}\, (Y)$ of $Y$} in $\ell^2(X)$ is a proper linear subspace of $\ell^2(X)$. Then 
$$h (x)=\left\{
\begin{array}{ll}
g(x), &  x\in \overline{\rm Span}\, (Y)\cap {\rm B}, \\
0, & x\in \overline{\rm Span}\, (Y)^{\perp}\cap {\rm B} \\
\end{array}
\right.\in T\setminus {\rm U}\, (\ell^2(X)),$$ 
$e_{\rm B}\, {\rm U}\, (\ell^2(X))\setminus {\rm U}\, (\ell^2(X))$ is $G_{\delta}$-dense in $e_{\rm B}\, {\rm U}\, (\ell^2(X))$ and Remark~\ref{rem61} finishes the proof in this case. 

\medskip

(b) $b_r {\rm S}(X)\setminus {\rm S}(X)$ is $G_{\delta}$-dense in $b_r {\rm S}(X)$. Indeed, a $G_{\delta}$-set $T$ in $b_r {\rm S}(X)$  which contains element $g\in {\rm S}(X)$ is, without loss of generality, of the form $\bigcap\limits_{n=1}^{\infty}[x_n, y_n]$, $x_n, y_n\in X$. The map $x_n\to y_n$, $n\in\mathbb N$, with domain $\{x_n\ |\ n\in\mathbb N\}$ is a partial bijection of $X$ and results from \S\ \ref{ultdiscrete} and Remark~\ref{rem61} finishes the proof  in this case. 

\medskip

(c) For $(\aut X, \tau_{\partial})$ apply results from \S\ \ref{ultrchain} (instead of Roelcke compactification  examine the Ellis compactification $e_{\alpha X} \aut X$) and the same resonings as in (b).

\medskip

(1') --- (3'). For ultrahomogeneous LOTS $X$ examine Ellis compactification $e_{c_m X} G$ of $G=(\aut X, \tau_p)$ (see Section~\ref{ultLOT}). A  $G_{\delta}$-set $T$ in $e_{c_m X} G$  which contains element $g\in G$ is, without loss of generality, of the form $\bigcap\limits_{n=1}^{\infty}[x_n, W_n]$, $x_n\in X$, $W_n$ is a $G_{\delta}$-set in $c_m X$. 

(a) If $X$ is not separable, then $\exists$ an open interval $(a, b)\subset X$, $a< b\in X$, such that $(a, b)\subset X\setminus\{x_n\ |\ n\in\mathbb N\}$. Since $g\in G$, $g(a)<g(b)$. Then 
$$h (x)=\left\{
\begin{array}{ll}
g(x), &  x\leq a, \\
g(a), & a<x\leq b, \\
g(x), & b< x \\
\end{array}
\right.\in T\setminus G,$$ 
$e_{c_m X} G\setminus G$ is $G_{\delta}$-dense in $e_{c_m X} G$ and Remark~\ref{rem61} finishes the proof of sufficiency in (3'). 

\medskip

(b) If LOTS $X$ is ultrahomogeneous, continuously dense and separable, then by G.~Cantor's characterization of $\mathbb R$, $X$ is order isomorphic to $\mathbb R$. $(\aut\, \mathbb R, \tau_p)$ is a Polish group and sufficiency in (1') is proved. 

\medskip

(c) If LOTS $X$ is ultrahomogeneous, not continuously dense and separable, then its Dedekind completion is order isomorphic to $\mathbb R$. Hence, $G$ is a proper dense subgroup of $(\aut\, \mathbb R, \tau_p)$ and $G$ is not \v{C}ech-complete. Hence, the remainder $b_{[0, 1]} \aut\, [0, 1]\setminus G$ is  Baire. The Roelcke compactification of $(\aut\, \mathbb R, \tau_p)$ which coincides with the graph compactification $b_{[0, 1]} \aut\, [0, 1]$ is metrizable. Hence, the remainder $b_{[0, 1]} \aut\, [0, 1]\setminus G$ is Lindel\"of. 
\end{proof}

\begin{rem}
{\rm If $X$ is uncountable (respectively, $X$ is an uncountable ultrahomogeneous chain), then the remainder of $b_r {\rm S}(X)\setminus {\rm S}(X)$,  (respectively, $e_{\alpha X}\aut X\setminus \aut X$) contains a dense $\sigma$-compact subset (union of compact ideals)~\cite{KozlovSorin2025}. Therefore, the remainder $b_r {\rm S}(X)\setminus {\rm S}(X)$ (respectively, $e_{\alpha X}\aut X\setminus\aut X$) is {\it weakly  Lindel\"of} (a space $Y$ is weakly  Lindel\"of if $\forall$ open cover $u$ of $Y$ $\exists$ a countable family $u'\subset u$ such that $\cup u'$ is a dense subset of $Y$) and can serve as an example of a pseudocompact weakly  Lindel\"of space which is not compact. Moreover, remainders of all compactifications of ${\rm S}(X)$ (respectively, $\aut X$) which are less or equal than $b_r {\rm S}(X)$ (respectively, $e_{\alpha X}\aut X$) are the same.

The same observation is true for a non-separable Hilbert space ${\bf H}$. One can examine, for example, the set of compact ideals $I_n=\{A\in {\rm U}\,({\bf H})\ |\ ||A||\leq 1-\frac1n\}$, $n\in\mathbb N$, in $b_r {\rm U}\,({\bf H})$.}
\end{rem}


\end{document}